\documentclass[11pt,reqno]{amsart}
\usepackage{mathrsfs}
\usepackage{amsmath}
\usepackage{amssymb, amsmath}
\usepackage{color}
\usepackage{enumerate}

\numberwithin{equation}{section}

\allowdisplaybreaks

\let\lam=\lambda

\let\f=\frac

\let\Om=\Omega

\let\th=T
\let\pa=\partial

\def\dive{\mathop{\rm div}\nolimits}

\newcommand{\beq}{\begin{equation}}
\newcommand{\eeq}{\end{equation}}
\newcommand{\ben}{\begin{eqnarray}}
\newcommand{\een}{\end{eqnarray}}
\newcommand{\beno}{\begin{eqnarray*}}
\newcommand{\eeno}{\end{eqnarray*}}

\newtheorem{theorem}{Theorem}[section]
\newtheorem{definition}[theorem]{Definition}
\newtheorem{lemma}[theorem]{Lemma}
\newtheorem{proposition}[theorem]{Proposition}

\newtheorem{remark}[theorem]{Remark}

\begin{document}

\title[B\'enard convection]{ local well-posedness for the B\'enard convection without surface tension}

\author{Yunrui Zheng}
\address{Beijing International Center for Mathematical Research, Peking University, 100871, P. R. China}
\email{ruixue@mail.ustc.edu.cn}
\begin{abstract}
  We consider the B\'enard convection in a three-dimensional domain bounded below by a fixed flatten boundary and above by a free moving surface. The domain is horizontally periodic. The fluid dynamics are governed by the Boussinesq approximation and the effect of surface tension is neglected on the free surface. Here we develop a local well-posedness theory for the equations of general case in the framework of the nonlinear energy method.
\end{abstract}

\maketitle

\section{Introduction}

\subsection{Formulation of the problem}

In this paper, we consider the B\'enard convection in a shallow horizontal layer of a fluid heated from below evolving in a moving domain
\beno
\Om(t)=\left\{y\in\Sigma\times\mathbb{R}\mid-1<y_3<\eta(y_1,y_2,t)\right\}.
\eeno
Here we assume that $\Sigma=(L_1\mathbb{T})\times(L_2\mathbb{T})$ for $\mathbb{T}=\mathbb{R}/\mathbb{Z}$ the usual $1$-torus and $L_1$, $L_2>0$ the periodicity lengths.
Assuming the Boussinesq approximation \cite{Chand}, we obtain the basic hydrodynamic equations governing B\'enard convection as
\beno
  \pa_t u+u\cdot\nabla u+\frac{1}{\rho_0}\nabla p&=&\nu\Delta u+g\alpha \theta\mathbf{e}_{y_3},\quad \text{in}\ \Om(t),\\
  \pa_t \theta+u\cdot\nabla \theta&=&\kappa\Delta \theta,\quad \text{in}\ \Om(t),\\
  u\mid_{t=0}&=&u_0(y_1,y_2,y_3),\quad \theta\mid_{t=0}=\theta_0(y_1,y_2,y_3),
\eeno
Here, $u=(u_1,u_2,u_3)$ is the velocity field of the fluid satisfying $\dive u=0$, $p$ the pressure, $g>0$ the strength of gravity, $\nu>0$ the kinematic viscosity, $\alpha$ the thermal expansion coefficient, $e_{y_3}=(0,0,1)$ the unit upward vector, $\theta$ the temperature field of the fluid, $\kappa$ the thermal diffusively coefficient, and $\rho_0$ the density at the temperature $T_0$. Notice that, we have made the shift of actual pressure $\bar{p}$ by $p=\bar{p}+gy_3-p_{atm}$ with the constant atmosphere pressure $p_{atm}$.

 The boundary condition is
 \beno
 \pa_t\eta-u^\prime\cdot\nabla\eta+u_3&=&0,\quad \text{on}\ \{y_3=\eta(t,y_1,y_2)\},\\
  (pI-\nu\mathbb{D}(u))n&=&g\eta n+\sigma Hn+(\mathbf{t}\cdot\nabla)\sigma\mathbf{t},\quad \text{on}\ \{y_3=\eta(t,y_1,y_2)\},\\
  n\cdot\nabla \theta+Bi \theta&=&-1,\quad \text{on}\ \{y_3=\eta(t,y_1,y_2)\},\\
  u\mid_{y_3=-1}&=&0,\quad \theta\mid_{y_3=-1}=0,
 \eeno
 Here, $u^\prime=(u_1,u_2)$, $I$ the $3\times3$ identity matrix, $\mathbb{D}(u)_{ij}=\pa_iu_j+\pa_ju_i$ the symmetric gradient of $u$, $\mathscr{N}$ the upward normal vector of the free surface $\{y_3=\eta\}$, $n=\mathscr{N}/|\mathscr{N}|$ the unit upward normal vector of the free surface $\{y_3=\eta\}$ where $\mathscr{N}=(-\pa_1\eta, -\pa_2\eta, 1)$ is the upward normal vector of the free surface $\{y_3=\eta\}$ and $|\mathscr{N}|=\sqrt{(\pa_1\eta)^2+(\pa_2\eta)^2+1}$, $\mathbf{t}$ the unit tangential vector of the free surface, $Bi\ge0$ the Biot number and $H$ the mean curvature of the free surface. For simplicity, we only consider the case without surface tension in this paper, i.e. $\sigma=0$.

 We will always assume the natural condition that there exists a positive number $\delta_0$ such that $1+\eta_0\ge\delta_0>0$ on $\Sigma$, which means that the initial free surface is strictly separated from the bottom. And without loss of generality, we may assume that $\rho_0=\mu=\kappa=\alpha=g=Bi=1$. That is, we will consider the equations
 \ben\label{equ:BC}
 \left\{
 \begin{aligned}
   \pa_tu+u\cdot\nabla u+\nabla p-\Delta u-\theta e_{y_3}&=0& \quad \text{in}\quad \Om(t), \\
   \dive u&=0& \quad \text{in}\quad \Om(t),\\
   \pa_t\theta+u\cdot\nabla \theta-\Delta \theta&=0& \quad \text{in}\quad \Om(t),\\
   (pI-\mathbb{D}u)n&=\eta n& \quad \text{on}\quad \{y_3=\eta(t,y_1,y_2)\},\\
   \nabla \theta\cdot n+\theta &=-1&\quad\text{on}\quad\{y_3=\eta(t,y_1,y_2)\},\\
   u=0,\quad \theta&=0& \quad\text{on}\quad\{y_3=-1\},\\
   u\mid_{t=0}=u_0, \quad \theta\mid_{t=0}&=\theta_0&\quad \text{in}\quad \Om(0),\\
   \pa_t\eta+u_1\pa_1\eta+u_2\pa_2\eta_2&=u_3& \quad\text{on}\quad\{y_3=\eta(t,y_1,y_2)\},\\
   \eta\mid_{t=0}&=\eta_0&\quad\text{on}\quad\{y_3=\eta(t,y_1,y_2)\}.
 \end{aligned}
 \right.
 \een
 The discussion of fourth equation in \eqref{equ:BC} may be found in \cite{WL}. The eighth equation in \eqref{equ:BC} implies that the free surface is advected with the fluid.

 \subsection{Previous results}

 Traditionally, the B\'enard convection problem has been studied in fixed upper boundary and in free boundary surface with surface tension.

  For the problem with surface tension case, the existence and decay of global in time solutions of B\'enard convection problem with free boundary surface was proved by T. Iohara, T. Nishida and Y. Teramoto in $L^2$ spaces. T. Iohara proved this in $2$-D setting. T. Nishida and Y. Teramoto proved this in $3$-D background. They all utilized the framework of \cite{Beale2} in the Lagrangian coordinates.

 \subsection{Geometrical formulation}
 In the absence of surface tension effect, we will solve this problem in Eulerian coordinates. First, we straighten the time dependent domain $\Om(t)$ to a time independent domain $\Om$. The idea was introduced by J. T. Beale in section 5 of \cite{Beale2}. And in \cite{GT1}, \cite{GT2} and \cite{GT3}, Y. Guo and I. Tice proved the local and global existence results for the incompressible Navier--Stokes equations with a deformable surface using this idea.  In \cite{GT1}, \cite{GT2} and \cite{GT3}, Guo and  Tice assume that the surface function $\eta$ in some norms is small, which means $\eta$ is a small perturbation for the plane $\{y_3=0\}$. In order to study the free boundary problem of the incompressible Navier--Stokes equations with a general surface function $\eta$, L. Wu introduced the $\varepsilon$-Poisson integral method in \cite{LW}. In this paper, we will use the flattening transformation method introduced by L. Wu.  We define $\bar{\eta}^\varepsilon$ by
 \[
 \bar{\eta}^\varepsilon=\mathscr{P}^\varepsilon\eta=\thinspace\text{the parametrized harmonic extension of}\thinspace \eta.
 \]
 The definition of $\mathscr{P}^\varepsilon\eta$ can be seen in the section 1.3.1 of \cite{LW} for the periodic case.
  We introduce the mapping $\Phi^\varepsilon$ from $\Om$ to $\Om(t)$ as
 \beq\label{map:phi}
 \Phi^\varepsilon :(x_1,x_2,x_3)\mapsto (x_1,x_2,x_3+(1+x_3)\bar{\eta}^\varepsilon)=(y_1,y_2,y_3),
 \eeq
 and its Jacobian matrix
 \[
 \nabla\Phi^\varepsilon=\left(\begin{array}{ccc}
   1 & 0 & 0\\
   0 & 1 & 0\\
   A^\varepsilon  & B^\varepsilon  & J^\varepsilon
 \end{array}\right)
 \]
 and the transform matrix
 \[
 \mathscr{A}^\varepsilon=((\nabla\Phi^\varepsilon)^{-1})^\top=\left(\begin{array}{ccc}
   1 & 0 & -A^\varepsilon K^\varepsilon \\
   0 & 1 & -B^\varepsilon K^\varepsilon \\
   0 & 0 & K^\varepsilon
 \end{array}\right)
 \]
 where
 \beq\label{equ:components}
   A^\varepsilon=(1+x_3)\pa_1\bar{\eta}^{\varepsilon},\thinspace
   B^\varepsilon=(1+x_3)\pa_2\bar{\eta}^{\varepsilon},\thinspace
   J^\varepsilon=1+\bar{\eta}^\varepsilon+(1+x_3)\pa_3\bar{\eta}^\varepsilon,\thinspace
   K^\varepsilon=1/J^\varepsilon.
 \eeq
According to Theorem 2.7 in \cite{LW} and the assumption that $1+\eta_0>\delta_0>0$, there exists a $\delta>0$ such that $J^\varepsilon(0)>\delta>0$ for a sufficiently small $\varepsilon$ depending on $\|\eta_0\|_{H^{5/2}}$. This implies that $\Phi^\varepsilon(0)$ is a homomorphism. Furthermore, $\Phi^\varepsilon(0)$ is a $C^1$ diffeomorphism deduced from Lemma 2.5 and 2.6 in \cite{LW}. For simplicity, in the following, we just write $\bar{\eta}$ instead of $\bar{\eta}^\varepsilon$, while the same fashion applies to $\mathscr{A}$, $\Phi$, $A$, $B$, $J$ and $K$.
 Then, we define some transformed operators. The differential operators $\nabla_{\mathscr{A}}$, $\dive_{\mathscr{A}}$ and  $\Delta_{\mathscr{A}}$ are defined as follows.
 \begin{align*}
   &(\nabla_{\mathscr{A}}f)_i=\mathscr{A}_{ij}\pa_jf,\\
   &\dive_{\mathscr{A}} u=\mathscr{A}_{ij}\pa_ju_i,\\
   &\Delta_{\mathscr{A}}f=\nabla_{\mathscr{A}}\cdot\nabla_{\mathscr{A}}f.
 \end{align*}
  The symmetric $\mathscr{A}$-gradient $\mathbb{D}_{\mathscr{A}}$ is defined as $(\mathbb{D}_{\mathscr{A}}u)_{ij}=\mathscr{A}_{ik}\pa_ku_j+\mathscr{A}_{jk}\pa_ku_i$. And we write the stress tensor as $S_{\mathscr{A}}(p,u)=pI-\mathbb{D}_{\mathscr{A}}u$, where $I$ is the $3\times3$ identity matrix. Then we note that $\dive_{\mathscr{A}}S_{\mathscr{A}}(p,u)=\nabla_{\mathscr{A}}p-\Delta_{\mathscr{A}}u$ for vector fields satisfying $\dive_{\mathscr{A}}u=0$. We have also written $\mathscr{N}=(-\pa_1\eta,-\pa_2\eta,1)$ for the nonunit normal to $\{y_3=\eta(y_1,y_2,t)\}$.
 Then the original equations \eqref{equ:BC} becomes
 \ben\label{equ:NBC}
 \left\{
 \begin{aligned}
 &\pa_tu-\pa_t\bar{\eta}(1+x_3)K\pa_3u+u\cdot\nabla_{\mathscr{A}}u-\Delta_{\mathscr{A}}u+\nabla_{\mathscr{A}}p-\theta \nabla_{\mathscr{A}}y_3=0 &\quad\text{in}\quad\Om \\
 &\nabla_{\mathscr{A}}\cdot u=0&\quad\text{in}\quad\Om \\
 &\pa_t\theta-\pa_t\bar{\eta}(1+x_3)K\pa_3\theta+u\cdot\nabla_{\mathscr{A}}\theta-\Delta_{\mathscr{A}}\theta=0&\quad\text{in}\quad\Om \\
&(p I-\mathbb{D}_{\mathscr{A}}u)\mathscr{N}=\eta\mathscr{N}&\quad\text{on}\quad \Sigma\\
 &\nabla_{\mathscr{A}}\theta\cdot\mathscr{N}+\theta\left|\mathscr{N}\right|=-\left|\mathscr{N}\right|&\quad\text{on}\quad \Sigma\\
 &u=0,\quad\theta=0& \quad\text{on}\quad \Sigma_b\\
 &u(x,0)=u_0, \quad \theta(x,0)=\theta_0&\quad \text{in}\quad \Om\\
 &\pa_t\eta+u_1\pa_1\eta+u_2\pa_2\eta=u_3&\quad\text{on}\quad \Sigma\\
 &\eta(x^\prime,0)=\eta_0(x^\prime)&\quad\text{on}\quad \Sigma
 \end{aligned}
 \right.
 \een
 where $e_3=(0, 0, 1)$ and we can split the equation \eqref{equ:NBC} into a equation governing B\'enard convection and a transport equation, i.e.
 \ben\label{equ:nonlinear BC}
 \left\{
 \begin{aligned}
 &\pa_tu-\pa_t\bar{\eta}(1+x_3)K\pa_3u+u\cdot\nabla_{\mathscr{A}}u-\Delta_{\mathscr{A}}u+\nabla_{\mathscr{A}}p-\theta \nabla_{\mathscr{A}}y_3=0& \quad\text{in}\quad\Om\\
 &\nabla_{\mathscr{A}}\cdot u=0&\quad\text{in}\quad\Om\\
 &\pa_t\theta-\pa_t\bar{\eta}(1+x_3)K\pa_3\theta+u\cdot\nabla_{\mathscr{A}}\theta-\Delta_{\mathscr{A}}\theta=0&\quad\text{in}\quad\Om\\
& (p I-\mathbb{D}_{\mathscr{A}}u)\mathscr{N}=\eta\mathscr{N}&\quad\text{on}\quad \Sigma\\
 &\nabla_{\mathscr{A}}\theta\cdot\mathscr{N}+\theta\left|\mathscr{N}\right|=-\left|\mathscr{N}\right|&\quad\text{on}\quad \Sigma\\
 &u=0,\quad\theta=0& \quad\text{on}\quad \Sigma_b\\
 &u(x,0)=u_0, \quad \theta(x,0)=\theta_0&\quad \text{in}\quad \Om\\
 \end{aligned}
 \right.
 \een
 and
 \ben
 \left\{
 \begin{aligned}
 &\pa_t\eta+u_1\pa_1\eta+u_2\pa_2\eta=u_3&\quad\text{on}\quad \Sigma\\
 &\eta(x^\prime,0)=\eta_0(x^\prime)&\quad\text{on}\quad \Sigma
 \end{aligned}
 \right.
 \een
 Clearly, all the quantities in these two above systems are related to $\eta$.

 \subsection{Main theorem}

The main result of this paper is the local well-posedness of the B\'enard convection. Before stating our result, we need to mention the issue of compatibility conditions for the initial data $(u_0,\theta_0,\eta_0)$. We will study for the regularity up to $N$ temporal derivatives for $N\ge2$ an integer. This requires us to use $u_0$, $\theta_0$ and $\eta_0$ to construct the initial data $\pa_t^ju(0)$, $\pa_t^j\theta(0)$ and $\pa_t^j\eta(0)$ for $j=1,\ldots,N$ and $\pa_t^jp(0)$ for $j=0,\ldots,N-1$. These data must then satisfy various conditions, which we describe in detail in Section 5.1, so we will not state them here.

Now for stating our result, we need to explain the notation for spaces and norms. When we write $\|\pa_t^ju\|_{H^k}$, $\|\pa_t^j\theta\|_{H^k}$ and $\|\pa_t^jp\|_{H^k}$, we always mean that the space is $H^k(\Om)$, and when we write $\|\pa_t^j\eta\|_{H^s}$, we always mean that the space is $H^s(\Sigma)$, where $H^k(\Om)$ and $H^s(\Sigma)$ are usual Sobolev spaces for $k, s\ge0$.

 \begin{theorem}\label{thm:main}
   Let $N\ge2$ be an integer. Assume that $\eta_0+1\ge\delta>0$, and that the initial data $(u_0,\theta_0,\eta_0)$ satisfies
   \[
   \mathscr{E}_0:=\|u_0\|_{H^{2N}}^2+\|\theta_0\|_{H^{2N}}^2+\|\eta_0\|_{H^{2N+1/2}}^2<\infty,
   \]
   as well as the $N$-th compatibility conditions \eqref{cond:compatibility N}.  Then there exists a $0<T_0<1$ such that for any $0<T<T_0$, there exists a solution $(u,p,\theta,\eta)$ to \eqref{equ:NBC} on the interval $[0,T]$ that achieves the initial data. The solution obeys the estimate
   \ben \label{est:main est}
   \begin{aligned}
     &\sum_{j=0}^N\left(\sup_{0\le t\le T}\|\pa_t^ju\|_{H^{2N-2j}}^2+\|\pa_t^ju\|_{L^2H^{2N-2j+1}}^2\right)+\|\pa_t^{N+1}u\|_{(\mathscr{X}_T)^\ast}\\
     &\quad+\sum_{j=0}^{N-1}\left(\sup_{0\le t\le T}\|\pa_t^jp\|_{H^{2N-2j-1}}^2+\|\pa_t^jp\|_{L^2H^{2N-2j}}^2\right)\\
     &\quad+\sum_{j=0}^N\left(\sup_{0\le t\le T}\|\pa_t^j\theta\|_{H^{2N-2j}}^2+\|\pa_t^j\theta\|_{L^2H^{2N-2j+1}}^2\right)+\|\pa_t^{N+1}\theta\|_{(\mathscr{H}^1_T)^\ast}\\
     &\quad+\Bigg(\sup_{0\le t\le T}\|\eta\|_{H^{2N+1/2}(\Sigma)}^2+\sum_{j=1}^N\sup_{0\le t\le T}\|\pa_t^j\eta\|_{H^{2N-2j+3/2}}^2\\
     &\quad+\sum_{j=2}^{N+1}\|\pa_t^j\eta\|_{L^2H^{2N-2j+5/2}}^2\Bigg)\\
     &\le C(\Om_0,\delta)P(\mathscr{E}_0),
   \end{aligned}
   \een
   where $C(\Om_0,\delta)>0$ depends on the initial domain $\Om_0$ and $\delta$, $P(\cdot)$ is a polynomial satisfying $P(0)=0$, and the temporal norm $L^2$ is computed on $[0,T]$. The solution is unique among functions that achieve the initial data and for which the left-hand side of \eqref{est:main est} is finite. Moreover, $\eta$ is such that the mapping $\Phi(\cdot,t)$ defined by \eqref{map:phi} is a $C^{2N-1}$ diffeomorphism for each $t\in [0,T]$.
 \end{theorem}

 \begin{remark}
   The space $\mathscr{X}_T$ is defined in section $2$ of \cite{GT1}.
 \end{remark}
 \begin{remark}
   Since the mapping $\Phi(\cdot,t)$ is a $C^{2N-1}$ diffeomorphism, we may change coordinates to produce solutions to \eqref{equ:BC}.
 \end{remark}

 \subsection{Notation and terminology}

Now, we mention some definitions, notation and conventions that we will use throughout this paper.

\begin{enumerate}[1.]
  \item Constants. The constant $C>0$ will denote a universal constant that only depend on the parameters of the problem, $N$ and $\Om$, but does not depend on the data, etc. They are allowed to change from line to line. We will write $C=C(z)$ to indicate that the constant $C$ depends on $z$. And we will write $a\lesssim b$ to mean that $a\le C b$ for a universal constant $C>0$.\\
  \item Polynomials. We will write $P(\cdot)$ to denote polynomials in one variable and they may change from one inequality or equality to another.\\
  \item Norms. We will write $H^k$ for $H^k(\Om)$ for $k\ge0$, and $H^s(\Sigma)$ with $s\in\mathbb{R}$ for usual Sobolev spaces. Typically, we will write $H^0=L^2$, With the exception to this is we will use $L^2([0,T];H^k)$ (or $L^2([0,T];H^s(\Sigma))$) to denote the space of temporal square--integrable functions with values in $H^k$ (or $H^s(\Sigma)$).

      Sometimes we will write $\|\cdot\|_k$ instead of $\|\cdot\|_{H^k(\Om)}$ or $\|\cdot\|_{H^k(\Sigma)}$. We assume that functions have natural spaces. For example, the functions $u$, $p$, $\theta$ and $\bar{\eta}$ live on $\Om$, while $\eta$ lives on $\Sigma$. So we may write $\|\cdot\|_{H^k}$ for the norms of $u$, $p$, $\theta$ and $\bar{\eta}$ in $\Om$, and $\|\cdot\|_{H^s}$ for norms of $\eta$ on $\Sigma$.
\end{enumerate}

 \subsection{Plan of the paper}

In section 2, we develop the machinery of time--dependent function spaces based on \cite{GT1}. In section 3, we make some elliptic estimates for the linear steady equations of \eqref{equ:linear BC}.
 In section 4, we will study the local existence theory of the following linear problem for $(u,p,\theta)$, where we think of $\eta$ (and hence $\mathscr{A}$, $\mathscr{N}$, etc.) is given:
 \ben \label{equ:linear BC}
 \left\{
 \begin{aligned}
 &\pa_t u-\Delta_{\mathscr{A}}u+\nabla_{\mathscr{A}}p-\theta \nabla_{\mathscr{A}}y_3=F^1& \quad\text{in}\quad\Om,\\
 &\nabla_{\mathscr{A}}\cdot u=0&\quad\text{in}\quad\Om,\\
 &\pa_t\theta-\Delta_{\mathscr{A}}\theta=F^3&\quad\text{in}\quad\Om,\\
& (p I-\mathbb{D}_{\mathscr{A}}u)\mathscr{N}=F^4&\quad\text{on}\quad \Sigma,\\
 &\nabla_{\mathscr{A}}\theta\cdot\mathscr{N}+\theta\left|\mathscr{N}\right|=F^5&\quad\text{on}\quad \Sigma,\\
 &u=0,\quad\theta=0& \quad\text{on}\quad \Sigma_b,
 \end{aligned}
 \right.
 \een
 subject to the initial condition $u(0)=u_0$ and $\theta(0)=\theta_0$, with the time-dependent Galerkin method. In section 5, we construct the initial data and do some estimates for the forcing terms. In section 6, we construct solutions to \eqref{equ:NBC} using iteration and contraction, and complete the proof of Theorem \ref{thm:main}.

 \section{Functional setting}

 \subsection{Function spaces}

Throughout this paper, we utilize the functional spaces defined by Guo and Tice in section 2 of \cite{GT1}. The only modification is the definition of space $\mathscr{H}^1(t)$. For the vector-valued space $\mathscr{H}^1(t)$, its definition is the same as \cite{GT1}. The following is the definition for the scalar-valued space $\mathscr{H}^1(t)$.
\[
\mathscr{H}^1(t):=\left\{\theta| \|\theta\|_{\mathscr{H}^1}<\infty, \theta|_{\Sigma_b}=0\right\}
\]
with the norm $\|\theta\|_{\mathscr{H}^1}:=\left(\theta,\theta\right)_{\mathscr{H}^1}^{1/2}$, where the inner product $\left(\cdot,\cdot\right)_{\mathscr{H}^1}$ is defined as
\[
\left(\theta,\phi\right)_{\mathscr{H}^1}:=\int_{\Om}\left(\nabla_{\mathscr{A}(t)}\theta\cdot\nabla_{\mathscr{A}(t)}\phi\right)J(t).
\]
The following lemma implies that this space $\mathscr{H}^1$ is equivalent to the usual Sobolev space $H^1$.

\begin{lemma}\label{lem:theta H0 H1}
  Suppose that $0<\varepsilon_0<1$ and $\|\eta-\eta_0\|_{H^{5/2}(\Sigma)}<\varepsilon_0$. Then it holds that
  \beq\label{est:theta H0}
  \|\theta\|_{H^0}^2\lesssim\int_\Om J|\theta|^2\lesssim\left(1+\|\eta_0\|_{H^{5/2}(\Sigma)}\right)\|\theta\|_{H^0}^2,
  \eeq
  \beq\label{est:theta H1}
  \f{1}{\left(1+\|\eta_0\|_{H^{5/2}(\Sigma)}\right)^3}\|\theta\|_{H^1(\Om)}^2\lesssim\int_\Om J|\nabla_{\mathscr{A}}\theta|^2\lesssim\left(1+\|\eta_0\|_{H^{5/2}(\Sigma)}\right)^3\|\theta\|_{H^1(\Om)}^2.
  \eeq
\end{lemma}
\begin{proof}
  From the Poinc\'are inequality, we know that $\|\theta\|_{H^1}$ is equivalent to $\|\nabla\theta\|_{H^0}$. So in the following, we will use $\|\theta\|_{H^1}$ instead of $\|\nabla\theta\|_{H^0}$.

  From the assumption and the Sobolev inequalities, we may derive that
  \[
  \delta\lesssim\|J\|_{L^\infty}\lesssim1+\|\bar{\eta}\|_{L^\infty}+\|\nabla\bar{\eta}\|_{L^\infty}\lesssim 1+\|\eta\|_{H^{5/2}}\lesssim 1+\|\eta_0\|_{H^{5/2}},
  \]
  and
  \begin{align*}
  \|\mathscr{A}\|_{L^\infty}&\lesssim\max\{1, \|AK\|_{L^\infty}^2, \|BK\|_{L^\infty}^2, \|K\|_{L^\infty}^2\}\\
  &\lesssim1+(1+\|\nabla\bar{\eta}\|_{L^\infty}^2)\|K\|_{L^\infty}^2\lesssim \left(1+\|\eta_0\|_{H^{5/2}}\right)^2.
  \end{align*}
  Thus \eqref{est:theta H0} is clearly derived from the estimate of $\|J\|_{L^\infty}$ and we have that
  \begin{align*}
    \int_\Om J|\nabla_{\mathscr{A}}\theta|^2&\lesssim(1+\|\eta_0\|_{H^{5/2}})\int_\Om |\nabla_{\mathscr{A}}\theta|^2\\
    &\lesssim(1+\|\eta_0\|_{H^{5/2}})\max\{1, \|AK\|_{L^\infty}^2, \|BK\|_{L^\infty}^2, \|K\|_{L^\infty}^2\}\|\theta\|_{H^1}^2\\
    &\lesssim\left(1+\|\eta_0\|_{H^{5/2}}\right)^3\|\theta\|_{H^1}^2.
  \end{align*}
  Now we have proved the second inequality of \eqref{est:theta H1}.

  To prove the first inequality of \eqref{est:theta H1}, we rewrite the $\|\theta\|_{\mathscr{H}^1}$ as
  \[
  \int_\Om J|\nabla_{\mathscr{A}}\theta|^2=\int_\Om J|\nabla_{\mathscr{A}_0}\theta|^2+\int_\Om J(\nabla_{\mathscr{A}}\theta+\nabla_{\mathscr{A}_0}\theta)\cdot(\nabla_{\mathscr{A}}\theta-\nabla_{\mathscr{A}_0}\theta),
  \]
   Here $\mathscr{A}_0$ is in terms of $\eta_0$. By the estimates of $\|J\|_{L^\infty}$, we derive that
  \begin{align*}
    \int_\Om J|\nabla_{\mathscr{A}_0}\theta|^2&\gtrsim \f{1}{1+\|\eta_0\|_{H^{5/2}}}\int_\Om J_0|\nabla_{\mathscr{A}_0}\theta|^2\\
    &=\f{1}{1+\|\eta_0\|_{H^{5/2}}}\int_{\Om_0} |\nabla(\theta\circ\Phi(0))|^2\\
    &\gtrsim \f{1}{(1+\|\eta_0\|_{H^{5/2}})^3}\|\theta\|_{H^1},
  \end{align*}
  where in the last inequality, we have used the following Lemma \ref{lem:transport}, since $\Phi(0)$ is a diffeomorphism. Here $J_0$ is in terms of $\eta_0$.
   Then, using the estimates of $\|\mathscr{A}\|_{L^\infty}$ and $\|J\|_{L^\infty}$, we have that
   \begin{align*}
     \left|\int_\Om J(\nabla_{\mathscr{A}}\theta+\nabla_{\mathscr{A}_0}\theta)\cdot(\nabla_{\mathscr{A}}\theta-\nabla_{\mathscr{A}_0}\theta)\right|&\lesssim \|J\|_{L^\infty}\|\mathscr{A}+\mathscr{A}_0\|_{L^\infty}\|\mathscr{A}-\mathscr{A}_0\|_{L^\infty}\|\theta\|_{H^1}\\
     &\lesssim\varepsilon_0\left(1+\|\eta_0\|_{H^{5/2}}\right)^3\|\theta\|_{H^1}.
   \end{align*}
   Then taking $\varepsilon_0$ sufficiently small, we may derive that
   \begin{align*}
     \int_\Om J|\nabla_{\mathscr{A}}\theta|^2&\gtrsim\int_\Om J|\nabla_{\mathscr{A}_0}\theta|^2-\left|\int_\Om J(\nabla_{\mathscr{A}}\theta+\nabla_{\mathscr{A}_0}\theta)\cdot(\nabla_{\mathscr{A}}\theta-\nabla_{\mathscr{A}_0}\theta)\right|\\
     &\gtrsim \f{1}{(1+\|\eta_0\|_{H^{5/2}})^3}\|\theta\|_{H^1}.
   \end{align*}
   This is the first inequality of \eqref{est:theta H1}.
\end{proof}

 We define an operator $\mathcal{K}_t$ by $\mathcal{K}_t\theta=K(t)\theta$, where $K(t):=K$ is defined as \eqref{equ:components}. Clearly, $\mathcal{K}_t$ is invertible and $\mathcal{K}_t^{-1}\Theta=K(t)^{-1}\Theta=J(t)\Theta$, and $J(t):=J=1/K$.
 \begin{proposition}\label{prop:k}
  For each $t\in[0, T]$, $\mathcal{K}_t$ is a bounded linear isomorphism: from $H^k(\Om)$ to $H^k(\Om)$ for $k=0, 1, 2$; from $L^2(\Om)$ to $\mathscr{H}^0(t)$; and from ${}_0H^1(\Om)$ to $\mathscr{H}^1(t)$. In each case, the norms of the operators $\mathcal{K}_t$, $\mathcal{K}_t^{-1}$ are bounded by a polynomial $P(\|\eta(t)\|_{H^{\f72}})$.
  The mapping $\mathcal{K}$ defined by $\mathcal{K}\theta(t):=\mathcal{K}_t\theta(t)$ is a bounded linear isomorphism: from $L^2([0, T]; H^k(\Om))$ to $L^2([0, T]; H^k(\Om))$ for $k=0, 1, 2$; from $L^2([0, T]; H^0(\Om))$ to $\mathscr{H}^0_T$ and from ${_0}{H}{^1}(\Om)$ to $\mathscr{H}^1_{T}$. In each case, the operators $\mathcal{K}$ and $\mathcal{K}^{-1}$ are bounded by the polynomial $P(\sup_{0\le t\le T}\|\eta(t)\|_{H^\f72})$.
 \end{proposition}
 \begin{proof}
It is easy to see that for each $t\in[0,T]$,
\beq
\|\mathcal{K}_t\theta\|_{H^0}\lesssim \|\mathcal{K}_t\|_{C^0}\|\theta\|_{H^0}\lesssim P(\|\eta(t)\|_{H^{\f72}})\|\theta\|_{H^0},
\eeq
\beq
\|\mathcal{K}_t\theta\|_{H^1}\lesssim \|\mathcal{K}_t\|_{C^1}\|\theta\|_{H^1}\lesssim P(\|\eta(t)\|_{H^{\f72}})\|\theta\|_{H^1},
\eeq
\beq
\|\mathcal{K}_t\theta\|_{H^2}\lesssim \|\mathcal{K}_t\|_{C^1}\|\theta\|_{H^2}+\|\mathcal{K}_t\|_{H^2}\|\theta\|_{C^0}\lesssim P(\|\eta(t)\|_{H^{\f72}})\|\theta\|_{H^2}.
\eeq
These inequalities imply that $\mathcal{K}_t$ is a bounded operator from $H^k$ to $H^k$, for $k=0,1,2$. Since $\mathcal{K}_t$ is invertible, we may have the estimate $\|\mathcal{K}_t^{-1}\Theta\|_{H^k}\lesssim P(\|\eta(t)\|_{H^{\f72}})\|\Theta\|_{H^k}$. Thus, $\mathcal{K}_t$ is an isomorphism of $H^k$ to $H^k$, for $k=0,1,2$. With this fact in hand, Lemma \ref{lem:theta H0 H1} implies that $\mathcal{K}_t$ is an isomorphism of $L^2(\Om)$ to $\mathscr{H}^0(t)$ and of ${}_0H^1(\Om)$ to $\mathscr{H}^1(t)$.

The mapping properties of the operator $\mathcal{K}$ on space-time functions may be established in a similar manner.
 \end{proof}

\subsection{Pressure as a Lagrange multiplier}

The introduction of pressure function has been studied by Guo and Tice in section 2 of \cite{GT1}, of which the modification was given by L. Wu in section 2.2 of \cite{LW}. So we omit the details here.

 \section{Elliptic estimates}

 \subsection{Preliminary}

 Before studying the linear problem \eqref{equ:linear BC}, we need some elliptic estimates. In order to study the elliptic problem, we may transform the equations on the domain $\Om$ into constant coefficient equations on the domain $\Om^\prime=\Phi(\Om)$, where $\Phi$ is defined by \eqref{map:phi}. The following lemma shows that the mapping $\Phi$ is an isomorphism between $H^k(\Om^\prime)$ and $H^k(\Om)$. Here, the Sobolev spaces are either vector-valued or scalar-valued.
 \begin{lemma}\label{lem:transport}
   Let $\Psi: \Om\to\Om^\prime$ be a $C^1$ diffeomorphism satisfying $\Psi\in H^{k+1}_{loc}$, $\nabla\Psi-I\in H^k(\Om)$ and the Jacobi $J=\det(\nabla\Psi)>\delta>0$ almost everywhere in $\Om$ for an integer $k\ge3$. If $v\in H^m(\Om^\prime)$, then $v\circ\Psi\in H^m(\Om)$ for $m=0,1, \ldots, k+1$, and
   \[
   \|v\circ\Psi\|_{H^m(\Om)}\lesssim C\left(\|\nabla\Psi-I\|_{H^k(\Om)}\right)\|v\|_{H^m(\Om^\prime)},
   \]
   where $C(\|\nabla\Psi-I\|_{H^k(\Om)})$ is a constant depending on $\|\nabla\Psi-I\|_{H^k(\Om)}$. Similarly, for $u\in H^m(\Om)$, we have $u\circ\Psi^{-1}\in H^m(\Om^\prime)$ for $m=0,1, \ldots, k+1$, and
   \[
   \|u\circ\Psi^{-1}\|_{H^m(\Om^\prime)}\lesssim C\left(\|\nabla\Psi-I\|_{H^k(\Om)}\right)\|u\|_{H^m(\Om)}.
   \]
   Let $\Sigma^\prime=\Psi(\Sigma)$ be the top boundary of $\Om^\prime$. If $v\in H^{m-\f12}(\Sigma^\prime)$ for $m=1, \ldots, k-1$, then $v\circ\Psi\in H^{m-\f12}(\Sigma)$, and
   \[
   \|v\circ\Psi\|_{H^{m-\f12}(\Sigma)}\lesssim C\left(\|\nabla\Psi-I\|_{H^k(\Om)}\right)\|v\|_{H^{m-\f12}(\Sigma^\prime)}.
   \]
   If $u\in H^{m-\f12}(\Sigma)$ for $m=1, \ldots, k-1$, then $u\circ\Psi^{-1}\in H^{m-\f12}(\Sigma^\prime)$ and
   \[
   \|u\circ\Psi^{-1}\|_{H^{m-\f12}(\Sigma^\prime)}\lesssim C\left(\|\nabla\Psi-I\|_{H^k(\Om)}\right)\|u\|_{H^{m-\f12}(\Sigma)}.
   \]
 \end{lemma}
 \begin{proof}
   The proof of this lemma is the same as Lemma $3.1$ in \cite{GT1}, which has been proved by Y. Guo and I. Tice, so we omit the details here.
 \end{proof}
 \subsection{The $\mathscr{A}$-stationary convection problem}

 In this section, we consider the stationary equations
\ben\label{equ:SBC}
\left\{
\begin{aligned}
  \dive_{\mathscr{A}}S_{\mathscr{A}}(p, u)-\theta \nabla_{\mathscr{A}}y_3&=F^1 \quad\text{in}\quad\Om\\
  \dive_{\mathscr{A}}u&=F^2\quad\text{in}\quad\Om\\
 -\Delta_{\mathscr{A}}\theta&=F^3\quad\text{in}\quad\Om\\
 S_{\mathscr{A}}(p, u)\mathscr{N}&=F^4\quad\text{on}\quad \Sigma\\
 \nabla_{\mathscr{A}}\theta\cdot\mathscr{N}+\theta\left|\mathscr{N}\right|&=F^5\quad\text{on}\quad \Sigma\\
 u=0,\quad\theta&=0\quad\text{on}\quad \Sigma_b\\
\end{aligned}
\right.
\een

Before discussing the regularity for strong solution to \eqref{equ:SBC}, we need to define the weak solution of equation \eqref{equ:SBC}. Suppose $F^1\in (\mathscr{H}^1)^\ast$, $F^2\in H^0$, $F^3\in (\mathscr{H}^1)^\ast$ $F^4\in H^{-\f12}(\Sigma)$ and $F^5\in H^{-\f12}(\Sigma)$, $(u,p,\theta)$ is called a weak solution of equation \eqref{equ:SBC} if it satisfies $\nabla_{\mathscr{A}}\cdot u=F^2$,
\beq\label{equ:weak theta}
\left(\nabla_{\mathscr{A}}\theta,\nabla_{\mathscr{A}}\phi\right)_{\mathscr{H}^0}+\left(\theta\left|\mathscr{N}\right|,\phi\right)_{H^0(\Sigma)}=\left<F^3, \phi\right>_{(\mathscr{H}^1)^\ast}+\left<F^5,\phi\right>_{H^{-\f12}(\Sigma)},
\eeq
and
\beq\label{equ:weak u}
\f12\left(\mathbb{D}_\mathscr{A}u,\mathbb{D}_\mathscr{A}\psi\right)_{\mathscr{H}^0}+\left(p, \nabla_{\mathscr{A}}\psi\right)_{\mathscr{H}^0}-\left(\theta \nabla_{\mathscr{A}}y_3, \psi\right)_{\mathscr{H}^0}=\left<F^1,\psi\right>_{(\mathscr{H}^1)\ast}-\left<F^4, \psi\right>_{H^{-\f12}(\Sigma)},
\eeq
for any $\phi, \psi\in \mathscr{H}^1$.
\begin{lemma}
  Suppose $F^1\in (\mathscr{H}^1)^\ast$, $F^2\in \mathscr{H}^0$, $F^3\in (\mathscr{H}^1)^\ast$, $F^4\in H^{-\f12}(\Sigma)$ and $F^5\in H^{-\f12}(\Sigma)$. Then there exists a unique weak solution $(u, p, \theta) \in \mathscr{H}^1\times \mathscr{H}^0 \times \mathscr{H}^1$ to \eqref{equ:SBC}.
\end{lemma}
\begin{proof}
  For the Hilbert space $\mathscr{H}^1$ with the inner product $\left(\theta,\phi\right)=\left(\nabla_{\mathscr{A}}\theta, \nabla_{\mathscr{A}}\phi\right)_{\mathscr{H}^0}+\left(\theta\left|\mathscr{N}\right|,\phi\right)_{H^0(\Sigma)}$, we can define a linear functional $\ell\in (\mathscr{H}^1)^\ast$ by
  \[
  \ell(\phi)=\left<F^3, \phi\right>_{(\mathscr{H}^1)^\ast}+\left<H^5,\phi\right>_{H^{-\f12}(\Sigma)},
  \]
  for all $\phi\in \mathscr{H}^1$. Then by using the Riesz representation theorem, there exists a unique $\theta\in \mathscr{H}^1$ such that
  \[
  \left(\nabla_{\mathscr{A}}\theta, \nabla_{\mathscr{A}}\phi\right)_{\mathscr{H}^0}+\left(\theta\left|\mathscr{N}\right|,\phi\right)_{H^0(\Sigma)}=\left<F^3, \phi\right>_{(\mathscr{H}^1)^\ast}+\left<H^5,\phi\right>_{H^{-\f12}(\Sigma)},
  \]
  for all $\phi\in \mathscr{H}^1$.

  By Lemma 2.6 in \cite{GT1}, there exists a $\bar{u}\in \mathscr{H}^1$ such that $\dive_{\mathscr{A}}\bar{u}=F^2$. Then, we may restrict our test function to $\psi\in \mathscr{X}$. A straight application of Riesz representation theorem to the Hilbert space $\mathscr{X}$ with inner product defined as $\left(u,\psi\right)=\left(\mathbb{D}_{\mathscr{A}}u, \mathbb{D}_{\mathscr{A}}\psi\right)_{\mathscr{H}^0}$ provides a unique $w\in \mathscr{X}$ such that
  \beq \label{eq:velocity}
  \f12\left(\mathbb{D}_{\mathscr{A}}w, \mathbb{D}_{\mathscr{A}}\psi\right)_{\mathscr{H}^0}=-\f12\left(\mathbb{D}_{\mathscr{A}}\bar{u}, \mathbb{D}_{\mathscr{A}}\psi\right)_{\mathscr{H}^0}+\left(\theta \nabla_{\mathscr{A}}y_3, \psi\right)_{\mathscr{H}^0}+\left<F^1, \psi\right>_{(\mathscr{H}^1)^\ast}-\left<F^4, \psi\right>_{H^{-\f12}(\Sigma)}
  \eeq
  for all $\psi\in \mathscr{X}$. Then we can find $u$ satisfying
  \beq\label{equ:pressureless weak u}
\f12\left(\mathbb{D}_\mathscr{A}u,\mathbb{D}_\mathscr{A}\psi\right)_{\mathscr{H}^0}-\left(\theta \nabla_{\mathscr{A}}y_3, \psi\right)_{\mathscr{H}^0}=\left<F^1,\psi\right>_{(\mathscr{H}^1)\ast}-\left<F^4, \psi\right>_{H^{-\f12}(\Sigma)},
\eeq
 by $u=w+\bar{u}\in\mathscr{H}^1$, with $\dive_{\mathscr{A}}u=F^2$.

It is easily to be seen that $u$ is unique. Suppose that there exists another $\tilde{u}$ still satisfies \eqref{equ:pressureless weak u}. Then we have $\dive_{\mathscr{A}}(u-\tilde{u})=0$, and $\left(\mathbb{D}_\mathscr{A}(u-\tilde{u}),\mathbb{D}_\mathscr{A}\psi\right)_{\mathscr{H}^0}=0$ for any $\psi\in \mathscr{X}$. By taking $\psi=u-\tilde{u}$, and using the Korn's inequality, we know that $\|u-\tilde{u}\|_{H^0}=0$ which implies $u=\tilde{u}$.

  In order to introduce the pressure $p$, we can define $\lam\in (\mathscr{H}^1)^\ast$ as the difference of the left and right hand sides of \eqref{eq:velocity}. Then $\lam(\psi)=0$ for all $\psi\in \mathscr{X}$. According to the Proposition $2.12$ in \cite{LW}, there exists a unique $p\in \mathscr{H}^0$ satisfying $\left(p, \dive_{\mathscr{A}}\psi\right)_{\mathscr{H}^0}=\lam(\psi)$ for all $\psi\in \mathscr{H}^1$.
\end{proof}
In the next result, we establish the strong solutions of \eqref{equ:SBC} and present some elliptic estimates.
\begin{lemma}\label{lem:S lower regularity}
  Suppose that $\eta\in H^{k+\f12}(\Sigma)$ for $k\ge3$ such that the mapping $\Phi$ defined in \eqref{map:phi} is a $C^1$ diffeomorphism of $\Om$ to $\Om^\prime=\Phi(\Om)$. If $F^1\in H^0$, $F^2\in H^1$, $F^3\in H^0$, $F^4\in H^{\f12}$ and $F^5\in H^{\f12}$, then the problem \eqref{equ:SBC} admits a unique strong solution $(u, p, \theta)\in H^2(\Om)\times H^1(\Om)\times H^2(\Om)$, i.e. $(u, p, \theta)$ satisfy \eqref{equ:SBC} a.e. in $\Om$ and on $\Sigma$, $\Sigma_b$. Moreover, for $r=2, \ldots, k-1$, we have the estimate
  \begin{equation}\label{ineq:lower elliptic}
  \begin{aligned}
  \|u\|_{H^r}+\|p\|_{H^{r-1}}+\|\theta\|_{H^r}
  &\lesssim& C(\eta)\Big(\|F^1\|_{H^{r-2}}+\|F^2\|_{H^{r-1}}+\|F^3\|_{H^{r-2}}\\
  &&+\|F^4\|_{H^{r-\f32}(\Sigma)}+\|F^5\|_{H^{r-\f32}(\Sigma)}\Big),
  \end{aligned}
  \end{equation}
  whenever the right-hand side is finite, where $C(\eta)$ is a constant depending on $\|\eta\|_{H^{k+\f12}(\Sigma)}$.
\end{lemma}
\begin{proof}
  First, we consider the problem
  \beno
  \left\{
  \begin{aligned}
    -\Delta_{\mathscr{A}}\theta&=F^3 \quad \text{in}\quad \Om,\\
    \nabla_{\mathscr{A}}\theta\cdot\mathscr{N}+\theta\left|\mathscr{N}\right|&=F^5 \quad \text{on}\quad \Sigma,\\
    \theta&=0 \quad \text{on}\quad \Sigma_b.
  \end{aligned}
  \right.
  \eeno
  Since the coefficients of this equation are not constants, We transform this problem to one on $\Om^\prime=\Phi(\Om)$ by introducing the unknowns $\Theta$ according to $\theta=\Theta\circ\Phi$. Then $\Theta$ should be solutions to the usual problem on $\Om^\prime=\{-1\le y_3\le \eta(y_1, y_2)\}$ with upper boundary $\Sigma^\prime=\{y_3=\eta\}$:
  \ben\label{equ:theta}
  \left\{
  \begin{aligned}
    -\Delta \Theta&=F^3\circ\Phi^{-1}=G^3 &\quad\text{in}\quad\Om^\prime,\\
    \nabla \Theta\cdot\mathscr{N}&+\Theta\left|\mathscr{N}\right|= F^5\circ\Phi^{-1}=G^5&\quad\text{on}\quad\Sigma^\prime,\\
    \Theta&=0 &\quad\text{on}\quad\Sigma^\prime_b.
  \end{aligned}
  \right.
  \een
  Note that, according to Lemma \ref{lem:transport}, $G^3\in H^0(\Om^\prime)$ and $G^5\in H^{1/2}(\Sigma^\prime)$.
  Then we may argue as the Lemma 2.8 in \cite{Beale1} and use the Theorem 10.5 in \cite{ADN}, to obtain that there exists a unique $\Theta\in H^2(\Om^\prime)$, solving problem \eqref{equ:theta} with
  \[
  \|\Theta\|_{H^2(\Om^\prime)}\lesssim C(\eta)(\|G^3\|_{H^0(\Om^\prime)}+\|G^5\|_{H^{\f12}(\Sigma^\prime)}),
  \]
  for $C(\eta)$ a constant depending on $\|\eta\|_{H^{k+\f12}}$.

  For the $\mathscr{A}$-Stokes equations, we introduce the unknowns $v, q$ by $u=v\circ\Phi$ and $q=p\circ\Phi$. For the usual Stokes problem
  \ben\label{equ:Stokes}
\left\{
\begin{aligned}
  S(q, v)-\Theta e_3&=F^1\circ\Phi^{-1}=G^1& \quad\text{in}\quad\Om^\prime\\
 \nabla\cdot v&=F^2\circ\Phi^{-1}=G^2&\quad\text{in}\quad\Om^\prime\\
 S(q, v)\mathscr{N}&=F^4\circ\Phi^{-1}=G^4&\quad\text{on}\quad \Sigma^\prime\\
 v&=0&\quad\text{on}\quad \Sigma_b,
\end{aligned}
\right.
\een
we use the same argument as in the proof of Lemma 3.6 in \cite{GT1} with $G^1+\Theta e_3$ instead of $G^1$. Then we have that there exist unique $v\in H^2(\Om^\prime)$, $q\in H^1(\Om^\prime)$, solving problem \eqref{equ:Stokes} with
\[
\|v\|_{H^2(\Om^\prime)}+\|q\|_{H^1(\Om^\prime)}\lesssim C(\eta)\left(\|G^1\|_{H^0(\Om^\prime)}+\|G^2\|_{H^1(\Om^\prime)}+\|G^4\|_{H^{\f12}(\Sigma^\prime)}+\|\Theta\|_{H^0(\Om^\prime)}\right),
\]
for $C(\eta)$ a constant depending on $\|\eta\|_{H^{k+\f12}}$.
so we have that
\begin{equation}
\begin{aligned}
\|v\|_{H^2(\Om^\prime)}+\|q\|_{H^1(\Om^\prime)}+\|\Theta\|_{H^2(\Om^\prime)}&\lesssim
C(\eta)\Big(\|G^1\|_{H^0(\Om^\prime)}+\|G^2\|_{H^1(\Om^\prime)}\\
&\quad+\|G^3\|_{H^0(\Om^\prime)}+\|G^4\|_{H^{\f12}(\Sigma^\prime)}+\|G^5\|_{H^{\f12}(\Sigma^\prime)}\Big),
\end{aligned}
\end{equation}
for $C(\eta)$ a constant depending on $\|\eta\|_{H^{k+\f12}}$.
Then we may argue it as in Lemma 3.6 of \cite{GT1} to derive that, for $r=2, \ldots, k-1$,
\begin{equation}
\begin{aligned}
&\|v\|_{H^r(\Om^\prime)}+\|q\|_{H^{r-1}(\Om^\prime)}+\|\Theta\|_{H^r(\Om^\prime)}\\
&\lesssim
C(\eta)\Big(\|G^1\|_{H^{r-2}(\Om^\prime)}+\|G^2\|_{H^{r-1}(\Om^\prime)}+\|G^3\|_{H^{r-2}(\Om^\prime)}\\
&\quad+\|G^4\|_{H^{r-\f32}(\Sigma^\prime)}+\|G^5\|_{H^{r-\f32}(\Sigma^\prime)}\Big),
\end{aligned}
\end{equation}
for $C(\eta)$ a constant depending on $\|\eta\|_{H^{k+\f12}}$.

Now, we transform back to $\Om$ with $u=v\circ\Phi$, $p=q\circ\Phi$ and $\theta=\Theta\circ\Phi$. It is readily verified that $(u, p, T)$ are strong solutions of \eqref{equ:SBC}. According to Lemma \ref{equ:SBC},
  \begin{align*}
  \|u\|_{H^r}+\|p\|_{H^{r-1}}+\|\theta\|_{H^r}
  &\lesssim C(\eta)\Big(\|F^1\|_{H^{r-2}}+\|F^2\|_{H^{r-1}}+\|F^3\|_{H^{r-2}}\\
  &\quad+\|F^4\|_{H^{r-\f32}(\Sigma)}+\|F^5\|_{H^{r-\f32}(\Sigma)}\Big),
  \end{align*}
  whenever the right-hand side is finite, where $C(\eta)$ is a constant depending on $\|\eta\|_{H^{k+\f12}(\Sigma)}$. This is what we want.
\end{proof}

In the next lemma, we verify that the constant in \eqref{ineq:lower elliptic} can actually only depend on the initial free surface.
\begin{lemma}\label{lem:initial lower regularity}
  Let $k\ge3$ be an integer and suppose that $\eta\in H^{k+\f12}(\Sigma)$ and $\eta_0\in H^{k+\f12}(\Sigma)$. Then there exists a positive number $\varepsilon_0<1$ such that if $\|\eta-\eta_0\|_{H^{k-\f32}}\le \varepsilon_0$, the solution to \eqref{equ:SBC} satisfies
  \begin{equation}\label{ineq:initial lower elliptic}
    \begin{aligned}
      \|u\|_{H^r}+\|p\|_{H^{r-1}}+\|\theta\|_{H^r}
  &\lesssim C(\eta_0)\Big(\|F^1\|_{H^{r-2}}+\|F^2\|_{H^{r-1}}+\|F^3\|_{H^{r-2}}\\
  &\quad+\|F^4\|_{H^{r-\f32}(\Sigma)}+\|F^5\|_{H^{r-\f32}(\Sigma)}\Big),
    \end{aligned}
  \end{equation}
  for $r=2, \ldots, k-1$, whenever the right hand side is finite, where $C(\eta_0)$ is a constant depending on $\|\eta_0\|_{H^{k+\f12}}$.
\end{lemma}
\begin{proof}
  Here, we use the same idea as in Lemma 2.17 of \cite{LW}.  We rewrite the equation \eqref{equ:SBC} with its coefficients determined by $\eta_0$, i.e. it can be thought as a perturbation of equations of \eqref{equ:SBC} in terms of initial data,
  \ben
\left\{
\begin{aligned}
  \dive_{\mathscr{A}_0}S_{\mathscr{A}_0}(p, u)-\theta \nabla_{\mathscr{A}_0}y_{3,0}&=F^1+F^{1,0}& \quad\text{in}\quad\Om\\
 \nabla_{\mathscr{A}_0}\cdot u&=F^2+F^{2,0}&\quad\text{in}\quad\Om\\
 -\Delta_{\mathscr{A}_0}\theta&=F^3+F^{3,0}&\quad\text{in}\quad\Om\\
 S_{\mathscr{A}_0}(p, u)\mathscr{N}_0&=F^4+F^{4,0}&\quad\text{on}\quad \Sigma\\
 \nabla_{\mathscr{A}_0}\theta\cdot\mathscr{N}_0+\theta&\left|\mathscr{N}_0\right|=F^5+F^{5,0}&\quad\text{on}\quad \Sigma\\
 u=0,\quad\theta&=0&\quad\text{on}\quad \Sigma_b\\
\end{aligned}
\right.
\een
where
\begin{align*}
  F^{1,0}&=\nabla_{\mathscr{A}_0-\mathscr{A}}\cdot S_{\mathscr{A}}(p,u)+\nabla_{\mathscr{A}_0}\cdot S_{\mathscr{A}_0-\mathscr{A}}(p, u)+\theta\nabla_{\mathscr{A}_0-\mathscr{A}}y_3+\theta\nabla_{\mathscr{A}_0}(y_{3,0}-y_3),\\
  F^{2,0}&=\dive_{\mathscr{A}_0-\mathscr{A}}u,\\
  F^{3,0}&=\nabla_{\mathscr{A}_0-\mathscr{A}}\cdot\nabla_{\mathscr{A}}\theta+\nabla_{\mathscr{A}_0}\cdot\nabla_{\mathscr{A}_0-\mathscr{A}}\theta,\\
  F^{4,0}&=S_{\mathscr{A}_0}(p, u)(\mathscr{N}_0-\mathscr{N})+S_{\mathscr{A}_0-\mathscr{A}}(p, u)\mathscr{N},\\
  F^{5,0}&=\nabla_{\mathscr{A}_0}\theta\cdot(\mathscr{N}_0-\mathscr{N})+\nabla_{\mathscr{A}_0-\mathscr{A}}\theta\cdot\mathscr{N}+\theta\left(\left|\mathscr{N}_0\right|-\left|\mathscr{N}\right|\right).
\end{align*}
Here, $\mathscr{A}_0$, $\mathscr{N}_0$ and $y_{3,0}$ are quantities of $\mathscr{A}$, $\mathscr{N}$ and $y_{3}$ in terms of $\eta_0$. By the assumption, we know that $\eta-\eta_0\in H^{k+\f12}(\Sigma)$ and $\|\eta-\eta_0\|_{H^{k-\f32}(\Sigma)}^\ell\le \|\eta-\eta_0\|_{H^{k-\f32}(\Sigma)}<1$ for any positive integer $\ell$.
By the straightforward computation, we may derive that
\begin{align*}
  &\|F^{1,0}\|_{H^{r-2}}\le C\left(1+\|\eta_0\|_{H^{k+\f12}}\right)^4\|\eta-\eta_0\|_{H^{k-\f32}}\left(\|u\|_{H^r}+\|p\|_{H^{r-1}}+\|\theta\|_{H^{r-2}}\right),\\
  &\|F^{2,0}\|_{H^{r-1}}\le C\Big(1+\|\eta_0\|_{H^{k+\f12}}\Big)^2\|\eta-\eta_0\|_{H^{k-\f32}}\|u\|_{H^r},\\
  &\|F^{3,0}\|_{H^{r-2}}\le C\Big(1+\|\eta_0\|_{H^{k+\f12}}\Big)^4\|\eta-\eta_0\|_{H^{k-\f32}}\|\theta\|_{H^r},\\
  &\|F^{4,0}\|_{H^{r-\f32}(\Sigma)}\le C\Big(1+\|\eta_0\|_{H^{k+\f12}}\Big)^2\|\eta-\eta_0\|_{H^{k-\f32}}\left(\|u\|_{H^r}+\|p\|_{H^{r-1}}\right),\\
  &\|F^{5,0}\|_{H^{r-\f32}(\Sigma)}\le C\Big(1+\|\eta_0\|_{H^{k+\f12}}\Big)^2\|\eta-\eta_0\|_{H^{k-\f32}}\|\theta\|_{H^r},
\end{align*}
for $r=2, \ldots, k-1$.

Based on the Lemma \ref{lem:S lower regularity}, we have the estimate
\begin{align*}
  &\|u\|_{H^r}+\|p\|_{H^{r-1}}+\|\theta\|_{H^r}\\
  &\lesssim C(\eta_0)\Big(\|F^1+F^{1,0}\|_{H^{r-2}}+\|F^2+F^{2,0}\|_{H^{r-1}}+\|F^3+F^{3,0}\|_{H^{r-2}}\\
  &\quad+\|F^4+F^{4,0}\|_{H^{r-\f32}(\Sigma)}+\|F^5+F^{5,0}\|_{H^{r-\f32}(\Sigma)}\Big),
\end{align*}
where $C(\eta_0)$ is a constant depending on $\|\eta_0\|_{H^{k+\f12}}$. Combining the above estimates, we have
\begin{equation}
  \begin{aligned}
    &\|u\|_{H^r}+\|p\|_{H^{r-1}}+\|\theta\|_{H^r}\\
  &\lesssim C(\eta_0)\Big(\|F^1\|_{H^{r-2}}+\|F^2\|_{H^{r-1}}+\|F^3\|_{H^{r-2}}+\|F^4\|_{H^{r-\f32}(\Sigma)}+\|F^5\|_{H^{r-\f32}(\Sigma)}\Big)\\
  &\quad+C(\eta_0)\Big(1+\|\eta_0\|_{H^{k+\f12}}\Big)^4\|\eta-\eta_0\|_{H^{k-\f32}}\left(\|u\|_{H^r}+\|p\|_{H^{r-1}}+\|\theta\|_{H^r}\right),
  \end{aligned}
\end{equation}
for $r=2, \ldots, k-1$.
Then, if $\|\eta-\eta_0\|_{H^{k-\f32}}$ is to be chosen small enough such that the second term of the above inequality on the right-hand side less than $\f12(\|u\|_{H^r}+\|p\|_{H^{r-1}}+\|\theta\|_{H^r})$, then it can be absorbed into the left hand side, and we have that
\begin{align*}
  &\|u\|_{H^r}+\|p\|_{H^{r-1}}+\|\theta\|_{H^r}\\
  &\lesssim C(\eta_0)\Big(\|F^1\|_{H^{r-2}}+\|F^2\|_{H^{r-1}}+\|F^3\|_{H^{r-2}}+\|F^4\|_{H^{r-\f32}(\Sigma)}+\|F^5\|_{H^{r-\f32}(\Sigma)}\Big),
\end{align*}
for $r=2, \ldots, k-1$.
\end{proof}
Notice that the estimate in \eqref{ineq:initial lower elliptic} can only go up to $k-1$ order, which does not satisfy our requirement. In the next result, we can achieve two more order with a bootstrap argument, where we use the idea of \cite{LW}.
\begin{proposition}\label{prop:high regulatrity}
  Let $k\ge3$ be an integer. Suppose that $\eta\in H^{k+\f12}(\Sigma)$ as well as $\eta_0\in H^{k+\f12}(\Sigma)$ satisfying $\|\eta-\eta_0\|_{H^{k+\f12}(\Sigma)}\le\varepsilon_0$. Then the solution to \eqref{equ:SBC} satisfies
  \begin{equation}
    \begin{aligned}
      &\|u\|_{H^r}+\|p\|_{H^{r-1}}+\|\theta\|_{H^r}\\
  &\lesssim C(\eta_0)\Big(\|F^1\|_{H^{r-2}}+\|F^2\|_{H^{r-1}}+\|F^3\|_{H^{r-2}}+\|F^4\|_{H^{r-\f32}(\Sigma)}+\|F^5\|_{H^{r-\f32}(\Sigma)}\Big),
    \end{aligned}
  \end{equation}
  for $r=2, \ldots, k+1$, whenever the right hand side is finite, where $C(\eta_0)$ is a constant depending on $\|\eta_0\|_{H^{k+\f12}(\Sigma)}$.
\end{proposition}
\begin{proof}
  Here, we only consider the case for $r=k$ and $r=k+1$, since the conclusion has been proved when $r\le k-1$. For $m\in\mathbb{N}$, we define $\eta^m$ by throwing away high frequencies:

  \begin{equation*}
    {\hat{\eta}}^m(n)=\left\{
    \begin{aligned}
      &\hat{\eta}(n),\quad &\text{for}\quad |n|\le m-1,\\
      &0, \quad &\text{for}\quad |n|\ge m.
    \end{aligned}
    \right.
  \end{equation*} Then for each $m$, $\eta^m\in H^j(\Sigma)$ for arbitrary $j\ge0$ and $\eta^m\to \eta$ in $H^{k+\f12}(\Sigma)$ as $m\to\infty$.

  We consider the problem \eqref{equ:SBC} with $\mathscr{A}$ and $\mathscr{N}$ replaced by $\mathscr{A}^m$ and $\mathscr{N}^m$, and $y_3$ replaced by $y_3^m$. Since $\eta^m\in H^{\f52}$, we may apply Lemma \ref{lem:S lower regularity} to deduce that there exists a unique $(u^m, p^m, \theta^m)$ which solves
  \ben
  \left\{
  \begin{aligned}
    \dive_{\mathscr{A}^m}S_{\mathscr{A}^m}(p^m, u^m)-\theta^m \nabla_{\mathscr{A}^m}y_3^m&=F^1 \quad\text{in}\quad\Om\\
  \dive_{\mathscr{A}^m}u^m&=F^2\quad\text{in}\quad\Om\\
 -\Delta_{\mathscr{A}^m}\theta^m&=F^3\quad\text{in}\quad\Om\\
 S_{\mathscr{A}^m}(p^m, u^m)\mathscr{N}^m&=F^4\quad\text{on}\quad \Sigma\\
 \nabla_{\mathscr{A}^m}\theta^m\cdot\mathscr{N}^m+\theta^m\left|\mathscr{N}^m\right|&=F^5\quad\text{on}\quad \Sigma\\
 u^m=0,\quad \theta^m&=0\quad\text{on}\quad \Sigma_b\\
  \end{aligned}
  \right.
  \een
  and satisfies
  \begin{align*}
    \|u^m\|_{H^r}+\|p^m\|_{H^{r-1}}+\|\theta^m\|_{H^r}
  &\lesssim& C(\|\eta^m\|_{H^{k+\f52}})\Big(\|F^1\|_{H^{r-2}}+\|F^2\|_{H^{r-1}}+\|F^3\|_{H^{r-2}}\\
  &&+\|F^4\|_{H^{r-\f32}(\Sigma)}+\|F^5\|_{H^{r-\f32}(\Sigma)}\Big)
  \end{align*}
  for $r=2, \ldots, k+1$.
  In the following, we will prove that the constant $C(\|\eta^m\|_{H^{k+\f52}})$ can be improved only in terms of $\|\eta^m\|_{H^{k+\f12}}$.

  For convenience, we define
  \[
  \mathscr{Z}=C(\eta_0)P(\eta^m)\Big(\|F^1\|_{H^{r-2}}^2+\|F^2\|_{H^{r-1}}^2
  +\|F^3\|_{H^{r-2}}^2+\|F^4\|_{H^{r-\f32}(\Sigma)}^2+\|F^5\|_{H^{r-\f32}(\Sigma)}^2\Big)
  \]
  where $C(\eta_0)$ is a constant depending on $\|\eta_0\|_{H^{k+\f12}}$ and $P(\eta)$ is a polynomial of $\|\eta^m\|_{H^{k+\f12}}$. Then after the same computation as in the proof of Proposition 2.18 in \cite{LW} except for the only modification of $F$ replaced by $F^1+\theta^m \nabla_{\mathscr{A}^m}y_3^m$, we have
  \[
  \|u^m\|_{H^r}+\|p^m\|_{H^{r-1}}\lesssim\mathscr{Z},
  \]
  for $r=2, \ldots, k+1$. That's because in the above estimate, we only need to consider the terms $\|\theta^m\|_{H^r}$, for $r=2, \cdots, k-1$, but $\|\theta^m\|_{H^r}\lesssim\mathscr{Z}$ is assured by the Lemma \ref{ineq:initial lower elliptic}.

 Then we consider the temperature $\theta^m$. In the following of bootstrap argument, we may abuse the notation $\theta$ instead of $\theta^m$ and also for $\eta$, $\mathscr{A}$, $\mathscr{N}$, but they should be thought as $\eta^m$, $\mathscr{A}^m$, $\mathscr{N}^m$. We write explicitly the equation of $\theta$ as
 \ben\label{equ:equ T}
 \begin{aligned}
   &\pa_{11}\theta+\pa_{22}\theta+(1+A^2+B^2)K^2\pa_{33}\theta-2AK\pa_{13}\theta-2BK\pa_{23}\theta\\
   &\quad+(AK\pa_3(AK)+BK\pa_3(BK)-\pa_1(AK)-\pa_2(BK)+K\pa_3K)\pa_3\theta=-F^3.
 \end{aligned}
 \een
 \begin{enumerate}[step 1]
 \item $r=k$ case.
 By Lemma \ref{ineq:initial lower elliptic},
 \[
 \|\theta\|_{H^{k-1}}^2\lesssim C(\eta_0)\Big(\|F^3\|_{H^{k-3}}^2+\|F^5\|_{H^{k-\f52}(\Sigma)}^2\Big)\lesssim\mathscr{Z},
 \]
 where the constant $C(\eta_0)$ only depends on $\|\eta_0\|_{H^{k+\f12}}$.
 For $i=1,2$, since $\pa_i \theta$ satisfies the equation
 \beno
  \left\{
  \begin{aligned}
    -\Delta_{\mathscr{A}}\pa_i \theta&=\bar{F}^3 \quad \text{in}\quad \Om,\\
    \nabla_{\mathscr{A}}\pa_i \theta\cdot\mathscr{N}+\pa_i \theta\left|\mathscr{N}\right|&=\bar{F}^5 \quad \text{on}\quad \Sigma,\\
    \pa_i \theta&=0 \quad \text{on}\quad \Sigma_b,
  \end{aligned}
  \right.
  \eeno
  where
  \begin{align*}
    \bar{F}^3&=\pa_i F^3+\dive_{\pa_i\mathscr{A}}\nabla_{\mathscr{A}}\theta+\dive_{\mathscr{A}}\nabla_{\pa_i\mathscr{A}}\theta,\\
    \bar{F}^5&=\pa_i F^5-\nabla_{\pa_i\mathscr{A}}\theta\cdot\mathscr{N}-\nabla_{\mathscr{A}}\theta\cdot\pa_i\mathscr{N}-\theta\pa_i\left|\mathscr{N}\right|.
  \end{align*}
  Applying the Lemma A.1--A.2 in \cite{GT1}, we have
  \begin{align*}
    &\|\bar{F}^3\|_{H^{k-3}}^2+\|\bar{F}^5\|_{H^{k-\f52}(\Sigma)}^2\\
    &\lesssim \|F^3\|_{H^{k-2}}^2+\|F^5\|_{H^{k-\f32}(\Sigma)}^2+P(\eta)\|\theta\|_{H^{k-1}}^2\\
    &\lesssim\mathscr{Z}.
  \end{align*}
  Employing the $k-1$ order elliptic estimate, we have
  \[
  \|\pa_i \theta\|_{H^{k-1}}^2\lesssim C(\eta_0)\Big(\|\bar{F}^3\|_{H^{k-3}}^2+\|\bar{F}^5\|_{H^{k-\f52}(\Sigma)}^2\Big)\lesssim\mathscr{Z}.
  \]
  Then taking derivative $\pa_3^{k-2}$ on both sides of \eqref{equ:equ T} and focusing on the term $(1+A^2+B^2)K^2\pa_3^k\theta$, the estimates of all the other terms in $H^0$-norm implies that
  \[
  \|\pa_3^k\theta\|_{H^0}^2\lesssim\mathscr{Z}.
  \]
  Thus, we have proved that
  \[
  \|\theta\|_{H^k}^2\lesssim\mathscr{Z}.
  \]
  \item $r=k+1$ case.

  For $i,j=1,2$, since $\pa_{ij}\theta$ satisfies the equation
  \beno
  \left\{
  \begin{aligned}
    -\Delta_{\mathscr{A}}\pa_{ij} \theta&=\tilde{F}^3 \quad \text{in}\quad \Om,\\
    \nabla_{\mathscr{A}}\pa_{ij} \theta\cdot\mathscr{N}+\pa_{ij} \theta\left|\mathscr{N}\right|&=\tilde{F}^5 \quad \text{on}\quad \Sigma,\\
    \pa_{ij} \theta&=0 \quad \text{on}\quad \Sigma_b,
  \end{aligned}
  \right.
  \eeno
  where
  \begin{align*}
    \tilde{F}^3&=\pa_{ij}F^3+\dive_{\pa_{ij}\mathscr{A}}\nabla_{\mathscr{A}}\theta
    +\dive_{\mathscr{A}}\nabla_{\pa_{ij}\mathscr{A}}\theta+\dive_{\pa_i\mathscr{A}}\nabla_{\pa_j\mathscr{A}}\theta
    +\dive_{\pa_j\mathscr{A}}\nabla_{\pa_i\mathscr{A}}\theta\\
    &\quad+\dive_{\pa_i\mathscr{A}}\nabla_{\mathscr{A}}\pa_j \theta+\dive_{\pa_j\mathscr{A}}\nabla_{\mathscr{A}}\pa_i \theta+\dive_{\mathscr{A}}\nabla_{\pa_i\mathscr{A}}\pa_j \theta+\dive_{\mathscr{A}}\nabla_{\pa_j\mathscr{A}}\pa_i \theta,\\
    \tilde{F}^5&=\pa_{ij}F^5-\nabla_{\mathscr{A}} \theta\cdot\pa_{ij}\mathscr{N}-(\nabla_{\pa_i\mathscr{A}} \theta+\nabla_{\mathscr{A}}\pa_i\theta)\cdot\pa_{j}\mathscr{N}-(\nabla_{\pa_j\mathscr{A}} \theta+\nabla_{\mathscr{A}}\pa_j\theta)\cdot\pa_{i}\mathscr{N}\\
    &\quad-(\nabla_{\pa_{ij}\mathscr{A}}\theta+\nabla_{\pa_i\mathscr{A}}\pa_j\theta-\nabla_{\pa_j\mathscr{A}}\pa_i\theta)\mathscr{N}
    -\theta\pa_{ij}\left|\mathscr{N}\right|-\pa_i\theta\pa_j\left|\mathscr{N}\right|-\pa_j\theta\pa_i\left|\mathscr{N}\right|.
  \end{align*}
  Applying the Lemma A.1--A.2 in \cite{GT1} to the forcing terms, we have
  \begin{align*}
    &\|\tilde{F}^3\|_{H^{k-3}}^2+\|\tilde{F}^5\|_{H^{k-\f52}(\Sigma)}^2\\
    &\lesssim \|F^3\|_{H^{k-1}}^2+\|F^5\|_{H^{k-\f12}(\Sigma)}^2+P(\eta)\|\theta\|_{H^k}^2\\
    &\lesssim\mathscr{Z}.
  \end{align*}
  Then the Lemma \ref{ineq:initial lower elliptic} implies that
  \[
  \|\pa_{ij}\theta\|_{H^{k-1}}^2\lesssim C(\eta_0)\left(\|\tilde{F}^3\|_{H^{k-3}}^2+\|\tilde{F}^5\|_{H^{k-\f52}(\Sigma)}^2\right)
  \lesssim\mathscr{Z}.
  \]

  Since we have proved the case $r=k$, we take derivative $\pa_3^{k-2}\pa_i$ on both sides of \eqref{equ:equ T} for $i=1,2$ and focus on the term of $(1+A^2+B^2)K^2\pa_3^k\pa_i\theta$. Utilizing the estimates of all the other terms in $H^0$-norm, we have
  \[
  \|\pa_3^k\pa_i\theta\|_{H^0}^2\lesssim\mathscr{Z}.
  \]
  Then, taking derivative $\pa_3^{k-1}$ on both sides of \eqref{equ:equ T} and focusing on the term of $(1+A^2+B^2)K^2\pa_3^{k+1}\theta$, by all the estimates above, we have
  \[
  \|\pa_3^{k+1}\theta\|_{H^0}^2\lesssim\mathscr{Z}.
  \]
  Therefore, we have proved
  \[
  \|\theta\|_{H^{k+1}}^2\lesssim\mathscr{Z}.
  \]
  \end{enumerate}
  Now, we go back to the original notation. According to the convergence of $\eta^m$, we have
  \ben\label{ineq:bound}
  \begin{aligned}
  &\|u^m\|_{H^r}^2+\|p^m\|_{H^{r-1}}^2+\|\theta^m\|_{H^r}^2\\
  &\lesssim C(\eta_0)P(\eta^m)\Big(\|F^1\|_{H^{r-2}}^2+\|F^2\|_{H^{r-1}}^2
  +\|F^3\|_{H^{r-2}}^2+\|F^4\|_{H^{r-\f32}(\Sigma)}^2+\|F^5\|_{H^{r-\f32}(\Sigma)}^2\Big)\\
  &\lesssim C(\eta_0)P(\eta)\Big(\|F^1\|_{H^{r-2}}^2+\|F^2\|_{H^{r-1}}^2
  +\|F^3\|_{H^{r-2}}^2+\|F^4\|_{H^{r-\f32}(\Sigma)}^2+\|F^5\|_{H^{r-\f32}(\Sigma)}^2\Big)\\
  &\lesssim C(\eta_0)\Big(\|F^1\|_{H^{r-2}}^2+\|F^2\|_{H^{r-1}}^2
  +\|F^3\|_{H^{r-2}}^2+\|F^4\|_{H^{r-\f32}(\Sigma)}^2+\|F^5\|_{H^{r-\f32}(\Sigma)}^2\Big),
  \end{aligned}
  \een
  for $r=2, \ldots, k+1$, where in the last inequality we have used the assumption that $\|\eta-\eta_0\|_{H^{k+\f12}}\le\varepsilon_0$ and the term $P(\eta_0)$ is absorbed by $C(\eta_0)$. Here $C(\eta_0)$ depends only on $\|\eta_0\|_{H^{k+\f12}}$.

  The inequality of boundedness \eqref{ineq:bound} implies that the sequence $\{(u^m, p^m, \theta^m)\}$ is uniformly bounded in $H^r\times H^{r-1}\times H^r$, so we can extract a weakly convergent subsequence, which is still denoted by $\{(u^m, p^m, \theta^m)\}$. That is, $u^m\rightharpoonup u^0$ in $H^r(\Om)$, $p^m\rightharpoonup p^0$ in $H^{r-1}(\Om)$ and $\theta^m\rightharpoonup \theta^0$ in $H^r(\Om)$. Since $\eta^m\rightarrow\eta$ in $H^{k+\f12}(\Sigma)$, we also have that $\mathscr{A}^m\to\mathscr{A}$, $J^m\to J$ in $H^k(\Om)$, and $\mathscr{N}^m\to\mathscr{N}$ in $H^{k-\f12}(\Sigma)$.

   After multiplying the equation $\dive_{\mathscr{A}^m}u^m=F^2$ by $wJ^m$ for $w\in C_c^\infty(\Om)$ and integrating by parts, we see that
   \begin{align*}
   \int_{\Om}F^2wJ^m=\int_{\Om}\dive_{\mathscr{A}^m}(u^m)wJ^m&=-\int_{\Om}u^m\cdot\nabla_{\mathscr{A}^m}wJ^m\\
   &\to-\int_{\Om}u^0\cdot\nabla_{\mathscr{A}}wJ=\int_{\Om}\dive_{\mathscr{A}}(u^0)wJ,
   \end{align*}
   from which we deduce that $\dive_{\mathscr{A}}u^0=F^2$.
   Then multiplying the third equation in \eqref{equ:SBC} by $wJ^m$ for $w\in{}_0H^1(\Om)$ and integrating by parts, we have that
    \[
    \int_{\Om}\nabla_{\mathscr{A}^m}\theta^m\cdot\nabla_{\mathscr{A}^m}wJ^m+\int_{\Sigma}\theta^mw\left|\mathscr{N}^m\right|=\int_{\Om}F^3wJ^m+\int_{\Sigma}F^5w,
    \]
    which, by passing to the limit $m\to\infty$, reveals that
    \[
    \int_{\Om}\nabla_{\mathscr{A}}\theta^0\cdot\nabla_{\mathscr{A}}wJ+\int_{\Sigma}\theta^0w\left|\mathscr{N}\right|=\int_{\Om}F^3wJ+\int_{\Sigma}F^5w.
    \]
    Finally we multiply the first equation in \eqref{equ:SBC} by $wJ^m$ for $w\in{}_0H^1(\Om)$ and integrate by parts to see that
   \[
   \int_{\Om}\f12\mathbb{D}_{\mathscr{A}^m}u^m:\mathbb{D}_{\mathscr{A}^m}wJ^m-p^mJ^m-\theta^m\nabla_{\mathscr{A}^m}y_3^m\cdot wJ^m=\int_{\Om}F^1\cdot wJ^m-\int_{\Sigma}F^4\cdot w.
   \]
   Passing to the limit $m\to\infty$ , we deduce that
   \[
   -\int_{\Om}\f12\mathbb{D}_{\mathscr{A}}u^0:\mathbb{D}_{\mathscr{A}} wJ+p^0\dive_{\mathscr{A}}(w)J-\theta^0\nabla_{\mathscr{A}}y_3\cdot wJ=\int_{\Om}F^1\cdot wJ-\int_{\Sigma}F^4\cdot w.
   \]
   After integrating by parts again, we deduce that $(u^0, p^0, \theta^0)$ satisfies \eqref{equ:SBC}. Since $(u, p, \theta)$ is the unique solution to \eqref{equ:SBC}, we have that $u=u^0$, $p=p^0$ and $\theta=\theta^0$. Then, according to the weak lower semicontinuity and the uniform boundedness of \eqref{ineq:bound}, we have that
    \begin{align*}
      &\|u\|_{H^r}+\|p\|_{H^{r-1}}+\|\theta\|_{H^r}\\
  &\lesssim C(\eta_0)\Big(\|F^1\|_{H^{r-2}}+\|F^2\|_{H^{r-1}}+\|F^3\|_{H^{r-2}}+\|F^4\|_{H^{r-\f32}(\Sigma)}+\|F^5\|_{H^{r-\f32}(\Sigma)}\Big),
    \end{align*}
  for $r=2, \ldots, k+1$, where $C(\eta_0)$ is a constant depending on $\|\eta_0\|_{H^{k+\f12}(\Sigma)}$.
\end{proof}

\subsection{The $\mathscr{A}$-Poisson problem}

Now we consider the elliptic problem
\ben\label{equ:poisson}
\left\{
\begin{aligned}
  &\Delta_{\mathscr{A}}p=f^1\quad&\text{in}\thinspace\Om,\\
  &p=f^2\quad&\text{on}\thinspace\Sigma,\\
  &\nabla_{\mathscr{A}}p\cdot\nu=f^3\quad&\text{on}\thinspace\Sigma_b,
\end{aligned}
\right.
\een
where $\nu$ is the outward--pointing normal on $\Sigma_b$. The details of elliptic estimates of \eqref{equ:poisson} has been interpreted in \cite{GT1} and \cite{LW}, so we omit them here.

\section{Linear estimates}

Now we study the problem \eqref{equ:linear BC}, following the path of \cite{GT1}. First, we will employ two notions of solution: weak and strong.

\subsection{The weak solution}

Suppose that a smooth solution to \eqref{equ:linear BC} exists, then by integrating over $\Om$ by parts, and in time from $0$ to $T$, we see that
\ben
\begin{aligned}
\left(\pa_tu, \psi\right)_{L^2\mathscr{H}^0}+\f12\left(u, \psi\right)_{L^2\mathscr{H}^1}-\left(p, \dive_{\mathscr{A}}\psi\right)_{L^2\mathscr{H}^0}-\left(\theta \nabla_{\mathscr{A}}y_3,\psi\right)_{L^2\mathscr{H}^0}\\
=\left(F^1, \psi\right)_{L^2\mathscr{H}^0}
-\left(F^4, \psi\right)_{L^2H^0(\Sigma)},\\
\left(\pa_t\theta, \phi\right)_{L^2\mathscr{H}^0}+\left(\nabla_{\mathscr{A}}\theta, \nabla_{\mathscr{A}}\phi\right)_{L^2\mathscr{H}^0}+\left(\theta\left|\mathscr{N}\right|, \phi\right)_{L^2H^0(\Sigma)}\\
=\left(F^3, \psi\right)_{L^2\mathscr{H}^0}
+\left(F^5, \psi\right)_{L^2H^0(\Sigma)},
\end{aligned}
\een
for $\phi$, $\psi\in\mathscr{H}^1_T$.

If we were to restrict the test function $\psi$ to $\psi\in\mathscr{X}$, the term $\left(p, \dive_{\mathscr{A}}\psi\right)_{L^2\mathscr{H}^0}$ would vanish. Then we have a pressureless weak formulation.
\ben
\begin{aligned}
\left(\pa_tu, \psi\right)_{L^2\mathscr{H}^0}+\f12\left(u, \psi\right)_{L^2\mathscr{H}^1}-\left(\theta \nabla_{\mathscr{A}}y_3,\psi\right)_{L^2\mathscr{H}^0}\\
=\left(F^1, \psi\right)_{L^2\mathscr{H}^0}
-\left(F^4, \psi\right)_{L^2H^0(\Sigma)},\\
\left(\pa_t\theta, \phi\right)_{L^2\mathscr{H}^0}+\left(\nabla_{\mathscr{A}}\theta, \nabla_{\mathscr{A}}\phi\right)_{L^2\mathscr{H}^0}+\left(\theta\left|\mathscr{N}\right|, \phi\right)_{L^2H^0(\Sigma)}\\
=\left(F^3, \psi\right)_{L^2\mathscr{H}^0}
+\left(F^5, \psi\right)_{L^2H^0(\Sigma)},
\end{aligned}
\een
This leads us to define a weak solution without pressure.
\begin{definition}
  Suppose that $u_0\in \mathscr{Y}(0)$, $\theta_0\in H^0(\Om)$, $F^1-F^4\in (\mathscr{X}_T)^\ast$ and $F^3+F^5\in(\mathscr{H}^1_T)^\ast$. If there exists a pair $(u, \theta)$ achieving the initial data $u_0$, $\theta_0$ and satisfies $u\in\mathscr{H}^1_T$, $\theta\in\mathscr{H}^1_T$ and $\pa_tu\in (\mathscr{X}_T)^\ast$, $\pa_t \theta\in (\mathscr{H}^1_T)^\ast$, such that
\ben\label{equ:lpws}
\begin{aligned}
&\left<\pa_tu, \psi\right>_{(\mathscr{X}_T)^\ast}+\f12\left(u, \psi\right)_{L^2\mathscr{H}^1}-\left(\theta \nabla_{\mathscr{A}}y_3,\psi\right)_{L^2\mathscr{H}^0}=\left(F^1-F^4, \psi\right)_{(\mathscr{X}_T)^\ast},\\
&\left<\pa_t\theta, \phi\right>_{(\mathscr{H}^1_T)^\ast}+\left(\theta,\phi\right)_{L^2\mathscr{H}^1}+\left(\theta\left|\mathscr{N}\right|, \phi\right)_{L^2H^0(\Sigma)}=\left(F^3+F^5, \psi\right)_{(\mathscr{H}^1_T)^\ast},
\end{aligned}
\een
   holds for any $\psi\in\mathscr{X}_T$ and $\phi\in\mathscr{H}^1_T$, we call the pair $(u, \theta)$ a pressureless weak solution.
\end{definition}
Since our aim is to construct solutions with high regularity to \eqref{equ:linear BC}, we will directly construct strong solutions to \eqref{equ:lpws}.  And it is easy to see that weak solutions will arise as a byproduct of the construction of  strong solutions to  \eqref{equ:linear BC}. Hence, we will not study the existence of weak solutions.

Now we derive some properties and uniqueness of weak solutions.
\begin{lemma}
  Suppose that $u$, $\theta$ are weak solutions of \eqref{equ:lpws}. Then, for almost every $t\in [0,T]$,
  \ben\label{eq:integral}
  \begin{aligned}
    \f12\|u(t)\|_{\mathscr{H}^0(t)}^2+\f12\int_0^t\|u(s)\|_{\mathscr{H}^1(s)}^2\,\mathrm{d}s=\f12\|u(0)\|_{\mathscr{H}^0(0)}^2+\left(F^1-F^4,u\right)_{(\mathscr{X}_t)^\ast}\\
    +\f12\int_0^t\int_{\Om}|u(s)|^2\pa_sJ(s)\,\mathrm{d}s+\int_0^t\int_{\Om}\theta(s)\nabla_{\mathscr{A}}y_3\cdot u(s)\,\mathrm{d}s,\\
    \f12\|\theta(t)\|_{\mathscr{H}^0(t)}^2+\int_0^t\|\theta(s)\|_{\mathscr{H}^1(s)}^2\,\mathrm{d}s+\int_0^t\int_{\Sigma}|\theta(s)|^2\left|\mathscr{N}\right|\,\mathrm{d}s=\f12\|\theta(0)\|_{\mathscr{H}^0(0)}^2\\
    +\left(F^3+F^5,\theta\right)_{(\mathscr{H}^1_t)^\ast}+\f12\int_0^t\int_{\Om}|\theta(s)|^2\pa_sJ(s)\,\mathrm{d}s.
  \end{aligned}
  \een
  Also,
  \beq\label{est:weak theta}
  \sup_{0\le t\le T}\|\theta(t)\|_{\mathscr{H}^0(t)}^2+\|\theta\|_{\mathscr{H}^1_T}^2\lesssim \exp\left(C_0(\eta)T\right)\left(\|\theta(0)\|_{\mathscr{H}^0(0)}^2+\|F^3+F^5\|_{(\mathscr{H}^1_T)^\ast}^2\right),
  \eeq
  \ben\label{est:weak u}
  \begin{aligned}
    \sup_{0\le t\le T}\|u(t)\|_{\mathscr{H}^0(t)}^2+\|u\|_{\mathscr{H}^1_T}^2\lesssim \exp\left(CC_0(\eta)T\right)\Big(\|u(0)\|_{\mathscr{H}^0(0)}^2+\|\theta(0)\|_{\mathscr{H}^0(0)}^2\\
    +\|F^1-F^4\|_{(\mathscr{X}_T)^\ast}^2+\|F^3+F^5\|_{(\mathscr{H}^1_T)^\ast}^2\Big),
  \end{aligned}
  \een
  where $C_0(\eta):=\max\{\sup_{0\le t\le T}\|\pa_tJK\|_{L^\infty}, \sup_{0\le t\le T}\|\nabla_{\mathscr{A}}y_3\|_{L^\infty}\}$.
\end{lemma}
\begin{proof}
  The identity \eqref{eq:integral} follows directly from Lemma 2.4 in \cite{GT1} and \eqref{equ:lpws} by using the test function $\psi=u\chi_{[0,t]}\in \mathscr{X}_T$, and $\phi=\theta\chi_{[0,t]}\in \mathscr{H}^1_T$, where $\chi_{[0,t]}$ is a temporal indicator function to $1$ on the interval $[0,t]$.

  From \eqref{eq:integral}, we can directly derive the inequalities
  \beq\label{est:weak theta1}
  \f12\|\theta(t)\|_{\mathscr{H}^0(t)}^2+\|\theta\|_{\mathscr{H}^1_t}^2\le \f12\|\theta(0)\|_{\mathscr{H}^0(0)}^2+\|F^3+F^5\|_{(\mathscr{H}^1_t)^\ast}\|\theta\|_{\mathscr{H}^1_t}+\f12C_0(\eta)\|\theta(t)\|_{\mathscr{H}^0_t}^2,
  \eeq
  \ben\label{est:weak u1}
  \begin{aligned}
  \f12\|u(t)\|_{\mathscr{H}^0(t)}^2+\f12\|u\|_{\mathscr{H}^1_t}^2\le \f12\|u(0)\|_{\mathscr{H}^0(0)}^2+\|F^1-F^4\|_{(\mathscr{X}_t)^\ast}\|u\|_{\mathscr{H}^1_t}\\
  +\f12C_0(\eta)\|u(t)\|_{\mathscr{H}^0_t}^2+CC_0(\eta)\|\theta\|_{\mathscr{H}^1_t}\|u\|_{\mathscr{H}^1_t},
  \end{aligned}
  \een
  where, for \eqref{est:weak u1}, we have used the Poincar\'e inequality in Lemma A.14 on \cite{GT1}, and
  \[
  \|u\|_{\mathscr{H}^k_t}^2=\int_0^t\|u(s)\|_{\mathscr{H}^k(s)}^2\,\mathrm{d}s\quad \text{for}\thinspace k=0,1,
  \]
  and similarly for $\|\theta\|_{\mathscr{H}^k_t}^2$, $\|F^1-F^4\|_{(\mathscr{X}_t)^\ast}$, $\|F^3+F^5\|_{(\mathscr{H}^1_t)^\ast}$. Inequalities \eqref{est:weak theta1}, \eqref{est:weak u1} and Cauchy inequality imply that
  \ben\label{ineq:integral}
  \begin{aligned}
    \f12\|\theta(t)\|_{\mathscr{H}^0(t)}^2+\f34\|\theta\|_{\mathscr{H}^1_t}^2\le \f12\|\theta(0)\|_{\mathscr{H}^0(0)}^2+\|F^3-F^5\|_{(\mathscr{H}^1_t)^\ast}^2+\f12C_0(\eta)\|\theta(t)\|_{\mathscr{H}^0_t}^2,\\
    \f12\|u(t)\|_{\mathscr{H}^0(t)}^2+\f18\|u\|_{\mathscr{H}^1_t}^2\le \f12\|u(0)\|_{\mathscr{H}^0(0)}^2+\|F^1-F^4\|_{(\mathscr{X}_t)^\ast}^2
  +\f12C_0(\eta)\|u(t)\|_{\mathscr{H}^0_t}^2\\
  +CC_0(\eta)\|\theta\|_{\mathscr{H}^1_t}^2,
  \end{aligned}
  \een
  Then \eqref{est:weak theta} and \eqref{est:weak u} follow from the integral inequality \eqref{ineq:integral} and Gronwall's lemma.
\end{proof}
\begin{proposition}
  Weak solutions to \eqref{equ:lpws} are unique.
\end{proposition}
\begin{proof}
  Suppose that $(u^1,\theta^1)$ and $(u^2,\theta^2)$ are both weak solutions to \eqref{equ:lpws}, then $(w,\vartheta)$, defined by $w=u^1-u^2$ and $\vartheta=\theta^1-\theta^2$, is a weak solution with $F^1-F^4=0$, $F^3+F^5=0$, $w(0)=u^1(0)-u^2(0)=0$ and $\vartheta(0)=\theta^1(0)-\theta^2(0)$. Then the bounds \eqref{est:weak theta} and \eqref{est:weak u} imply that $w=0$ and $\vartheta=0$. Hence, weak solutions to \eqref{equ:lpws} are unique.
\end{proof}

\subsection{The strong solution}
Before we define the strong solution, we need to define an operator $D_t$ as
\beq\label{def:Dt}
D_tu:=\pa_tu-Ru\quad \text{for}\quad R:=\pa_tMM^{-1},
\eeq
with $M=K\nabla\Phi$, where $K$, $\Phi$ are as defined in \eqref{map:phi} and \eqref{equ:components}. It is easily to be known that $D_t$ preserves the $\dive_{\mathscr{A}}$ -free condition, since
\[
J\dive_{\mathscr{A}}(D_tv)=J\dive_{\mathscr{A}}(M\pa_t(M^{-1}v))=\dive(\pa_t(M^{-1}v))=\pa_t\dive(M^{-1}v)=\pa_t(J\dive_{\mathscr{A}}v),
\]
where the equality $J\dive_{\mathscr{A}}v=\dive(M^{-1}v)$ can be found in Page 299 of \cite{GT1}.
\begin{definition}\label{def:strong solution}
  Suppose that the forcing functions satisfy
  \ben\label{cond:force}
  \begin{aligned}
    F^1&\in L^2([0, T]; H^1(\Om))\cap C^0([0, T]; H^0(\Om)),\\
    F^3&\in L^2([0, T]; H^1(\Om))\cap C^0([0, T]; H^0(\Om)),\\
    F^4&\in L^2([0, T]; H^{\f32}(\Sigma))\cap C^0([0, T]; H^{\f12}(\Sigma)),\\
    \pa_t(F^1-F^4)&\in L^2([0, T]; ({}_0H^1(\Om))^\ast), \quad\pa_t(F^3+F^5)\in L^2([0, T]; ({}_0H^1(\Om))^\ast).
  \end{aligned}
  \een
  We also assume that $u_0\in H^2\cap \mathscr{X}(0)$ and $\theta_0\in H^2\cap\mathscr{H}^1(0)$. If there exists a pair $(u, p, \theta)$ achieving the initial data $u_0$, $\theta_0$ and satisfies
  \ben\label{equ:strong solution}
  \begin{aligned}
    &u\in L^2([0, T]; H^3)\cap C^0([0,T];H^2)\cap \mathscr{X}_T \thinspace &\pa_tu\in L^2([0, T]; H^1)\cap C^0([0,T];H^0)\\
    &D_tu\in \mathscr{X}_T,\quad\pa_t^2u\in\mathscr{X}_T^\ast \thinspace &p\in L^2([0, T]; H^2)\cap C^0([0,T];H^1)\\
    &\theta\in L^2([0, T]; H^3)\cap C^0([0,T];H^2) \thinspace &\pa_t\theta\in L^2([0, T]; H^1)\cap C^0([0,T];H^0)\\
    &\pa_t^2\theta\in(\mathscr{H}_T^1)^\ast,
  \end{aligned}
  \een
  such that they satisfies \eqref{equ:linear BC} in the strong sense, we call it a strong solution.
\end{definition}
Then, we have to prove the lower regularity of strong solutions.
\begin{theorem}\label{thm:lower regularity}
  Suppose that the forcing terms and the initial data satisfy the condition in Definition \ref{def:strong solution}, and that $u_0$, $F^4(0)$ satisfy the compatibility condition
  \beq\label{cond:compatibility}
  \Pi_0\left(F^4(0)+\mathbb{D}_\mathscr{A_0}u_0\mathscr{N}_0\right)=0,\quad \text{where}\thinspace \mathscr{N}_0=(-\pa_1\eta_0, -\pa_2\eta_0, 1),
  \eeq
  and $\Pi_0$ is an orthogonal projection onto the tangent space of the surface $\{x_3=\eta_0\}$ defined by
  \beq\label{def:projection}
  \Pi_0v=v-(v\cdot\mathscr{N}_0)\mathscr{N}_0|\mathscr{N}_0|^{-2}.
  \eeq
  Then there exists a strong solution $(u, p, \theta)$ satisfying \eqref{equ:strong solution}. Moreover,
  \ben\label{inequ:est strong solution}
  \begin{aligned}
  &\|u\|_{L^\infty H^2}^2+\|u\|_{L^2H^3}^2+\|\pa_tu\|_{L^\infty H^0}^2+\|\pa_tu\|_{L^2H^1}^2+\|\pa_t^2u\|_{(\mathscr{X}_T)^\ast}+\|p\|_{L^\infty H^1}^2+\|p\|_{L^2H^2}^2\\
  &\quad+\|\theta\|_{L^\infty H^2}^2+\|\theta\|_{L^2H^3}^2+\|\pa_t\theta\|_{L^\infty H^0}^2+\|\pa_t\theta\|_{L^2H^1}^2+\|\pa_t^2\theta\|_{(\mathscr{H}^1_T)^\ast}\\
  &\lesssim P(\|\eta_0\|_{H^{5/2}})\left(1+\mathscr{K}(\eta)\right)\exp\left(C(1+\mathscr{K}(\eta))T\right)\Big(\|u_0\|_{H^2}^2+\|\theta_0\|_{H^2}^2+\|F^1(0)\|_{H^0}^2\\
  &\quad+\|F^3(0)\|_{H^0}^2+\|F^4(0)\|_{H^{1/2}(\Sigma)}^2+\|F^1\|_{L^2H^1}^2+\|F^3\|_{L^2H^1}^2+\|F^4\|_{L^2H^{3/2}(\Sigma)}^2\\
  &\quad+\|F^5\|_{L^2H^{3/2}(\Sigma)}^2+\|\pa_t(F^1-F^4)\|_{(\mathscr{X}_T)^\ast}^2+\|\pa_t(F^3+F^5)\|_{(\mathscr{H}^1_T)^\ast}^2\Big),
  \end{aligned}
  \een
  where $C$ is a constant independent of $\eta$ and $\mathscr{K}(\eta)$ is defined as
  \beq
  \mathscr{K}(\eta):=\sup_{0\le t\le T}\left(\|\eta\|_{H^{9/2}}^2+\|\pa_t\eta\|_{H^{7/2}}^2+\|\pa_t^2\eta\|_{H^{5/2}}^2\right).
  \eeq
  The initial pressure, $p(0)\in H^1(\Om)$ is determined by terms $u_0$, $\theta_0$, $F^1(0)$, $F^4(0)$ as a weak solution to
  \ben\label{equ:p0}
  \left\{
  \begin{aligned}
    &\dive_{\mathscr{A}_0}\left(\nabla_{\mathscr{A}_0}p(0)-F^1(0)-\theta_0\nabla_{\mathscr{A}_0}y_{3,0}\right)=-\dive_{\mathscr{A}_0}(R(0)u_0)\in H^0(\Om),\\
    &p(0)=(F^4(0)+\mathbb{D}_{\mathscr{A}_0}u_0\mathscr{N}_0)\cdot\mathscr{N}_0|\mathscr{N}_0|^{-2}\in H^{1/2}(\Sigma),\\
    &\left(\nabla_{\mathscr{A}_0}p(0)-F^1(0)\right)\cdot\nu=\Delta_{\mathscr{A}_0}u_0\cdot\nu\in H^{-1/2}(\Sigma_b),
  \end{aligned}
  \right.
  \een
  where $y_{3,0}$ in terms of $\eta_0$.
  Also, $\pa_t\theta(0)$ satisfies
  \beq
  \pa_t\theta(0)=\Delta_{\mathscr{A}_0}\theta_0+F^3(0) \in H^0(\Om),
  \eeq
  and $D_tu(0)=\pa_tu(0)-R(0)u_0$ satisfies
  \beq
  D_tu(0)=\Delta_{\mathscr{A}_0}u_0-\nabla_{\mathscr{A}_0}p(0)+F^1(0)+\theta_0e_3-R(0)u_0\in\mathscr{Y}(0).
  \eeq
  Moreover, $\pa_t\theta$ satisfies
  \ben\label{equ:pat theta}
  \left\{
  \begin{aligned}
    &\pa_t(\pa_t\theta)-\Delta_{\mathscr{A}}(\pa_t\theta)=\pa_tF^3+G^3\quad&\text{in}\quad\Om,\\
    &\nabla_{\mathscr{A}}(\pa_t\theta)\cdot\mathscr{N}+\pa_t\theta\left|\mathscr{N}\right|=\pa_tF^5+G^5\quad&\text{on}\quad\Sigma,\\
    &\pa_t\theta=0\quad&\text{on}\quad\Sigma_b,
  \end{aligned}
  \right.
  \een
   and $D_tu$ satisfies
  \ben\label{equ:Dt u}
  \left\{
  \begin{aligned}
    &\pa_t(D_tu)-\Delta_{\mathscr{A}}(D_tu)+\nabla_{\mathscr{A}}(\pa_tp)-D_t(\theta \nabla_{\mathscr{A}}y_3)=D_tF^1+G^1\quad&\text{in}\quad\Om,\\
    &\dive_{\mathscr{A}}(D_tu)=0\quad&\text{in}\quad\Om,\\
    &S_{\mathscr{A}}(\pa_tp,D_tu)\mathscr{N}=\pa_tF^4+G^4\quad&\text{on}\quad\Sigma,\\
    &D_tu=0\quad&\text{on}\quad\Sigma_b,
  \end{aligned}
  \right.
  \een
  in the weak sense of \eqref{equ:lpws}, where $G^1$ is defined by
  \[
  G^1=-(R+\pa_tJK)\Delta_{\mathscr{A}}u-\pa_tRu+(\pa_tJK+R+R^\top)\nabla_{\mathscr{A}}p+\dive_{\mathscr{A}}(\mathbb{D}_{\mathscr{A}}(Ru)-R\mathbb{D}_{\mathscr{A}}u+\mathbb{D}_{\pa_t\mathscr{A}}u)
  \]
  ($R^\top$ denoting the matrix transpose of $R$), $G^3$ by
  \[
  G^3=-\pa_tJK\Delta_{\mathscr{A}}\theta+\dive_{\mathscr{A}}(-R\nabla_{\mathscr{A}}\theta+\nabla_{\pa_t\mathscr{A}}\theta),
  \]
  $G^4$ by
  \[
  G^4=\mathbb{D}_{\mathscr{A}}(Ru)\mathscr{N}-(pI-\mathbb{D}_{\mathscr{A}}u)\pa_t\mathscr{N}+\mathbb{D}_{\pa_t\mathscr{A}}u\mathscr{N},
  \]
  and $G^5$ by
  \[
  G^5=-\nabla_{\mathscr{A}}\theta\cdot\pa_t\mathscr{N}-\nabla_{\pa_t\mathscr{A}}\theta\cdot\mathscr{N}-\theta\pa_t\left|\mathscr{N}\right|.
  \]
  More precisely, \eqref{equ:pat theta} and \eqref{equ:Dt u} hold in the weak sense of \eqref{equ:lpws} in that
  \ben\label{equ:weak pat theta}
  \begin{aligned}
    &\left<\pa_t^2\theta,\phi\right>_{(\mathscr{H}^1_T)^\ast}+\left(\pa_t\theta,\phi\right)_{\mathscr{H}^1_T}+\left(\pa_t\theta\left|\mathscr{N}\right|,\phi\right)_{L^2H^0(\Sigma)}\\
          &=\left<\pa_t(F^3+F^5)\right>_{(\mathscr{H}^1_T)^\ast}+\left(\pa_tJKF^3,\phi\right)_{\mathscr{H}^0_T}-\left(\pa_tJK\pa_t\theta,\phi\right)_{\mathscr{H}^0_T}\\
          &\quad-\int_0^T\int_{\Om}\left(\pa_tJK\nabla_{\mathscr{A}}\theta\cdot\nabla_{\mathscr{A}}\phi+\nabla_{\pa_t\mathscr{A}}\theta\cdot\nabla_{\mathscr{A}}\phi+\nabla_{\mathscr{A}}\theta\cdot\nabla_{\pa_t\mathscr{A}}\phi\right)J
  \end{aligned}
  \een
  and
  \ben
  \begin{aligned}
    &\left<\pa_tD_tu,\psi\right>_{(\mathscr{X}_T)^\ast}+\f12\left(\pa_tu,\psi\right)_{\mathscr{H}^1_T}\\
    &=\left<\pa_t(F^1-F^4),\psi\right>_{(\mathscr{X}_T)^\ast}+\left(\pa_t(\theta \nabla_{\mathscr{A}}y_3),\psi\right)_{\mathscr{H}^0}-\left(\pa_tRu+R\pa_tu,\psi\right)_{\mathscr{H}^0_T}\\
    &\quad+\left(\pa_tJKF^1,\psi\right)_{\mathscr{H}^0_T}-\left(\pa_tJK\theta e_3,\psi\right)_{\mathscr{H}^0_T}-\left(\pa_tJK\pa_tu,\psi\right)_{\mathscr{H}^0_T}-\left(p,\dive_{\mathscr{A}}(R\psi)\right)_{\mathscr{H}^0_T}\\
    &\quad-\f12\int_0^T\int_\Om\left(\pa_tJK\mathbb{D}_{\mathscr{A}}u:\mathbb{D}_{\mathscr{A}}\psi+\mathbb{D}_{\pa_t\mathscr{A}}u:\mathbb{D}_{\mathscr{A}}\psi+\mathbb{D}_{\mathscr{A}}u:\mathbb{D}_{\pa_t\mathscr{A}}\psi\right)J
  \end{aligned}
  \een
  for all $\phi\in\mathscr{H}^1_T$, $\psi\in\mathscr{X}_T$.
\end{theorem}
\begin{proof}
Here we will use the Galerkin method, which may be referred to \cite{Evans}.

  Step 1. The construction of approximate solutions for $\theta$. Since the scalar-valued space $H^2(\Om)\cap{}_0H^1(\Om)$ is separable, we can choose a countable basis $\{\tilde{w}^j\}_{i=1}^\infty$. Note that this basis is time-independent. Now, we need to construct a time-dependent basis for $H^2\cap\mathscr{H}^1$. We define $\phi^j=\phi^j(t):=K(t)\tilde{w}^j$. According to the Proposition \ref{prop:k}, $\phi^j(t)\in H^2(\Om)\cap\mathscr{H}^1(t)$, and $\{\phi^j(t)\}_{j=1}^\infty$ is a basis of $H^2(\Om)\cap\mathscr{H}^1(t)$ for each $t\in [0, T]$. Moreover,
      \beq\label{equ:dt phi}
      \pa_t\phi^j(t)=\pa_tK(t)\tilde{w}^j=\pa_tKJK\tilde{w}^j=\pa_tKJ\phi^j(t),
      \eeq
      which allows us to express $\pa_t\phi^j$ in terms of $\phi^j$. For any integer $m\ge1$, we define the finite-dimensional space $\mathscr{H}^1_m(t):=$ span $\{\phi^1(t), \ldots, \phi^m(t)\}\subset H^2(\Om)\cap\mathscr{H}^1(t)$ and we define $\mathscr{P}^m_t: H^2(\Om)\to \mathscr{H}^1_m(t)$ for $H^2(\Om)$ orthogonal projection onto $\mathscr{H}^1_m(t)$. Clearly, if $\theta\in H^2(\Om)\cap\mathscr{H}^1(t)$, $\mathscr{P}^m_t\theta\to\theta$ as $m\to\infty$.

      For each $m\ge1$, we define an approximate solution
      \[
      \theta^m=d^m_j(t)\phi^j(t), \quad \text{with}\quad d^m_j(t): [0, T] \to \mathbb{R} \quad\text{for}\quad j=1, \ldots, m,
      \]
      where as usual we use the Einstein convention of summation of the repeated index $j$.
      We want to choose $d^m_j$ such that
      \beq\label{equ:thetam}
      \left(\pa_t\theta^m, \phi\right)_{\mathscr{H}^0}+\left(\theta^m, \phi\right)_{\mathscr{H}^1}+\left(\theta^m\left|\mathscr{N}\right|, \phi\right)_{H^0(\Sigma)}=\left(F^3,
      \phi\right)_{\mathscr{H}^0}+\left(F^5, \phi\right)_{H^0(\Sigma)},
      \eeq
      with the initial data $\theta^m(0)=\mathscr{P}^m_t\theta_0\in \mathscr{H}^1_m(0)$ for each $\phi\in\mathscr{H}^1_m(t)$.
      And \eqref{equ:thetam} is equivalent to the system of ODEs for $d^m_j$:
      \ben\label{equ:ode}
      \begin{aligned}
      \dot{d}^m_j\left(\phi^j, \phi^k\right)_{\mathscr{H}^0}+d^m_j\left(\left(\pa_tKJ\phi^j,\phi^k\right)_{\mathscr{H}^0}+\left(\phi^j, \phi^k\right)_{\mathscr{H}^1}+\left(\phi^j\left|\mathscr{N}\right|,\phi^k\right)_{H^0(\Sigma)}\right)\\
      =\left(F^3,
      \phi^k\right)_{\mathscr{H}^0}+\left(F^5, \phi^k\right)_{H^0(\Sigma)}
      \end{aligned}
      \een
      for $j, k=1, \ldots, m$. The $m\times m$ matrix with $j, k$ entry $\left(\phi^j, \phi^k\right)_{\mathscr{H}^0}$ is invertible, the coefficients of the linear system \eqref{equ:ode} are $C^1([0, T])$, and the forcing terms are $C^0([0, T])$, so the usual well-posedness of ODEs guarantees that the existence of a unique solution $d^m_j\in C^1([0, T])$ to \eqref{equ:ode} that satisfies the initial data. This provides the desired solution, $\theta^m$, to \eqref{equ:thetam}. Since $F^3$, $F^5$ satisfy \eqref{cond:force}, equation \eqref{equ:ode} may be differentiated in time to see that $d_j^m\in C^{1,1}([0, T])$, which means $d_j^m$ is twice differentiable almost everywhere in $[0, T]$.

  Step 2. The energy estimates for $\theta^m$. Since $\theta^m(t)\in \mathscr{H}^1_m(t)$, we take $\phi=\theta^m$ as a test function in \eqref{equ:thetam}, using the Poincar\'e-type inequalities in Lemma A.14 of \cite{GT1} and usual trace theory, we have
      \begin{align*}
      \pa_t\f12\|\theta^m\|_{\mathscr{H}^0}^2+\|\theta^m\|_{\mathscr{H}^1}^2\lesssim (\|F^3\|_{\mathscr{H}^0}+\|F^5\|_{H^{1/2}(\Sigma)})\|\theta^m\|_{\mathscr{H}^1}-\f12\int_{\Om}|\theta^m|^2\pa_tJ.
      \end{align*}
      Then, applying Cauchy's inequality, we may derive that
      \begin{align*}
        \pa_t\f12\|\theta^m\|_{\mathscr{H}^0}^2+\f14\|\theta^m\|_{\mathscr{H}^1}^2\lesssim\|F^3\|_{\mathscr{H}^0}^2+\|F^5\|_{H^{1/2}(\Sigma)}^2+C_0(\eta)\f12\|\theta^m\|_{\mathscr{H}^0}^2
      \end{align*}
      with $C_0(\eta):=1+\sup_{0\le t\le T}\|\pa_tJK\|_{L^\infty}$. Using the Lemma 2.9 in \cite{LW}, we may have
      \ben\label{equ:initial thetam}
      \begin{aligned}
      \|\theta^m(0)\|_{\mathscr{H}^0}&\le P(\|\eta_0\|_{H^{5/2}})\|\theta^m(0)\|_{H^0}\le P(\|\eta_0\|_{H^{5/2}})\|\theta^m(0)\|_{H^2}\\
      &=P(\|\eta_0\|_{H^{5/2}})\|\mathscr{P}^m_0\theta_0\|_{H^2}\le P(\|\eta_0\|_{H^{5/2}})\|\theta_0\|_{H^2}.
      \end{aligned}
      \een
      Now, we can utilize Gronwall's lemma to deduce energy estimates for $\theta^m$:
      \ben\label{equ:est thetam}
      \begin{aligned}
      \sup_{0\le t\le T}\|\theta^m\|_{\mathscr{H}^0}^2&+\|\theta^m\|_{\mathscr{H}^1_T}^2\\
      &\lesssim P(\|\eta_0\|_{H^{5/2}})\exp(C_0(\eta)T)(\|\theta_0\|_{H^2}^2+\|F^3\|_{\mathscr{H}^0_T}^2+\|F^5\|_{L^2H^{1/2}(\Sigma)}^2).
      \end{aligned}
      \een

   Step 3. Estimates for $\pa_t\theta^m(0)$. If $\theta\in H^2(\Om)\cap\mathscr{H}^1(t)$, $\phi\in\mathscr{H}^1$, the integration by parts reveals that
       \beq\label{equ:theta1}
       \left(\theta, \phi\right)_{\mathscr{H}^1}=\int_{\Om}-\Delta_{\mathscr{A}}\theta\phi J+\int_{\Sigma}(\nabla_{\mathscr{A}}\theta\cdot\mathscr{N})\phi=\left(-\Delta_{\mathscr{A}}\theta,\phi\right)_{\mathscr{H}^0}
       +\left(\nabla_{\mathscr{A}}\theta\cdot\mathscr{N},\phi\right)_{H^0(\Sigma)}
       \eeq
       Evaluating \eqref{equ:thetam} at $t=0$ and employing \eqref{equ:theta1}, we have that
       \beq\label{equ:thetam0}
       \left(\pa_t\theta^m(0), \phi\right)_{\mathscr{H}^0}=\left(\Delta_{\mathscr{A}_0}\theta^m(0)+F^3(0), \phi\right)_{\mathscr{H}^0},
       \eeq
       for all $\phi\in\mathscr{H}^1_m(t)$.

       By virtue of \eqref{equ:dt phi}, we have that
       \beq\label{equ:test theta}
       \pa_t\theta^{m}-\pa_tK(t)J(t)\theta^m(t)=\dot{d}^m_j(t)\phi^j(t)\in\mathscr{H}^1_m(t),
       \eeq
       so that $\phi=\pa_t\theta^m(0)-\pa_tK(0)J(0)\theta^m(0)\in\mathscr{H}^1_m(0)$ is a choice for the test function in \eqref{equ:thetam0}. So using this test function in \eqref{equ:thetam0}, we have
       \ben\label{equ:thetam1}
       \begin{aligned}
         \|\pa_t\theta^m(0)\|_{\mathscr{H}^0}^2&\le \|\pa_tK(0)J(0)\theta^m(0)\|_{\mathscr{H}^0}\|\pa_t\theta^m(0)\|_{\mathscr{H}^0}\\
         &+\|\pa_t\theta^m(0)-\pa_tK(0)J(0)\theta^m(0)\|_{\mathscr{H}^0}\|\Delta_{\mathscr{A}_0}\theta^m(0)+F^3(0)\|_{\mathscr{H}^0}.
       \end{aligned}
       \een
       Then after using  \eqref{equ:initial thetam} and Cauchy's inequality for the right--hand side of \eqref{equ:thetam1}, we have the bound
       \beq\label{equ:est thetam0}
       \|\pa_t\theta^m(0)\|_{\mathscr{H}^0}^2\lesssim C_1(\eta)\left(\|\theta_0\|_{H^2}^2+\|F^3(0)\|_{\mathscr{H}^0}^2\right)
       \eeq
       with $C_1(\eta)=P(\|\eta_0\|_{H^{5/2}})\left(1+\|\pa_tK(0)J(0)\|_{L^\infty}^2+\|\mathscr{A}_0\|_{C^1}^2\right)$.

   Step 4. Energy estimates for $\pa_t\theta^m$. Now, suppose that $\phi(t)=c_j^m(t)\phi^j$ for $c_j^m\in C^{0,1}([0, T])$, $j=1, \ldots, m$; it is proved as in \eqref{equ:test theta}, that $\pa_t\phi-\pa_tK(t)J(t)\phi\in \mathscr{H}^1_m(t)$ as well. Then in \eqref{equ:thetam}, using this $\phi$, and temporally differentiating the result equation, and then subtracting from the result equation \eqref{equ:thetam} with test function $\pa_t\phi-\pa_tK(t)J(t)\phi$, we find that
       \ben\label{equ:dt thetam1}
       \begin{aligned}
         &\left<\pa_t^2\theta^m, \phi\right>_{(\mathscr{H}^1)^\ast}+\left(\pa_t\theta^m,\phi\right)_{\mathscr{H}^1}+\left(\pa_t\theta^m\left|\mathscr{N}\right|, \phi\right)_{H^0(\Sigma)}\\
         &=\left<\pa_t(F^3+F^5),\phi\right>_{(\mathscr{H}^1)^\ast}+\left(F^3,(\pa_tKJ+\pa_tJK)\phi\right)_{\mathscr{H}^0}+\left(F^5,\pa_tKJ\phi\right)_{H^0(\Sigma)}\\
         &\quad-\left(\pa_t\theta^m, (\pa_tKJ+\pa_tJK)\phi\right)_{\mathscr{H}^0}-\left(\theta^m, \pa_tKJ\phi\right)_{\mathscr{H}^1}-\left(\theta^m, \pa_tKJ\phi\right)_{H^0(\Sigma)}\\
         &\quad-\int_{\Om}\left(\pa_tJK\nabla_{\mathscr{A}}\theta^m\cdot\nabla_{\mathscr{A}}\phi+\nabla_{\pa_t\mathscr{A}}\theta^m\cdot\nabla_{\mathscr{A}}\phi+\nabla_{\mathscr{A}}\theta^m\cdot\nabla_{\pa_t\mathscr{A}}\phi\right)J.
       \end{aligned}
       \een
       According to \eqref{equ:test theta} and the fact that $d_j^m(t)$ is twice differentiable almost everwhere as we have pointed in the first step, we use $\phi=\pa_t\theta^m-\pa_tKJ\theta^m$ as a test function in \eqref{equ:dt thetam1}. Utilizing Cauchy's inequality, trace theory and the Remark $2.3$ in \cite{GT1}, we have that
       \ben\label{equ:dt thetam2}
       \begin{aligned}
       &\pa_t\left(\f12\|\pa_t\theta^m\|_{\mathscr{H}^0}^2-\left(\pa_t\theta^m, \pa_tKJ\theta^m\right)_{\mathscr{H}^0}\right)+\f14\|\pa_t\theta^m\|_{\mathscr{H}^1}^2\\
       &\le C_0(\eta)\left(\f12\|\theta^m\|_{\mathscr{H}^0}^2-\left(\pa_t\theta^m, \pa_tKJ\theta^m\right)_{\mathscr{H}^0}\right)+C_2(\eta)\|\pa_t\theta^m\|_{\mathscr{H}^1}^2\\
       &\quad+C\left(\|F^3\|_{\mathscr{H}^0}^2+\|F^5\|_{H^{1/2}(\Sigma)}^2\right)+C\|\pa_t(F^3+F^5)\|_{(\mathscr{H}^1)^\ast}
       \end{aligned}
       \een
       for $C_2(\eta)$ is defined as
       \beno
       C_2(\eta):&=&\sup_{0\le t\le T}\big[1+\|\pa_t(\pa_tKJ)\|_{L^\infty}^2+\|\pa_tKJ\|_{C^1}^2+\|\pa_t\mathscr{A}\|_{L^\infty}^2\\
       &&\quad+(1+\|\mathscr{A}\|_{L^\infty}^2)(1+\|\pa_tJ K\|_{L^\infty}^2)\big](1+\|\pa_tKJ\|_{C^1}^2).
       \eeno
       Then according to Cauchy's inequality and Gronwall's lemma, \eqref{equ:dt thetam2} implies that
       \ben\label{equ:dt thetam3}
       \begin{aligned}
         &\sup_{0\le t\le T}(\|\pa_t\theta^m\|_{\mathscr{H}^0}^2+\|\pa_t\theta^m\|_{\mathscr{H}^1_T}^2\\
         &\lesssim \exp(C_0(\eta)T)\Big(\|\pa_t\theta^m(0)\|_{\mathscr{H}^0}^2 +C_1(\eta)\|\theta^m(0)\|_{\mathscr{H}^0}^2+\|F^3\|_{\mathscr{H}^0_T}^2\\
         &\quad+\|F^5\|_{L^2H^{1/2}}^2+\|\pa_t(F^3+F^5)\|_{(\mathscr{H}^1_T)^\ast}^2\Big)\\
         &\quad+C_2(\eta)\left(\sup_{0\le t\le T}\|\theta^m\|_{\mathscr{H}^0}^2+\int_0^T\exp(C_0(\eta)(T-s))\|\theta^m(s)\|_{\mathscr{H}^1}^2\,\mathrm ds\right).
       \end{aligned}
       \een
       Now, the energy estimates for $\pa_t\theta^m$ is deduced by combining \eqref{equ:dt thetam3} with the estimates \eqref{equ:initial thetam}, \eqref{equ:est thetam} and \eqref{equ:est thetam0},
       \ben\label{equ:dt thetam4}
       \begin{aligned}
         &\sup_{0\le t\le T}\|\pa_t\theta^m\|_{\mathscr{H}^0}^2+\|\pa_t\theta^m\|_{\mathscr{H}^1_T}^2\\
         &\lesssim \left(C_1(\eta)+C_2(\eta)\right)\exp(C_0(\eta)T)\left(\|\theta^m(0)\|_{\mathscr{H}^0}^2+\|F^3(0)\|_{\mathscr{H}^0}^2\right)\\
         &\quad+\exp(C_0(\eta)T)\left[C_2(\eta)\left(\|F^3\|_{\mathscr{H}^0_T}^2+\|F^5\|_{L^2H^{1/2}}^2\right)+\|\pa_t(F^3+F^5)\|_{(\mathscr{H}^1_T)^\ast}^2\right].
       \end{aligned}
       \een

   Step 5. Improved estimates for $\theta^m$. Using the $\phi=\pa_t\theta^m-\pa_tKJ\theta^m\in \mathscr{H}^1_m(t)$ as a test function in \eqref{equ:thetam}, we can improve the energy estimates for $\theta^m$.
       \ben\label{equ:improve thetam}
       \begin{aligned}
         &\pa_t\f12\left(\|\theta^m\|_{\mathscr{H}^1}^2+\|\theta^m\|_{H^0(\Sigma)}^2\right)+\|\pa_t\theta^m\|_{\mathscr{H}^0}^2\\
         &=\left(\pa_t\theta^m,\pa_tKJ\theta^m\right)_{\mathscr{H}^0}+\left(\theta^m,\pa_tKJ\theta^m\right)_{\mathscr{H}^1}+\left(F^3,\pa_t\theta^m-\pa_tKJ\theta^m\right)_{\mathscr{H}^0}\\
         &\quad+\left(F^5,\pa_t\theta^m-\pa_tKJ\theta^m\right)_{H^0(\Sigma)}+\int_{\Om}\left(\nabla_{\mathscr{A}}\theta^m\cdot\nabla_{\pa_t\mathscr{A}}\theta^m+\pa_tJK\f{|\nabla_{\mathscr{A}}\theta^m|^2}{2}J\right).
       \end{aligned}
       \een
       Since we have already controlled $\|\theta^m\|_{\mathscr{H}^1_T}^2$ and $\|\pa_t\theta^m\|_{\mathscr{H}^1_T}^2$, integrating \eqref{equ:improve thetam} in time implies that
       \ben\label{equ:improve thetam1}
       \begin{aligned}
         &\sup_{0\le t\le T}\|\theta^m\|_{\mathscr{H}^1}^2+\|\pa_t\theta^m\|_{\mathscr{H}^0_T}^2\\
         &\lesssim P(\|\eta_0\|_{H^{5/2}}) \left(C_1(\eta)+C_2(\eta)\right)\exp(C_0(\eta)T)\left(\|\theta_0\|_{H^0}^2+\|F^3(0)\|_{\mathscr{H}^0}^2\right)\\
         &\quad+P(\|\eta_0\|_{H^{5/2}})\exp(C_0(\eta)T)\Big[C_2(\eta)\left(\|F^3\|_{\mathscr{H}^0_T}^2+\|F^5\|_{L^2H^{1/2}}^2\right)\\
         &\quad+\|\pa_t(F^3+F^5)\|_{(\mathscr{H}^1_T)^\ast}^2\Big].
       \end{aligned}
       \een

   Step 6. Uniform bounds for \eqref{equ:dt thetam4} and \eqref{equ:improve thetam1}. Now, we seek to estimate the constants $C_i(\eta)$, $i=0, 1, 2$ in terms of the quantity $\mathscr{K}(\eta)$. A direct computation combining with the Lemma A.10 in \cite{GT1} reveal that
       \beq
       C_0(\eta)+C_1(\eta)+C_2(\eta)\le C(1+\mathscr{K}(\eta)),
       \eeq
       For a constant $C$ independent of $\eta$.

   Step 7. Passing to the limit. According to the energy estimates \eqref{equ:dt thetam4} and \eqref{equ:improve thetam1} and Lemma \ref{lem:theta H0 H1}, we know that the sequence $\{\theta^m\}$ is uniformly bounded in $L^\infty H^1$ and $\{\pa_t\theta^m\}$ is uniformly bounded in $L^\infty H^0\cap L^2H^1$. Then, up to extracting a subsequence, we know that
       \[
       \theta^m\stackrel{\ast}\rightharpoonup \theta \thinspace \text{weakly-}\ast \thinspace\text{in}\thinspace L^\infty H^1,\thinspace\pa_t\theta^m\stackrel{\ast}\rightharpoonup\pa_t\theta\thinspace\text{in}\thinspace L^\infty H^0,\thinspace \pa_t\theta^m\rightharpoonup\pa_t\theta\thinspace\text{weakly in}\thinspace L^2H^1,
       \]
       as $m\to\infty$. By lower semicontinuity, the energy estimates reveal that
       \[
       \|\theta\|_{L^\infty H^1}^2+\|\pa_t\theta\|_{L^\infty H^0}^2+\|\pa_t\theta\|_{L^2H^1}^2
       \]
       is bounded from above by the right-hand side of \eqref{inequ:est strong solution}.

       According these convergence results, we can integrate \eqref{equ:dt thetam1} termporally from $0$ to $T$ and let $m\to\infty$ to deduce that $\pa_t^2\theta^m\rightharpoonup\pa_t^2\theta$ weakly in $(\mathscr{H}^1_T)^\ast$, with an action of $\pa_t^2\theta$ on an element $\phi\in\mathscr{H}^1_T$ defined by replacing $\theta^m$ with $\theta$ everywhere in \eqref{equ:dt thetam1}. From passing to the limit in \eqref{equ:dt thetam1}, it is straightforward to show that $\|\pa_t^2\theta\|_{(\mathscr{H}^1_T)^\ast}^2$ is bounded from above by the right-hand side of \eqref{inequ:est strong solution}. This bound shows that $\pa_t\theta\in C^0L^2$.

   Step 8. In the limit, \eqref{equ:thetam} implies that for almost every $t$,
    \beq\label{equ:limit theta}
    \left(\pa_t\theta,\phi\right)_{\mathscr{H}^0}+\left(\theta,\phi\right)_{\mathscr{H}^1}+\left(\theta\left|\mathscr{N}\right|,\phi\right)_{H^0(\Sigma)}=\left(F^3,\phi\right)_{\mathscr{H}^0}+\left(F^5,\phi\right)_{H^0(\Sigma)}\quad\text{for every}\thinspace\phi\in\mathscr{H}^1.
    \eeq
    For almost every $t\in [0, T]$, $\theta(t)$ is the unique weak solution to the elliptic problem \eqref{equ:SBC} in the sense of \eqref{equ:weak theta}, with $F^3$ replaced by $F^3(t)-\pa_t\theta(t)$ and $F^5$ replaced by $F^5(t)$. Since $F^3(t)-\pa_t\theta(t)\in H^0(\Om)$ and $F^5(t)\in H^{1/2}(\Sigma)$, Lemma \ref{lem:S lower regularity} shows that this elliptic problem admits a unique strong solution, which must coincide with the weak solution. Then applying Proposition \ref{prop:high regulatrity}, we have the bound
       \beq\label{est:bound}
       \|\theta(t)\|_{H^r}^2\lesssim C(\eta_0) \left(\|\pa_t\theta(t)\|_{\mathscr{H}^{r-2}}^2+\|F^3(t)\|_{\mathscr{H}^{r-2}}^2+\|F^5(t)\|_{H^{r-3/2}(\Sigma)}^2\right)
       \eeq
       when $r=2,3$. When $r=2$, we take the superemum of \eqref{est:bound} over $t\in [0, T]$, and when $r=3$, we integrate over $[0, T]$; the resulting inequalities imply that $\theta\in L^\infty H^2\cap L^2H^3$ with estimates as in \eqref{inequ:est strong solution}.

       Then for the linear Navier--Stokes equations, the process is exactly the same as \cite{GT1}. Then we know that $(u, p, \theta)$ is a strong solution of \eqref{equ:linear BC} with the estimates as in \eqref{inequ:est strong solution}.

    Step 9. The weak solution satisfied by $\pa_t\theta$ and $D_tu$. We may integrate \eqref{equ:dt thetam1} in time from $0$ to $T$ and pass the limit $m\to\infty$. For any $\phi\in\mathscr{H}^1$, we have $\pa_tKJ\phi\in\mathscr{H}^1$, so that we may subsititute $\pa_tKJ\phi$ for $\phi$ in \eqref{equ:limit theta}; this yields
        \ben
        \begin{aligned}
          &\left<\pa_t^2\theta,\phi\right>_{(\mathscr{H}^1_T)^\ast}+\left(\pa_t\theta,\phi\right)_{\mathscr{H}^1_T}+\left(\pa_t\theta\left|\mathscr{N}\right|,\phi\right)_{L^2H^0(\Sigma)}\\
          &=\left<\pa_t(F^3+F^5)\right>_{(\mathscr{H}^1_T)^\ast}+\left(\pa_tJKF^3,\phi\right)_{\mathscr{H}^0_T}-\left(\pa_tJK\pa_t\theta,\phi\right)_{\mathscr{H}^0_T}\\
          &\quad-\int_0^T\int_{\Om}\left(\pa_tJK\nabla_{\mathscr{A}}\theta\cdot\nabla_{\mathscr{A}}\phi+\nabla_{\pa_t\mathscr{A}}\theta\cdot\nabla_{\mathscr{A}}\phi+\nabla_{\mathscr{A}}\theta\cdot\nabla_{\pa_t\mathscr{A}}\phi\right)J
        \end{aligned}
        \een
        for all $\phi\in\mathscr{H}^1_T$. This is exactly the \eqref{equ:weak pat theta}. To justify that \eqref{equ:weak pat theta} implies \eqref{equ:pat theta}, we may integrate by parts for the equality
        \ben
        \begin{aligned}
         &-\int_0^T\int_{\Om}\left(\pa_tJK\nabla_{\mathscr{A}}\theta\cdot\nabla_{\mathscr{A}}\phi+\nabla_{\pa_t\mathscr{A}}\theta\cdot\nabla_{\mathscr{A}}\phi+\nabla_{\mathscr{A}}\theta\cdot\nabla_{\pa_t\mathscr{A}}\phi\right)J\\
         &=-\int_0^T\int_\Om\left(-R\nabla_{\mathscr{A}}u+\nabla_{\pa_t\mathscr{A}}u\right)\cdot\nabla_{\mathscr{A}}\phi J\\
         &=\left(\dive_{\mathscr{A}}(-R\nabla_{\mathscr{A}}u+\nabla_{\pa_t\mathscr{A}}u),\phi\right)_{\mathscr{H}^0_T}-\left<\nabla_{\mathscr{A}}u\cdot\pa_t\mathscr{N}+\nabla_{\pa_t\mathscr{A}}u\cdot\mathscr{N},\phi\right>_{L^2H^{-1/2}}.
        \end{aligned}
        \een
        We then may deduce from \eqref{equ:weak pat theta} that $\pa_t\theta$ is a weak solution of \eqref{equ:pat theta} in the sense of \eqref{equ:lpws} with $\pa_t\theta(0)\in\mathscr{H}^0(0)$. Then we may appeal to  the computation in \cite{GT1} to deduce that $p(0)$ satisfies the equation \eqref{equ:p0} and $D_tu$ is a weak solution of \eqref{equ:Dt u} in the sense of \eqref{equ:lpws} with $D_tu(0)\in\mathscr{Y}(0)$.
\end{proof}

\subsection{Higher regularity}
In order to state our higher regularity results for \eqref{equ:linear BC}, we need to construct the initial data and compatible conditions. First, we define the vector or scalar fields $\mathfrak{E}^{01}$, $\mathfrak{E}^{02}$, $\mathfrak{E}^1$, $\mathfrak{E}^3$ in $\Om$ and $\mathfrak{E}^4$, $\mathfrak{E}^5$ on $\Sigma$ by
\ben
\begin{aligned}
  \mathfrak{E}^{01}(G^1, v, q)&=\Delta_{\mathscr{A}}v-\nabla_{\mathscr{A}}q+G^1-Rv,\\
  \mathfrak{E}^{02}(G^3,\Theta)&=\Delta_{\mathscr{A}}\Theta+G^3,\\
  \mathfrak{E}^1(v,q)&=-(R+\pa_tJK)\Delta_{\mathscr{A}}v-\pa_tRv+(\pa_tJK+R+R^\top)\nabla_{\mathscr{A}}q\\
  &\quad+\dive_{\mathscr{A}}(\mathbb{D}_{\mathscr{A}}(Rv)-R\mathbb{D}_{\mathscr{A}}v+\mathbb{D}_{\pa_t\mathscr{A}}v),\\
  \mathfrak{E}^3(\Theta)&=-\pa_tJK\Delta_{\mathscr{A}}\Theta+\dive_{\mathscr{A}}(-R\nabla_{\mathscr{A}}\Theta+\nabla_{\pa_t\mathscr{A}}\Theta),\\
  \mathfrak{E}^4(v,q)&=\mathbb{D}_{\mathscr{A}}(Rv)\mathscr{N}-(qI-\mathbb{D}_{\mathscr{A}}v)\pa_t\mathscr{N}+\mathbb{D}_{\pa_t\mathscr{A}}v\mathscr{N},\\
  \mathfrak{E}^5(\Theta)&=-\nabla_{\mathscr{A}}\Theta\cdot\pa_t\mathscr{N}-\nabla_{\pa_t\mathscr{A}}\Theta\cdot\mathscr{N}-\Theta\pa_t\left|\mathscr{N}\right|,
\end{aligned}
\een
and we define functions $\mathfrak{f}^1$ in $\Om$, $\mathfrak{f}^2$ on $\Sigma$ and $\mathfrak{f}^3$ on $\Sigma_b$ by
\ben
\begin{aligned}
  \mathfrak{f}^1(G^1, v)&=\dive_{\mathscr{A}}(G^1-Rv),\\
  \mathfrak{f}^2(G^4, v)&=(G^4+\mathbb{D}_{\mathscr{A}}v{\mathscr{N}})\cdot{\mathscr{N}}|{\mathscr{N}}|^{-2},\\
  \mathfrak{f}^3(G^1, v)&=(G^1+\Delta_{\mathscr{A}}v)\cdot\nu.
\end{aligned}
\een

We write $F^{1,0}=F^1+\theta \nabla_{\mathscr{A}}y_3$, $F^{3,0}=F^3$, $F^{4,0}=F^4$ and $F^{5,0}=F^5$. When $F^1$, $F^3$, $F^4$, $F^5$, $u$, $p$, and $\theta$ are regularly enough, we can recursively define
\ben \label{equ:force 1}
\begin{aligned}
  F^{1,j}&:=D_tF^{1,j-1}-\pa_t^{j-1}(\theta \nabla_{\mathscr{A}}y_3)+D_t^{j-1}(\theta \nabla_{\mathscr{A}}y_3)+\mathfrak{E}^1(D_t^{j-1}u, \pa_t^{j-1}p)\\
  &=D_t^jF^1-\left(\pa_t^{j-1}(\theta \nabla_{\mathscr{A}}y_3)-D_t^{j-1}(\theta \nabla_{\mathscr{A}}y_3)\right)+\sum_{\ell=0}^{j-1}D_t^\ell\mathfrak{E}^1(D_t^{j-\ell-1}u, \pa_t^{j-\ell-1}p),\\
  F^{3,j}&:=\pa_tF^{3,j-1}+\mathfrak{E}^3(\pa_t^{j-1}\theta)=\pa_t^jF^3+\sum_{\ell=0}^{j-1}\pa_t^\ell\mathfrak{E}^3(\pa_t^{j-\ell-1}\theta),
\end{aligned}
\een
in $\Om$ and
\ben\label{equ:force 2}
\begin{aligned}
  F^{4,j}&:=\pa_tF^{4,j-1}+\mathfrak{E}^4(D_t^{j-1}u, \pa_t^{j-1}p)=\pa_t^jF^4+\sum_{\ell=0}^{j-1}\pa_t^\ell\mathfrak{E}^4(D_t^{j-\ell-1}u, \pa_t^{j-\ell-1}p),\\
  F^{5,j}&:=\pa_tF^{5,j-1}+\mathfrak{E}^5(\pa_t^{j-1}\theta)=\pa_t^jF^5+\sum_{\ell=0}^{j-1}\pa_t^\ell\mathfrak{E}^5(\pa_t^{j-\ell-1}\theta)
\end{aligned}
\een
on $\Sigma$, for $j=1, \ldots, N$.

Now, we define the sums of norms with $F^1$, $F^3$, $F^4$ and $F^5$.
\ben \label{def:force F F0}
\begin{aligned}
  \mathfrak{F}(F^1,F^3,F^4,F^5)&:=\sum_{j=0}^{N-1}\left(\|\pa_t^jF^1\|_{L^2H^{2N-2j-1}}+\|\pa_t^jF^3\|_{L^2H^{2N-2j-1}}\right)\\
  &\quad+\|\pa_t^{N}F^1\|_{L^2({}_0H^1(\Om))^\ast}+\|\pa_t^{N}F^3\|_{L^2({}_0H^1(\Om))^\ast}\\
  &\quad+\sum_{j=0}^N\left(\|\pa_t^jF^4\|_{L^2H^{2N-2j-1/2}}+\|\pa_t^jF^5\|_{L^2H^{2N-2j-1/2}}\right)\\
  &\quad+\sum_{j=0}^{N-1}\left(\|\pa_t^jF^1\|_{L^\infty H^{2N-2j-2}}+\|\pa_t^jF^3\|_{L^\infty H^{2N-2j-2}}\right)\\
  &\quad+\sum_{j=0}^{N-1}\left(\|\pa_t^jF^4\|_{L^\infty H^{2N-2j-3/2}}+\|\pa_t^jF^5\|_{L^\infty H^{2N-2j-3/2}}\right),\\
  \mathfrak{F}_0(F^1,F^3,F^4,F^5)&:=\sum_{j=0}^{N-1}\left(\|\pa_t^jF^1(0)\|_{H^{2N-2j-2}}+\|\pa_t^jF^3(0)\|_{H^{2N-2j-2}}\right)\\
  &\quad+\sum_{j=0}^{N-1}\left(\|\pa_t^jF^4(0)\|_{H^{2N-2j-3/2}}+\|\pa_t^jF^5(0)\|_{H^{2N-2j-3/2}}\right).
\end{aligned}
\een
For simplicity, we will write $\mathfrak{F}$ for $\mathfrak{F}(F^1,F^3,F^4,F^5)$ and $\mathfrak{F}_0$ for $\mathfrak{F}_0(F^1,F^3,F^4,F^5)$ throughout the rest of this paper. From the Lemma A.4 and Lemma 2.4 of \cite{GT1}, we know that if $\mathfrak{F}<\infty$, then
\begin{align*}
  &\pa_t^jF^1\in C^0([0,T];H^{2N-2j-2}(\Om)),\quad \pa_t^jF^3\in C^0([0,T];H^{2N-2j-2}(\Om)),\\
  &\pa_t^jF^4\in C^0([0,T];H^{2N-2j-3/2}(\Sigma)),\quad\text{and}\quad \pa_t^jF^5\in C^0([0,T];H^{2N-2j-3/2}(\Sigma))
\end{align*}
for $j=0, \ldots, N-1$. For $\eta$, we define
\ben \label{def:norm eta}
\begin{aligned}
  \mathfrak{D}(\eta)&:=\sum_{j=2}^{N+1}\|\pa_t^j\eta\|_{L^2H^{2N-2j+5/2}}^2,\\
  \mathfrak{E}(\eta)&:=\|\eta\|_{L^\infty H^{2N+1/2}(\Sigma)}^2+\sum_{j=1}^{N}\|\pa_t^j\eta\|_{L^\infty H^{2N-2j+3/2}(\Sigma)}^2,\\
  \mathfrak{K}(\eta)&:=\mathfrak{D}(\eta)+\mathfrak{E}(\eta),\\
  \mathfrak{E}_0(\eta)&:=\|\eta_0\|_{H^{2N+1/2}(\Sigma)}^2+\sum_{j=1}^{N}\|\pa_t^j\eta(0)\|_{H^{2N-2j+3/2}(\Sigma)}^2.
\end{aligned}
\een

These following lemmas are similar to Lemma 4.5, 4.6, 4.7 in \cite{GT1} as well as the idea of proof, so we omit these details here.
\begin{lemma}\label{lem:pa tv Dt v}
  If $k=0,\ldots, 2N-1$ and $v$, $\Theta$ are sufficiently regular, then
  \beq\label{est:pat v Dt v l2}
  \|\pa_tv-D_tv\|_{L^2H^k}^2\lesssim P(\mathfrak{K}(\eta))\|v\|_{L^2H^k}^2,
  \eeq
  \beq
  \|\pa_t(\Theta \nabla_{\mathscr{A}}y_3)-D_t(\Theta \nabla_{\mathscr{A}}y_3)\|_{L^2H^k}^2\lesssim P(\mathfrak{K}(\eta))\|\Theta\|_{L^2H^k}^2,
  \eeq
  and if $k=0,\ldots, 2N-2$, then
  \beq\label{est:pat v Dt v linfty}
  \|\pa_tv-D_tv\|_{L^\infty H^k}^2\lesssim P(\mathfrak{K}(\eta))\|v\|_{L^\infty H^k}^2,
  \eeq
  \beq
  \|\pa_t(\Theta \nabla_{\mathscr{A}}y_3)-D_t(\Theta \nabla_{\mathscr{A}}y_3)\|_{L^\infty H^k}^2\lesssim P(\mathfrak{K}(\eta))\|\Theta\|_{L^\infty H^k}^2.
  \eeq
  If $m=1, \ldots, N-1$, $j=1, \ldots, m$, and $v$, $\Theta$ are sufficiently regular, then
  \beq
  \|\pa_t^jv-D_t^jv\|_{L^2H^{2m-2j+3}}^2\lesssim P(\mathfrak{K}(\eta))\sum_{\ell=0}^{j-1}\left(\|\pa_t^\ell v\|_{L^2H^{2m-2j+3}}^2+\|\pa_t^\ell v\|_{L^\infty H^{2m-2j+2}}^2\right),
  \eeq
  \beq\label{est:pa t Dt v j}
  \|\pa_t^jv-D_t^jv\|_{L^\infty H^{2m-2j+2}}^2\lesssim P(\mathfrak{K}(\eta))\sum_{\ell=0}^{j-1}\|\pa_t^\ell v\|_{L^\infty H^{2m-2j+2}}^2,
  \eeq
  \beq
  \|\pa_t^j(\Theta \nabla_{\mathscr{A}}y_3)-D_t^j(\Theta \nabla_{\mathscr{A}}y_3)\|_{L^2H^{2m-2j+2}}^2\lesssim P(\mathfrak{K}(\eta))\sum_{\ell=0}^{j-1}\left(\|\pa_t^\ell \Theta\|_{L^2H^{2m-2j+3}}^2+\|\pa_t^\ell \Theta\|_{L^\infty H^{2m-2j+2}}^2\right),
  \eeq
  \beq
  \|\pa_t^j(\Theta \nabla_{\mathscr{A}}y_3)-D_t^j(\Theta \nabla_{\mathscr{A}}y_3)\|_{L^\infty H^{2m-2j+3}}^2\lesssim P(\mathfrak{K}(\eta))\sum_{\ell=0}^{j-1}\|\pa_t^\ell \Theta\|_{L^\infty H^{2m-2j+2}}^2,
  \eeq
  and
  \ben
  \begin{aligned}
    &\|\pa_tD_t^mv-\pa_t^{m+1}v\|_{L^2H^1}^2+\|\pa_t^2D_t^mv-\pa_t^{m+2}v\|_{(\mathscr{X}_T)^\ast}^2\\
    &\lesssim P(\mathfrak{K}(\eta))\left(\|\pa_t^{m+1}v\|_{(\mathscr{X}_T)^\ast}^2+\sum_{\ell=0}^m\left(\|\pa_t^\ell v\|_{L^2H^1}^2+\|\pa_t^\ell v\|_{L^\infty H^2}^2\right)\right).
  \end{aligned}
  \een
  Also, if $j=0, \ldots, N$ and $v$ is sufficiently regular, then
  \beq\label{equ:initial v j}
  \|\pa_t^jv(0)-D_t^jv(0)\|_{H^{2N-2j}}^2\lesssim P(\mathfrak{E}_0(\eta))\sum_{\ell=0}^{j-1}\|\pa_t^\ell v(0)\|_{H^{2N-2j}}^2,
  \eeq
  and
  if $j=0, \ldots, N-1$ and $\Theta$ is sufficiently regular, then
  \beq\label{equ:initial theta j}
  \|\pa_t^j(\Theta(0)\nabla_{\mathscr{A}_0}y_{3,0})-D_t^j(\Theta(0)\nabla_{\mathscr{A}_0}y_{3,0})\|_{H^{2N-2j-2}}^2\lesssim P(\mathfrak{E}_0(\eta))\sum_{\ell=0}^{j-1}\|\pa_t^\ell \Theta(0)\|_{H^{2N-2j-2}}^2.
  \eeq
  Here all of the $P(\cdot)$ are polynomial, allowed to be changed from line to line.
\end{lemma}

\begin{lemma}\label{lem:force linear}
  For $m=1, \ldots, N-1$ and $j=1, \ldots, m$, the following estimates hold whenever the right--hand sides are finite:
  \ben\label{equ:force l2}
  \begin{aligned}
    &\|F^{1,j}\|_{L^2H^{2m-2j+1}}^2+\|F^{3,j}\|_{L^2H^{2m-2j+1}}^2+\|F^{4,j}\|_{L^2H^{2m-2j+3/2}}^2+\|F^{5,j}\|_{L^2H^{2m-2j+3/2}}^2\\
    &\lesssim P(\mathfrak{K}(\eta))\bigg(\mathfrak{F}+\sum_{\ell=0}^{j-1}\left(\|\pa_t^\ell u\|_{L^2H^{2m-2\ell+3}}^2+\|\pa_t^\ell \theta\|_{L^2H^{2m-2\ell+3}}^2\right)\\
    &\quad+\sum_{\ell=0}^{j-1}\Big(\|\pa_t^\ell u\|_{L^\infty H^{2m-2\ell+2}}^2+\|\pa_t^\ell \theta\|_{L^\infty H^{2m-2\ell+2}}^2+\|\pa_t^\ell p\|_{L^2H^{2m-2\ell+2}}^2\\
    &\quad+\|\pa_t^\ell p\|_{L^\infty H^{2m-2\ell+1}}^2\Big)\bigg),
  \end{aligned}
  \een
  \ben\label{equ:force l infity}
  \begin{aligned}
    &\|F^{1,j}\|_{L^\infty H^{2m-2j}}^2+\|F^{3,j}\|_{L^\infty H^{2m-2j}}^2+\|F^{4,j}\|_{L^\infty H^{2m-2j+1/2}}^2+\|F^{5,j}\|_{L^\infty H^{2m-2j+1/2}}^2\\
    &\lesssim P(\mathfrak{K}(\eta))\bigg(\mathfrak{F}+\sum_{\ell=0}^{j-1}\Big(\|\pa_t^\ell u\|_{L^\infty H^{2m-2\ell+2}}^2+\|\pa_t^\ell \theta\|_{L^\infty H^{2m-2\ell+2}}^2\\
    &\quad+\|\pa_t^\ell p\|_{L^\infty H^{2m-2\ell+1}}^2\Big)\bigg),
  \end{aligned}
  \een
  \ben\label{equ:force dual}
  \begin{aligned}
    &\|\pa_t(F^{1,m}-F^{4,m})\|_{L^2({}_0H^1(\Om))^\ast}^2+\|\pa_t(F^{3,m}+F^{5,m})\|_{L^2({}_0H^1(\Om))^\ast}^2\\
    &\lesssim P(\mathfrak{K}(\eta))\bigg(\mathfrak{F}+\|\pa_t^m u\|_{L^2 H^{2}}^2+\|\pa_t^m \theta\|_{L^2 H^{2}}^2+\|\pa_t^m p\|_{L^2 H^{1}}^2\\
    &\quad+\sum_{\ell=0}^{m-1}\Big(\|\pa_t^\ell u\|_{L^\infty H^{2}}^2+\|\pa_t^\ell u\|_{L^2 H^{2}}^3+\|\pa_t^\ell \theta\|_{L^\infty H^{2}}^2+\|\pa_t^\ell \theta\|_{L^2 H^{2}}^3\\
    &\quad+\|\pa_t^\ell p\|_{L^\infty H^{1}}^2+\|\pa_t^\ell p\|_{L^2 H^{2}}^2\Big)\bigg).
  \end{aligned}
  \een
  Similarly, for $j=1, \ldots, N-1$,
  \ben\label{equ:initial force j}
  \begin{aligned}
    &\|F^{1,j}(0)\|_{H^{2N-2j-2}}^2+\|F^{3,j}(0)\|_{H^{2N-2j-2}}^2+\|F^{4,j}(0)\|_{H^{2N-2j-3/2}}^2+\|F^{5,j}(0)\|_{H^{2N-2j-3/2}}^2\\
    &\lesssim P(\mathfrak{E}_0(\eta))\bigg(\mathfrak{F}_0+\|\pa_t^j \theta(0)\|_{H^{2N-2j}}+\sum_{\ell=0}^{j-1}\big(\|\pa_t^\ell u(0)\|_{H^{2N-2\ell}}\\
    &\quad+\|\pa_t^\ell \theta(0)\|_{H^{2N-2\ell}}+\|\pa_t^\ell p(0)\|_{H^{2N-2\ell-1}}\big)\bigg).
  \end{aligned}
  \een
  Here all of the $P(\cdot)$ are polynomial allowed to be changed from line to line.
\end{lemma}

\begin{lemma}\label{lem:v,q,G}
  Suppose that $v$, $q$, $G^1$, $G^3$ are evaluated at $t=0$ and are sufficiently regular for the right--hand sides of the following estimates to make sense. If $j=0, \ldots, N-1$, then
  \ben\label{equ:initial G1 v q}
  \begin{aligned}
    &\|\mathfrak{E}^{01}(G^1,v,q)\|_{H^{2N-2j-2}}^2\\
    &\lesssim P(\mathfrak{E}_0(\eta))\left(\|v\|_{H^{2N-2j}}^2+\|q\|_{H^{2N-2j-1}}^2+\|G^1\|_{H^{2N-2j-2}}^2\right),
  \end{aligned}
  \een
  \beq\label{equ:e02}
  \|\mathfrak{E}^{02}(G^3,\Theta)\|_{H^{2N-2j-2}}^2\lesssim P(\mathfrak{E}_0(\eta))\left(\|\Theta\|_{H^{2N-2j}}^2+\|G^3\|_{H^{2N-2j-2}}^2\right).
  \eeq
  If $j=0,\ldots,N-2$, then
  \ben\label{equ:initial g1 g4}
  \begin{aligned}
    &\|\mathfrak{f}^1(G^1,v)\|_{H^{2N-2i-3}}^2+\|\mathfrak{f}^2(G^4,v)\|_{H^{2N-2i-3/2}}^2+\|\mathfrak{f}^3(G^1,v)\|_{H^{2N-2i-5/2}}^2\\
    &\lesssim P(\mathfrak{E}_0(\eta))\left(\|G^1\|_{H^{2N-2j-2}}^2+\|G^4\|_{H^{2N-2j-3/2}}^2+\|v\|_{H^{2N-2j}}^2\right).
  \end{aligned}
  \een
  For $j=N-1$, if $\dive_{\mathscr{A}(0)}v(0)=0$ in $\Om$, then
  \beq
  \|\mathfrak{f}^2(G^4,v)\|_{H^{1/2}}^2+\|\mathfrak{f}^3(G^1,v)\|_{H^{-1/2}}^2\lesssim P(\mathfrak{E}_0(\eta))\left(\|G^1\|_{H^{2}}^2+\|G^4\|_{H^{1/2}}^2+\|v\|_{H^{2}}^2\right).
  \eeq
  Here all of the $P(\cdot)$ are polynomial allowed to be changed from line to line.
\end{lemma}
Now we can construct the initial data and compatible conditions. We assume that $u_0\in H^{2N}(\Om)$, $\theta_0\in H^{2N}$, $\eta_0\in H^{2N+1/2}(\Sigma)$. Then we will iteratively construct the initial data $D_t^ju(0)$, $\pa_t^j\theta(0)$ for $j=1,\ldots, N$ and $\pa_t^jp(0)$ for $j=1,\ldots, N-1$. First, we denote $F^{1,0}(0)=F^1(0)\in H^{2N-2}$, $F^{3,0}(0)=F^3(0)\in H^{2N-2}$, $F^{4,0}(0)=F^4(0)\in H^{2N-3/2}$, $F^{5,0}(0)=F^5(0)\in H^{2N-3/2}$ and $D_t^0u(0)=u_0\in H^{2N}$, $\pa_t^0\theta(0)=\theta_0\in H^{2N}$. Suppose now that we have constructed $F^{1,\ell}\in H^{2N-2\ell-2}$, $F^{3,\ell}\in H^{2N-2\ell-2}$, $F^{4,\ell}\in H^{2N-2\ell-3/2}$, $F^{5,\ell}\in H^{2N-2\ell-3/2}$, and $D_t^ju(0)\in H^{2N-2\ell}$, $\pa_t^\ell\theta(0)\in H^{2N-2\ell}$ for $0\le\ell\le j\le N-2$; we will construct $\pa_t^jp(0)\in H^{2N-2j-1}$ as well as $D_t^{j+1}u(0)\in H^{2N-2j-2}$, $\pa_t^{j+1}\theta(0)\in H^{2N-2j-2}$, $F^{1,j+1}(0)\in H^{2N-2j-4}$, $F^{3,j+1}(0)\in H^{2N-2j-4}$, $F^{4,j+1}(0)\in H^{2N-2j-7/2}$ and $F^{5,j+1}(0)\in H^{2N-2j-7/2}$ as follows.

By virtue of estimate, we know that
\ben
\begin{aligned}
  f^1&=\mathfrak{f}^1(F^{1,j}(0),D_t^ju(0))\in H^{2N-2j-3},\\
  f^2&=\mathfrak{f}^2(F^{4,j}(0),D_t^ju(0))\in H^{2N-2j-3/2},\\
  f^3&=\mathfrak{f}^3(F^{1,j}(0),D_t^ju(0))\in H^{2N-2j-5/2}
\end{aligned}
\een
This allows us to define $\pa_t^jp(0)$ as the solution to \eqref{equ:poisson}. The choice of $f^1$, $f^2$, $f^3$, implies that $\pa_t^jp(0)\in H^{2N-2j-1}$, according to the Proposition 2.15 of \cite{WL}. Now the estimates \eqref{equ:initial force j}, \eqref{equ:initial v j} and
\eqref{equ:initial G1 v q} allows us to define
\begin{align*}
  D_t^{j+1}u(0)&:=\mathfrak{E}^{01}\left(F^{1,j}(0)+\pa_t^j(\theta(0) \nabla_{\mathscr{A}_0}y_{3,0}), D_t^ju(0), \pa_t^jp(0)\right)\in H^{2N-2j-2},\\
  \pa_t^{j+1}\theta(0)&:=\mathfrak{E}^{02}\left(F^{3,j}(0),\pa_t^j\theta(0)\right)\in H^{2N-2j-2},\\
  F^{1,j+1}(0)&:=D_t^jF^{1,j}(0)-\pa_t^j(\theta(0) \nabla_{\mathscr{A}_0}y_{3,0})+D_t^j(\theta(0) \nabla_{\mathscr{A}_0}y_{3,0})\\
  &\quad+\mathfrak{E}^1\left(D_t^ju(0),\pa_t^jp(0)\right)\in H^{2N-2j-4},\\
  F^{3,j+1}(0)&:=\pa_tF^{3,j}(0)+\mathfrak{E}^3\left(\pa_t^j\theta(0)\right)\in H^{2N-2j-4},\\
  F^{4,j+1}(0)&:=\pa_tF^{4,j}(0)+\mathfrak{E}^4\left(D_t^ju(0),\pa_t^jp(0)\right)\in H^{2N-2j-7/2},\\
  F^{5,j+1}(0)&:=\pa_tF^{5,j}(0)+\mathfrak{E}^5\left(\pa_t^j\theta(0)\right)\in H^{2N-2j-7/2}.
\end{align*}
Then, from the above analysis, we can iteratively construct all of the desired data except for $D_t^Nu(0)$, $\pa_t^{N-1}p(0)$ and $\pa_t^N\theta(0)$.

By construction, the initial data $D_t^ju(0)$, $\pa_t^jp(0)$ and $\pa_t^j\theta(0)$ are determined in terms of $u_0$, $\theta_0$ as well as $\pa_t^\ell F^1(0)$, $\pa_t^\ell F^3(0)$, $\pa_t^\ell F^4(0)$ and $\pa_t^\ell F^5(0)$ for $\ell=0, \ldots, N-1$. In order to use these in Theorem \ref{thm:lower regularity} and to construct $D_t^Nu(0)$, $\pa_t^{N-1}p(0)$ and $\pa_t^N\theta(0)$, we must enforce compatibility conditions for $j=0,\ldots,N-1$. We say that the $j$--th compatibility condition is satisfied if
\ben\label{cond:compatibility j}
\left\{
\begin{aligned}
  &D_t^ju(0)\in \mathscr{X}(0)\cap H^2(\Om),\\
  &\Pi_0\left(F^{4,j}(0)+\mathbb{D}_{\mathscr{A}_0}D_t^ju(0)\mathscr{N}_0\right)=0.
\end{aligned}
\right.
\een
The construction of $D_t^ju(0)$ and $\pa_t^jp(0)$ ensures that $D_t^ju(0)\in H^2(\Om)$ and $\dive_{\mathscr{A}_0}(D_t^ju(0))=0$.

In the following, we define $\pa_t^N\theta(0)\in H^0$, $\pa_t^{N-1}p(0)\in H^1$ and $D_t^N u(0)\in H^0$. First, we can define
\[
\pa_t^N\theta(0)=\mathfrak{E}^{02}(F^{3,N-1}(0),\pa_t^{N-1}\theta(0))\in H^0(\Om),
\]
employing \eqref{equ:e02} for the inclusion in $H^0$. Then using the same analysis in \cite{GT1}, the data $\pa_t^{N-1}p(0)\in H^1$ can be defined as a weak solution to \eqref{equ:poisson}. Then we define
\[
D_t^Nu(0)=\mathfrak{E}^{01}\left(F^{1,N-1}(0)+\pa_t^{N-1}(\theta(0) \nabla_{\mathscr{A}_0}y_{3,0}), D_t^{N-1}u(0),\pa_t^{N-1}p(0)\right)\in H^0,
\]
employing \eqref{equ:initial G1 v q} and \eqref{equ:initial theta j} for the inclusion in $H^0$. And $D_t^Nu(0)\in\mathscr{Y}(0)$ is guaranteed by the construction of $\pa_t^{N-1}p(0)$. Combining the inclusions above with the bounds \eqref{equ:initial force j}, \eqref{equ:initial g1 g4} ,
\eqref{equ:initial G1 v q} and \eqref{equ:e02} implies that
\ben\label{equ:est initial data j}
\begin{aligned}
  &\sum_{j=0}^N\|D_t^ju(0)\|_{H^{2N-2j}}^2+\sum_{j=0}^{N-1}\|\pa_t^jp(0)\|_{H^{2N-2j-1}}^2+\sum_{j=0}^N\|\pa_t^j\theta(0)\|_{H^{2N-2j}}^2\\
  &\lesssim P(\mathfrak{E}_0(\eta))\left(\|u_0\|_{H^{2N}}^2+\|\theta_0\|_{H^{2N}}^2+\mathfrak{F}_0\right).
\end{aligned}
\een
Before stating the result on higher regularity for solutions to \eqref{equ:linear BC} , we define some quantities:
\ben
\begin{aligned}
  \mathfrak{D}(u,p,\theta)&:=\sum_{j=0}^N\left(\|\pa_t^ju\|_{L^2H^{2N-2j+1}}^2+\|\pa_t^j\theta\|_{L^2H^{2N-2j+1}}^2\right)+\|\pa_t^{N+1}u\|_{(\mathscr{X}_T)^\ast}\\
  &\quad+\|\pa_t^{N+1}\theta\|_{(\mathscr{H}^1_T)^\ast}+\sum_{j=0}^{N-1}\|\pa_t^jp\|_{L^2H^{2N-2j}},\\
  \mathfrak{E}(u,p,\theta)&:=\sum_{j=0}^N\left(\|\pa_t^ju\|_{L^\infty H^{2N-2j}}^2+\|\pa_t^j\theta\|_{L^\infty H^{2N-2j}}^2\right)+\sum_{j=0}^{N-1}\|\pa_t^jp\|_{L^\infty H^{2N-2j-1}},\\
  \mathfrak{K}(u,p,\theta)&:=\mathfrak{D}(u,p,\theta)+\mathfrak{E}(u,p,\theta).
\end{aligned}
\een
\begin{theorem}\label{thm:higher regularity}
  Suppose that $u_0\in H^{2N}(\Om)$, $\theta_0\in H^{2N}(\Om)$, $\eta_0\in H^{2N+1/2}(\Sigma)$, and $\mathfrak{F}<\infty$. Let $D_t^ju(0)\in H^{2N-2j}(\Om)$, $\pa_t^j\theta(0)\in H^{2N-2j}(\Om)$ and $\pa_t^jp(0)\in H^{2N-2j-1}(\Om)$, for $j=1, \ldots, N-1$ along with $D_t^Nu(0)\in\mathscr{Y}(0)$ and $\pa_t^N\theta(0)\in H^0$, all be determined in terms of $u_0$, $\theta_0$ and $\pa_t^jF^1(0)$, $\pa_t^jF^3(0)$, $\pa_t^jF^4(0)$, $\pa_t^jF^5(0)$ for $j=0, \ldots, N-1$.

  There exists a universal constant $T_0>0$ such that if $0<T\le T_0$, then there exists a unique strong solution $(u,p,\theta)$ on $[0,T]$ such that
  \[
  \pa_t^ju\in C^0\left([0,T]; H^{2N-2j}(\Om)\right)\cap L^2\left([0,T];H^{2N-2j+1}(\Om)\right)\quad\text{for}\thinspace j=0,\ldots,N,
  \]
  \[
  \pa_t^jp\in C^0\left([0,T]; H^{2N-2j-1}(\Om)\right)\cap L^2\left([0,T];H^{2N-2j}(\Om)\right)\quad\text{for}\thinspace j=0,\ldots,N-1,
  \]
  \[
  \pa_t^j\theta\in C^0\left([0,T]; H^{2N-2j}(\Om)\right)\cap L^2\left([0,T];H^{2N-2j+1}(\Om)\right)\quad\text{for}\thinspace j=0,\ldots,N,
  \]
  \[
  \pa_t^{N+1}u\in(\mathscr{X}_T)^\ast,\quad\text{and}\quad\pa_t^{N+1}\theta\in(\mathscr{H}^1_T)^\ast.
  \]
  The pair $(D_t^ju,\pa_t^jp, \pa_t^j\theta)$ satisfies
  \ben\label{equ:higher linear BC}
  \left\{
  \begin{aligned}
    &\pa_t(D_t^ju)-\Delta_{\mathscr{A}}(D_t^ju)+\nabla_{\mathscr{A}}(\pa_t^jp)-\pa_t^j(\theta \nabla_{\mathscr{A}}y_3)=F^{1,j}\quad &\text{in}\thinspace\Om,\\
    &\dive_{\mathscr{A}}(D_t^ju)=0\quad &\text{in}\thinspace\Om,\\
    &\pa_t(\pa_t^j\theta)-\Delta_{\mathscr{A}}(\pa_t^j\theta)=F^{3,j}\quad &\text{in}\thinspace\Om,\\
    &S_{\mathscr{A}}(\pa_t^jp,D_t^ju)\mathscr{N}=F^{4,j}\quad &\text{on}\thinspace \Sigma,\\
    &\nabla_{\mathscr{A}}(\pa_t^j\theta)\cdot\mathscr{N}+\pa_t^j\theta\left|\mathscr{N}\right|=F^{5,j}\quad &\text{on}\thinspace \Sigma,\\
    &D_t^ju=0,\quad \pa_t^j\theta=0\quad &\text{on}\thinspace \Sigma_b,
  \end{aligned}
  \right.
  \een
  in the strong sense with initial data $\left(D_t^ju(0),\pa_t^jp(0),\pa_t^j\theta(0)\right)$ for $j=0,\ldots,N-1$, and in the weak sense with initial data $D_t^{N}u(0)\in\mathscr{Y}(0)$ and $\pa_t^{N}\theta(0)\in H^0$. Here the forcing terms $F^{1,j}$, $F^{3,j}$, $F^{4,j}$ and $F^{5,j}$ are as defined by \eqref{equ:force 1} and \eqref{equ:force 2}. Moreover, the solution satisfies the estimate
  \beq \label{est:higher regularity}
  \mathfrak{K}(u,p,\theta)\lesssim P(\mathfrak{E}_0(\eta),\mathfrak{K}(\eta))\exp\left(T P(\mathfrak{E}(\eta))\right)\left(\|u_0\|_{H^{2N}}^2+\|\theta_0\|_{H^{2N}}^2+\mathfrak{F}_0+\mathfrak{F}\right),
  \eeq
  where the constant $C>0$, is independent of $\eta$.
\end{theorem}
\begin{proof}
  First, notice that $P(\cdot,\cdot)$ and $P(\cdot)$ throughout this proof is allowed to change from line to line.
  Theorem \ref{thm:lower regularity} guarantees the existence of $(u,p,\theta)$ satisfying the inclusions \eqref{equ:strong solution}. The $(D_t^ju,\pa_t^jp,\pa_t^j\theta)$ are solutions of \eqref{equ:higher linear BC} in the strong sense when $j=0$ and in the weak sense when $j=1$.  Finally, the estimate \eqref{inequ:est strong solution} holds.

  For an integer $m\ge0$, let $\mathbb{P}_m$ denote the proposition asserting the following three statements. First, $(D_t^ju,\pa_t^jp,\pa_t^j\theta)$ are solutions of \eqref{equ:higher linear BC} in the strong sense for $j=0,\ldots,m$ and in the weak sense when $j=m+1$. Second,
  \[
  \pa_t^ju\in L^\infty H^{2m-2j+2}\cap L^2H^{2m-2j+3},\quad \pa_t^j\theta\in L^\infty H^{2m-2j+2}\cap L^2H^{2m-2j+3}
  \]
  for $j=0,1,\ldots,m+1$, $\pa_t^{m+2}u\in(\mathscr{X}_T)^\ast$, $\pa_t^{m+2}\theta\in(\mathscr{H}^1_T)^\ast$ and
  \[
  \pa_t^jp\in L^\infty H^{2m-2j+1}\cap L^2H^{2m-2j+2}
  \]
  for $j=0,1,\ldots,m$. Third, the estimate
  \ben\label{est:bound pm}
  \begin{aligned}
    &\sum_{j=0}^{m+1}\left(\|\pa_t^ju\|_{L^\infty H^{2m-2j+2}}^2+\|\pa_t^ju\|_{L^2H^{2m-2j+3}}^2+\|\pa_t^j\theta\|_{L^\infty H^{2m-2j+2}}^2+\|\pa_t^j\theta\|_{L^2H^{2m-2j+3}}^2\right)\\
    &\quad+\|\pa_t^{m+2}u\|_{(\mathscr{X}_T)^\ast}^2+\|\pa_t^{m+2}\theta\|_{\mathscr{H}^1_T}^2+\sum_{j=0}^m\left(\|\pa_t^jp\|_{L^\infty H^{2m-2j+1}}^2+\|\pa_t^jp\|_{L^2H^{2m-2j+2}}^2\right)\\
    &\lesssim P(\mathfrak{E}_0(\eta),\mathfrak{K}(\eta))\exp\left(T P(\mathfrak{E}(\eta))\right)\left(\|u_0\|_{H^{2N}}^2+\|\theta_0\|_{H^{2N}}^2+\mathfrak{F}_0+\mathfrak{F}\right)
  \end{aligned}
  \een
  holds.

  We will use a finite induction method to prove that $\mathbb{P}_m$ holds. Theorem \ref{thm:lower regularity} implies that $\mathbb{P}_0$ holds. Then in the rest of this proof, we will divide the proof into two steps.

    Step 1. Proving the first assertion. Suppose that $\mathbb{P}_m$ holds for $m=0, \ldots, N-2$.

    From \eqref{equ:force l2}--\eqref{equ:force dual} of Lemma \ref{lem:force linear}, we have that
    \ben\label{equ:force m+1 l2}
    \begin{aligned}
    &\|F^{1,m+1}(v,q)\|_{L^2H^1}^2+\|F^{3,m+1}(\Theta)\|_{L^2H^1}^2+\|F^{4,m+1}(v,q)\|_{L^2H^{3/2}}^2\\
    &\quad+\|F^{5,m+1}(\Theta)\|_{L^2H^{3/2}}^2\\
    &\lesssim P(\mathfrak{K}(\eta))\bigg(\mathfrak{F}+\sum_{\ell=0}^{m}\left(\|\pa_t^\ell v\|_{L^2H^3}^2+\|\pa_t^\ell \Theta\|_{L^2H^3}^2\right)\\
    &\quad+\sum_{\ell=0}^{m}\Big(\|\pa_t^\ell v\|_{L^\infty H^2}^2+\|\pa_t^\ell \Theta\|_{L^\infty H^2}^2+\|\pa_t^\ell q\|_{L^2H^2}^2+\|\pa_t^\ell q\|_{L^\infty H^1}^2\Big)\bigg),
  \end{aligned}
  \een
  \ben\label{equ:force m+1 l infty}
  \begin{aligned}
    &\|F^{1,m+1}(v,q)\|_{L^\infty H^0}^2+\|F^{3,m+1}(\Theta)\|_{L^\infty H^0}^2+\|F^{4,j}(v,q)\|_{L^\infty H^{1/2}}^2\\
    &\quad+\|F^{5,j}(\Theta)\|_{L^\infty H^{1/2}}^2\\
    &\lesssim P(\mathfrak{K}(\eta))\bigg(\mathfrak{F}+\sum_{\ell=0}^m\Big(\|\pa_t^\ell v\|_{L^\infty H^2}^2+\|\pa_t^\ell \Theta\|_{L^\infty H^2}^2+\|\pa_t^\ell q\|_{L^\infty H^1}^2\Big)\bigg),
  \end{aligned}
  \een
  \ben\label{equ:force m+1 dual}
  \begin{aligned}
    &\|\pa_t(F^{1,m+1}(v,q)-F^{4,m+1}(v,q))\|_{L^2({}_0H^1(\Om))^\ast}^2\\
    &\quad+\|\pa_t(F^{3,m+1}(\Theta)-F^{5,m+1}(\Theta))\|_{L^2({}_0H^1(\Om))^\ast}^2\\
    &\lesssim P(\mathfrak{K}(\eta))\bigg(\mathfrak{F}+\|\pa_t^{m+1} v\|_{L^2 H^{2}}^2+\|\pa_t^{m+1} \Theta\|_{L^2 H^{2}}^2+\|\pa_t^{m+1} q\|_{L^2 H^{1}}^2\\
    &\quad+\sum_{\ell=0}^m\Big(\|\pa_t^\ell v\|_{L^\infty H^{2}}^2+\|\pa_t^\ell v\|_{L^2 H^{2}}^3+\|\pa_t^\ell \Theta\|_{L^\infty H^{2}}^2+\|\pa_t^\ell \Theta\|_{L^2 H^{2}}^3\\
    &\quad+\|\pa_t^\ell q\|_{L^\infty H^{1}}^2+\|\pa_t^\ell q\|_{L^2 H^{2}}^2\Big)\bigg).
  \end{aligned}
  \een
    Now we will use the iteration method. We let $u^0$ be the extension of the initial data $\pa_t^ju(0)$, $j=1,\ldots,N$, given by Lemma A.5 in \cite{GT1}, which may also give $\theta^0$, the extension of the initial data $\pa_t^j\theta(0)$, $j=1,\ldots,N$, and similarly let $p^0$ be the extension of $\pa_t^jp(0)$, $j=1,\ldots,N-1$, given by Lemma A.6 in \cite{GT1}. By \eqref{equ:est initial data j} and the estimates given in the Lemma A.5 and Lemma A.6 in \cite{GT1}, we have
  \ben\label{est:u0 p0 theta0}
    \begin{aligned}
      &\sum_{j=0}^N\left(\|\pa_t^ju^0\|_{L^2H^{2N-2j+1}}^2+\|\pa_t^ju^0\|_{L^\infty H^{2N-2j}}^2+\|\pa_t^j\theta^0\|_{L^2H^{2N-2j+1}}^2+\|\pa_t^j\theta^0\|_{L^\infty H^{2N-2j}}^2\right)\\
      &\quad+\sum_{j=0}^{N-1}\left(\|\pa_t^jp^0\|_{L^2H^{2N-2j}}^2+\|\pa_t^jp^0\|_{L^\infty H^{2N-2j-1}}^2\right)\\
      &\lesssim \sum_{j=0}^N\|D_t^ju(0)\|_{H^{2N-2j}}^2+\sum_{j=0}^{N-1}\|\pa_t^jp(0)\|_{H^{2N-2j-1}}^2+\sum_{j=0}^N\|\pa_t^j\theta(0)\|_{H^{2N-2j}}^2\\
      &\lesssim P(\mathfrak{E}_0(\eta))\left(\|u_0\|_{H^{2N}}^2+\|\theta_0\|_{H^{2N}}^2+\mathfrak{F}_0\right).
    \end{aligned}
    \een
    According to \eqref{equ:force m+1 l2}--\eqref{est:u0 p0 theta0}, we may derive that $F^{1,m+1}(u^0,p^0)$, $F^{3,m+1}(\theta^0)$, $F^{4,m+1}(u^0,p^0)$ and $F^{5,m+1}(\theta^0)$ satisfy \eqref{cond:force}. Also the compatibility condition \eqref{cond:compatibility} with $F^4$ replaced by $F^{4,m+1}(u^0,p^0)$ and $u_0$ replaced by $D_t^{m+1}u(0)$ holds by \eqref{cond:compatibility j} since $u^0$ and $p^0$ achieve the initial data. Then we can apply Theorem \ref{thm:lower regularity} to find a pair $(v^1,q^1,\Theta^1)$ satisfying the conclusions of the theorem.
    For simplicity, we abbreviate \eqref{equ:linear BC} as
    $\mathcal{L}(v,q,\Theta)=\mathbb{F}=(F^1,F^3,F^4,F^5)$. Then
    \[
    \mathcal{L}(v^1,q^1,\Theta^1)=\mathbb{F}^{m+1}:=(F^{1,m+1}(u^0,p^0),F^{3,m+1}(\theta^0),F^{4,m+1}(u^0,p^0),F^{5,m+1}(\theta^0)),
    \]
    \[
    v^1(0)=D_t^{m+1}u(0),\quad q^1(0)=\pa_t^{m+1}p(0),\quad \Theta^1(0)=\pa_t^{m+1}\theta(0).
    \]
    If we denote the left--hand side of \eqref{inequ:est strong solution} as $\mathfrak{B}(u,p,\theta)$, then we may combine \eqref{inequ:est strong solution}, \eqref{equ:initial force j}, \eqref{equ:force m+1 l2}, \eqref{equ:force m+1 dual} and \eqref{est:u0 p0 theta0} to derive that
    \[
    \mathfrak{B}(v^1,q^1,\Theta^1)\lesssim P(\mathfrak{E}_0(\eta),\mathfrak{K}(\eta))\exp\left(P(\mathfrak{E}(\eta))T\right)\left(\|u_0\|_{H^{2N}}^2+\|\theta_0\|_{H^{2N}}^2+\mathfrak{F}_0+\mathfrak{F}\right).
    \]
    Now, suppose that $(v^n,q^n,\Theta^n)$ is given and satisfies $\mathfrak{B}(v^n,q^n,\Theta^n)<\infty$, we define $(u^n,p^n,\theta^n)$ which satisfies the ODEs
    \ben\label{equ:ode uv}
    \left\{
    \begin{aligned}
      &D_t^{m+1}u^n=v^n,\\
      &\pa_t^ju^n(0)=v^n(0) \quad\text{for}\thinspace j=0,\ldots,m,
    \end{aligned}
    \right.
    \een

    \ben\label{equ:ode pq}
    \left\{
    \begin{aligned}
    &\pa_t^{m+1}p^n=q^n,\\
    &\pa_t^jp^n(0)=q^n(0) \quad\text{for}\thinspace j=0,\ldots,m,
    \end{aligned}
    \right.
    \een

    \ben\label{equ:ode theta Theta}
    \left\{
    \begin{aligned}
    &\pa_t^{m+1}\theta^n=\Theta^n,\\
    &\pa_t^j\theta^n(0)=\Theta^n(0) \quad\text{for}\thinspace j=0,\ldots,m.
    \end{aligned}
    \right.
    \een
    From the wellposedness theory of linear ODEs, we know that these ODEs have unique solutions. If we define $\mathfrak{K}(v,q,\Theta)$ by
    \begin{align*}
    \mathfrak{K}(v,q,\Theta):&=\|\pa_t^{m+1}v\|_{L^2H^2}^2+\|\pa_t^{m+1}q\|_{L^2H^1}^2+\|\pa_t^{m+1}\Theta\|_{L^2H^2}^2+\sum_{\ell=0}^m\Big(\|\pa_t^\ell v\|_{L^2H^3}^2\\
    &\quad+\|\pa_t^\ell v\|_{L^\infty H^2}^2+\|\pa_t^\ell \Theta\|_{L^2H^3}^2+\|\pa_t^\ell \Theta\|_{L^\infty H^2}^2+\|\pa_t^\ell q\|_{L^2H^2}^2+\|\pa_t^\ell q\|_{L^\infty H^1}^2\Big),
    \end{align*}
    then the solutions of \eqref{equ:ode uv}--\eqref{equ:ode theta Theta} satisfy the estimate
    \begin{equation}\label{est:un,pn,thetan}
    \begin{aligned}
    \mathfrak{K}(u^n,p^n,\theta^n)&\lesssim P(T)P(\mathfrak{K}(\eta))\bigg(\sum_{j=0}^m\|\pa_t^ju(0)\|_{H^3}^2+\|\pa_t^jp(0)\|_{H^2}^2\\
    &\quad+\|\pa_t^j\theta(0)\|_{H^3}^2+T\mathfrak{B}(v^n,q^n,\Theta^n)\bigg)<\infty,
    \end{aligned}
    \end{equation}
    where $P(T)$ is a polynomial in $T$.

    Applying Theorem \ref{thm:lower regularity} iteratively, we can obtain sequences $\{(v^n,q^n,\Theta^n)\}_{n=1}^\infty$ and $\{u^n,p^n,\theta^n\}_{n=1}^\infty$ satisfying \eqref{equ:ode uv}--\eqref{equ:ode theta Theta} and
    \ben\label{equ:iterate n n-1}
    \begin{aligned}
      \mathcal{L}(v^n,q^n,\Theta^n)&=\mathbb{F}^{m+1}(u^{n-1},p^{n-1},\theta^{n-1}),\\
      v^n(0)&=D_t^{m+1}u(0),\quad q^n(0)=\pa_t^{m+1}p(0),\quad \Theta^n(0)=\pa_t^{m+1}\theta(0).
    \end{aligned}
    \een
    Then
    \begin{align*}
      \mathcal{L}(v^{n+1}-v^n,q^{n+1}-q^n,\Theta^{n+1}-\Theta^n)&=\mathbb{F}^{m+1}(u^n-u^{n-1},p^n-p^{n-1},\theta^n-\theta^{n-1}),\\
      v^{n+1}(0)-v^n(0)=0,\quad &q^{n+1}(0)-q^n(0)=0,\quad \Theta^{n+1}(0)-\Theta^n(0)=0.
    \end{align*}
    Since the terms involving $F^1$, $F^3$, $F^4$ and $F^5$ are canceled in $\mathbb{F}^{m+1}(u^n-u^{n-1},p^n-p^{n-1},\theta^n-\theta^{n-1})$, we can use \eqref{equ:force m+1 l2} and \eqref{equ:force m+1 dual} to derive that
    \begin{align*}
      &\|F^{1,m+1}(u^n-u^{n-1},p^n-p^{n-1})\|_{L^2H^1}^2+\|F^{3,m+1}(\theta^n-\theta^{n-1})\|_{L^2H^1}^2\\
      &\quad+\|F^{4,m+1}(u^n-u^{n-1},p^n-p^{n-1})\|_{L^2H^{3/2}}^2+\|F^{5,m+1}(\theta^n-\theta^{n-1})\|_{L^2H^{3/2}}^2\\
      &\quad+\|\pa_t(F^{1,m+1}(u^n-u^{n-1},p^n-p^{n-1})-F^{4,m+1}(u^n-u^{n-1},p^n-p^{n-1}))\|_{L^2({}_0H^1(\Om))^\ast}^2\\
      &\quad+\|\pa_t(F^{3,m+1}(\theta^n-\theta^{n-1})-F^{5,m+1}(\theta^n-\theta^{n-1}))\|_{L^2({}_0H^1(\Om))^\ast}^2\\
      &\lesssim P(\mathfrak{K}(\eta))\mathfrak{K}(u^n-u^{n-1},p^n-p^{n-1},\theta^n-\theta^{n-1}).
    \end{align*}
    Since, for each $n$, $(u^n,p^n,\theta^n)$ achieves the same initial data, similar to the ODEs \eqref{equ:ode uv}--\eqref{equ:ode theta Theta}, we have that
    \beq
    \mathfrak{K}(u^n-u^{n-1},p^n-p^{n-1},\theta^n-\theta^{n-1})\lesssim P(\mathfrak{K}(\eta))T P(T)\mathfrak{B}(v^n-v^{n-1},q^n-q^{n-1},\Theta^n-\Theta^{n-1}).
    \eeq
    The above two estimates with \eqref{inequ:est strong solution} imply that
    \ben
    \begin{aligned}
      &\mathfrak{B}(v^{n+1}-v^n,q^{n+1}-q^n,\Theta^{n+1}-\Theta^n)\\
      &\lesssim P(\mathfrak{E}_0(\eta),\mathfrak{K}(\eta))\exp\left(P(\mathfrak{E}(\eta))T\right)\\
      &\quad\times T P(T)\mathfrak{B}(v^n-v^{n-1},q^n-q^{n-1},\Theta^n-\Theta^{n-1}),
    \end{aligned}
    \een
    which implies that there exists a universal $T_0>0$ such that if $T\le T_0$, then the sequence $\{(v^n,q^n,\Theta^n)\}_{n=1}^\infty$ converges to $(v,q,\Theta)$ in the norm $\sqrt{\mathfrak{B}(\cdot,\cdot)}$, which reveals that $\{(u^n,p^n,\theta^n)\}_{n=1}^\infty$ converges to $(u,p,\theta)$ in the norm $\sqrt{\mathfrak{K}(\cdot,\cdot)}$.

    By passing to the limit in \eqref{equ:ode uv}--\eqref{equ:ode theta Theta}, we have that $v=D_t^{m+1}u$, $q=\pa_t^{m+1}p$ and $\Theta=\pa_t^{m+1}\theta$. Then, passing to the limit in \eqref{equ:iterate n n-1}, we have that
    \[
    \mathcal{L}(D_t^{m+1}u,\pa_t^{m+1}p,\pa_t^{m+1}\theta)=\mathbb{F}^{m+1}(u,p,\theta).
    \]
    Then Theorem \ref{thm:lower regularity} with the assumption of $\mathbb{P}_m$, which provides that $(D_t^{m+1}u,\\ \pa_t^{m+1}p,\pa_t^{m+1}\theta)$ are solutions of \eqref{equ:higher linear BC} in the strong sense for $j=0,\ldots,m$, enables us to deduce the first assertion of $\mathbb{P}_{m+1}$.

    Theorem \ref{thm:lower regularity}, together with the estimates \eqref{equ:force l2}, \eqref{equ:force m+1 dual} and \eqref{est:bound pm}, gives us that
    \ben
    \begin{aligned}
      &\mathfrak{B}(D_t^{m+1}u,\pa_t^{m+1}p,\pa_t^{m+1}\theta)\\
      &\lesssim P(\mathfrak{E}_0(\eta),\mathfrak{K}(\eta))\exp\left(P(\mathfrak{E}(\eta))T\right)\big(\|u_0\|_{H^{2N}}^2+\|\theta_0\|_{H^{2N}}^2\\
      &\quad+\mathfrak{F}_0+\mathfrak{F}+\|\pa_t^{m+1}u\|_{L^2H^2}^2+\|\pa_t^{m+1}p\|_{L^2H^1}^2+\|\pa_t^{m+1}\theta\|_{L^2H^2}^2\big).
    \end{aligned}
    \een
    On the other hand, the estimate \eqref{est:pa t Dt v j} implies that
    \ben
    \begin{aligned}
      &\|\pa_t^{m+1}u\|_{L^2H^2}^2+\|\pa_t^{m+1}p\|_{L^2H^1}^2+\|\pa_t^{m+1}\theta\|_{L^2H^2}^2\\
      &\le T\left(\|\pa_t^{m+1}u\|_{L^\infty H^2}^2+\|\pa_t^{m+1}p\|_{L^\infty H^1}^2+\|\pa_t^{m+1}\theta\|_{L^\infty H^2}^2\right)\\
      &\lesssim T\left(\|\pa_t^{m+1}u-D_t^{m+1}u\|_{L^\infty H^2}^2+\|D_t^{m+1}u\|_{L^\infty H^2}^2+\|\pa_t^{m+1}p\|_{L^\infty H^1}^2+\|\pa_t^{m+1}\theta\|_{L^\infty H^2}^2\right)\\
      &\lesssim T\Big(P(\mathfrak{K}(\eta))\sum_{\ell=0}^m\|\pa_t^\ell u\|_{L^\infty H^2}^2+\mathfrak{B}(D_t^{m+1}u,\pa_t^{m+1}p,\pa_t^{m+1}\theta)\Big)\\
      &\lesssim T\Big(P(\mathfrak{E}_0(\eta),\mathfrak{K}(\eta))\exp\left(P(\mathfrak{E}(\eta))T\right)\big(\|u_0\|_{H^{2N}}^2+\|\theta_0\|_{H^{2N}}^2+\mathfrak{F}_0+\mathfrak{F}\big)\\
      &\quad+\mathfrak{B}(D_t^{m+1}u,\pa_t^{m+1}p,\pa_t^{m+1}\theta)\Big),
    \end{aligned}
    \een
    where in the last inequality, we have used \eqref{est:bound pm} again. Combining the above two estimates, we may further restrict the size of universal $T_0>0$ such that if $T\le T_0$, then
    \ben\label{est:B Dtu pa tp pa ttheta}
    \begin{aligned}
      &\mathfrak{B}(D_t^{m+1}u,\pa_t^{m+1}p,\pa_t^{m+1}\theta)\\
      &\lesssim P(\mathfrak{E}_0(\eta),\mathfrak{K}(\eta))\exp\left(P(\mathfrak{E}(\eta))T\right)\big(\|u_0\|_{H^{2N}}^2+\|\theta_0\|_{H^{2N}}^2+\mathfrak{F}_0+\mathfrak{F}\big).
    \end{aligned}
    \een

 Step 2. Proving the second and third assertions. In the following, the second and third assertions will be derived simultaneously. The estimate of \eqref{est:B Dtu pa tp pa ttheta} with Lemma \ref{lem:pa tv Dt v} and estimate \eqref{est:bound pm} imply that
      \ben
      \begin{aligned}
        &\|\pa_t^{m+1}u\|_{L^2H^3}^2+\|\pa_t^{m+2}u\|_{L^2H^1}^2+\|\pa_t^{m+3}u\|_{(\mathscr{X}_T)^\ast}^2+\|\pa_t^{m+1}u\|_{L^\infty H^2}^2+\|\pa_t^{m+2}u\|_{L^\infty H^0}^2\\
        &\lesssim P(\mathfrak{K}(\eta))\left(\sum_{\ell=0}^{m+2}\|\pa_t^\ell u\|_{L^2H^{2m-2\ell+3}}^2+\|\pa_t^\ell u\|_{L^\infty H^{2m-2\ell+2}}^2\right)\\
        &\quad+P(\mathfrak{E}_0(\eta),\mathfrak{K}(\eta))\exp\left(P(\mathfrak{E}(\eta))T\right)\big(\|u_0\|_{H^{2N}}^2+\|\theta_0\|_{H^{2N}}^2+\mathfrak{F}_0+\mathfrak{F}\big)\\
        &\lesssim P(\mathfrak{K}(\eta))P(\mathfrak{E}_0(\eta),\mathfrak{K}(\eta))\exp\left(p(\mathfrak{E}(\eta))T\right)\big(\|u_0\|_{H^{2N}}^2+\|\theta_0\|_{H^{2N}}^2+\mathfrak{F}_0+\mathfrak{F}\big)\\
        &\quad+P(\mathfrak{E}_0(\eta),\mathfrak{K}(\eta))\exp\left(P(\mathfrak{E}(\eta))T\right)\big(\|u_0\|_{H^{2N}}^2+\|\theta_0\|_{H^{2N}}^2+\mathfrak{F}_0+\mathfrak{F}\big)\\
        &\lesssim P(\mathfrak{E}_0(\eta),\mathfrak{K}(\eta))\exp\left(P(\mathfrak{E}(\eta))T\right)\big(\|u_0\|_{H^{2N}}^2+\|\theta_0\|_{H^{2N}}^2+\mathfrak{F}_0+\mathfrak{F}\big).
      \end{aligned}
      \een
      Thus
      \ben
      \begin{aligned}
        &\sum_{j=m+1}^{m+2}\left(\|\pa_t^ju\|_{L^2H^{2(m+1)-2j+3}}^2+\|\pa_t^ju\|_{L^\infty H^{2(m+1)-2j+2}}^2\right)+\|\pa_t^{m+3}u\|_{(\mathscr{X}_T)^\ast}^2\\
        &\quad+\sum_{j=m+1}^{m+2}\left(\|\pa_t^jp\|_{L^2H^{2(m+1)-2j+2}}^2+\|\pa_t^jp\|_{L^\infty H^{2(m+1)-2j+1}}^2\right)\\
        &\quad+\sum_{j=m+1}^{m+2}\left(\|\pa_t^j\theta\|_{L^2H^{2(m+1)-2j+3}}^2+\|\pa_t^j\theta\|_{L^\infty H^{2(m+1)-2j+2}}^2\right)+\|\pa_t^{m+3}\theta\|_{(\mathscr{X}_T)^\ast}^2\\
        &\lesssim P(\mathfrak{E}_0(\eta),\mathfrak{K}(\eta))\exp\left(P(\mathfrak{E}(\eta))T\right)\big(\|u_0\|_{H^{2N}}^2+\|\theta_0\|_{H^{2N}}^2+\mathfrak{F}_0+\mathfrak{F}\big).
      \end{aligned}
      \een
      Thus, in order to derive the second and third assertions of $\mathbb{P}_{m+1}$, it suffices to prove that
      \ben\label{est:bound pm 1}
      \begin{aligned}
        &\sum_{j=0}^{m}\left(\|\pa_t^ju\|_{L^2H^{2(m+1)-2j+3}}^2+\|\pa_t^jp\|_{L^2H^{2(m+1)-2j+2}}^2+\|\pa_t^j\theta\|_{L^2H^{2(m+1)-2j+3}}^2\right)\\
        &\quad+\sum_{j=0}^m\left(\|\pa_t^ju\|_{L^\infty H^{2(m+1)-2j+2}}^2+\|\pa_t^jp\|_{L^\infty H^{2(m+1)-2j+1}}^2+\|\pa_t^j\theta\|_{L^\infty H^{2(m+1)-2j+2}}^2\right)\\
        &\lesssim P(\mathfrak{E}_0(\eta),\mathfrak{K}(\eta))\exp\left(P(\mathfrak{E}(\eta))T\right)\big(\|u_0\|_{H^{2N}}^2+\|\theta_0\|_{H^{2N}}^2+\mathfrak{F}_0+\mathfrak{F}\big).
      \end{aligned}
      \een
      In order to prove this estimate, we will use the elliptic regularity of Proposition \ref{prop:high regulatrity} with $k=2N$ and iteration argument. As the first step, we need the estimates for the forcing terms. Combining \eqref{est:bound pm} with the estimates \eqref{equ:force l2} and \eqref{equ:force l infity} of Lemma \ref{lem:force linear} implies that
      \ben\label{est:sum force j}
  \begin{aligned}
    &\sum_{j=1}^{m+1}\Big(\|F^{1,j}\|_{L^2H^{2m-2j+1}}^2+\|F^{3,j}\|_{L^2H^{2m-2j+1}}^2+\|F^{4,j}\|_{L^2H^{2m-2j+3/2}}^2\\
    &\quad+\|F^{5,j}\|_{L^2H^{2m-2j+3/2}}^2+\|F^{1,j}\|_{L^\infty H^{2m-2j}}^2+\|F^{3,j}\|_{L^\infty H^{2m-2j}}^2\\
    &\quad+\|F^{4,j}\|_{L^\infty H^{2m-2j+1/2}}^2+\|F^{5,j}\|_{L^\infty H^{2m-2j+1/2}}^2\Big)\\
    &\lesssim P(\mathfrak{K}(\eta))\bigg(\mathfrak{F}+\sum_{\ell=0}^{j-1}\left(\|\pa_t^\ell u\|_{L^2H^{2m-2\ell+3}}^2+\|\pa_t^\ell \theta\|_{L^2H^{2m-2\ell+3}}^2\right)\\
    &\quad+\sum_{\ell=0}^{j-1}\Big(\|\pa_t^\ell u\|_{L^\infty H^{2m-2\ell+2}}^2+\|\pa_t^\ell \theta\|_{L^\infty H^{2m-2\ell+2}}^2+\|\pa_t^\ell p\|_{L^2H^{2m-2\ell+2}}^2\\
    &\quad+\|\pa_t^\ell p\|_{L^\infty H^{2m-2\ell+1}}^2\Big)\bigg)\\
    &\lesssim P(\mathfrak{E}_0(\eta),\mathfrak{K}(\eta))\exp\left(P(\mathfrak{E}(\eta))T\right)\big(\|u_0\|_{H^{2N}}^2+\|\theta_0\|_{H^{2N}}^2+\mathfrak{F}_0+\mathfrak{F}\big),
  \end{aligned}
  \een
   The estimates of \eqref{est:B Dtu pa tp pa ttheta} , \eqref{est:bound pm} as well as \eqref{est:pat v Dt v l2}, \eqref{est:pat v Dt v linfty} of Lemma \ref{lem:pa tv Dt v}, allow us to deduce that
   \ben\label{est:pat Dtm u}
   \begin{aligned}
     &\|\pa_tD_t^mu\|_{L^\infty H^2}^2+\|\pa_tD_t^mu\|_{L^2 H^3}^2\\
     &\lesssim \|\pa_tD_t^mu-D_t^{m+1}u\|_{L^\infty H^2}^2+\|\pa_tD_t^mu-D_t^{m+1}u\|_{L^2 H^3}^2\\
     &\quad+\|D_t^{m+1}u\|_{L^\infty H^2}^2+\|D_t^{m+1}u\|_{L^2 H^3}^2\\
     &\lesssim P(\mathfrak{K}(\eta))\left(\|D_t^mu\|_{L^\infty H^2}^2+\|D_t^mu\|_{L^2 H^3}^2\right)+\|D_t^{m+1}u\|_{L^\infty H^2}^2+\|D_t^{m+1}u\|_{L^2 H^3}^2\\
     &\lesssim P(\mathfrak{E}_0(\eta),\mathfrak{K}(\eta))\exp\left(P(\mathfrak{E}(\eta))T\right)\big(\|u_0\|_{H^{2N}}^2+\|\theta_0\|_{H^{2N}}^2+\mathfrak{F}_0+\mathfrak{F}\big).
   \end{aligned}
   \een
   Since \eqref{equ:higher linear BC} is satisfied in the strong sense for $j=m$, for almost $t\in [0,T]$, $(D_t^mu, \pa_t^mp,\\
   \pa_t^m\theta)$ solves elliptic system \eqref{equ:SBC} with $F^1$ replaced by $F^{1,m}-\pa_tD_t^mu$, $F^2=0$, $F^3$ replaced by $F^{3,m}-\pa_t(\pa_t^m\theta)$ and $F^4$, $F^5$ replaced by $F^{4,m}$, $F^{5,m}$, respectively. Then, we apply Proposition \ref{prop:high regulatrity} with $r=5$, then square the resulting estimate and integrate over $[0,T]$, to deduce that
   \ben
   \begin{aligned}
   &\|D_t^mu\|_{L^2H^5}^2+\|\pa_t^mp|_{L^2H^4}^2+\|\pa_t^m\theta\|_{L^2H^5}^2\\
   &\lesssim \|F^{1,m}-\pa_tD_t^mu\|_{L^2H^3}^2+\|F^{3,m}-\pa_t(\pa_t^m\theta)\|_{L^2H^3}^2\\
   &\quad+\|F^{4,m}\|_{L^2H^{7/2}}^2+\|F^{5,m}\|_{L^2H^{7/2}}^2\\
   &\lesssim \|F^{1,m}\|_{L^2H^3}^2+\|\pa_tD_t^mu\|_{L^2H^3}^2+\|F^{3,m}\|_{L^2H^3}^2+\|\pa_t(\pa_t^m\theta)\|_{L^2H^3}^2\\
   &\quad+\|F^{4,m}\|_{L^2H^{7/2}}^2+\|F^{5,m}\|_{L^2H^{7/2}}^2\\
   &\lesssim P(\mathfrak{E}_0(\eta),\mathfrak{K}(\eta))\exp\left(P(\mathfrak{E}(\eta))T\right)\big(\|u_0\|_{H^{2N}}^2+\|\theta_0\|_{H^{2N}}^2+\mathfrak{F}_0+\mathfrak{F}\big),
   \end{aligned}
   \een
   where in the last inequality, we have used \eqref{est:B Dtu pa tp pa ttheta}, \eqref{est:sum force j} and \eqref{est:pat Dtm u}. Similarly, Proposition \ref{prop:high regulatrity} with $r=4$ reveals that
   \ben
   \begin{aligned}
     &\|D_t^mu\|_{L^\infty H^4}^2+\|\pa_t^mp\|_{L^\infty H^3}^2+\|\pa_t^m\theta\|_{L^\infty H^4}^2\\
   &\lesssim \|F^{1,m}-\pa_tD_t^mu\|_{L^\infty H^2}^2+\|F^{3,m}-\pa_t(\pa_t^m\theta)\|_{L^\infty H^2}^2\\
   &\quad+\|F^{4,m}\|_{L^\infty H^{5/2}}^2+\|F^{5,m}\|_{L^\infty H^{5/2}}^2\\
   &\lesssim P(\mathfrak{E}_0(\eta),\mathfrak{K}(\eta))\exp\left(P(\mathfrak{E}(\eta))T\right)\big(\|u_0\|_{H^{2N}}^2+\|\theta_0\|_{H^{2N}}^2+\mathfrak{F}_0+\mathfrak{F}\big).
   \end{aligned}
   \een
    By iterating to estimate $\pa_t^ju$, $\pa_t^jp$ and $\pa_t^j\theta$ for $j=1,\ldots,m$, as well as the above two estimates, we have that
   \begin{align*}
    &\|\pa_t^m u\|_{L^\infty H^4}^2+\|\pa_t^mu\|_{L^2H^5}^2\\
    &\lesssim P(\mathfrak{E}_0(\eta),\mathfrak{K}(\eta))\exp\left(P(\mathfrak{E}(\eta))T\right)\big(\|u_0\|_{H^{2N}}^2+\|\theta_0\|_{H^{2N}}^2+\mathfrak{F}_0+\mathfrak{F}\big).
   \end{align*}
   Thus, we have that
   \ben\label{est:bound pm j ge1}
      \begin{aligned}
        &\sum_{j=1}^{m}\left(\|\pa_t^ju\|_{L^2H^{2(m+1)-2j+3}}^2+\|\pa_t^jp\|_{L^2H^{2(m+1)-2j+2}}^2+\|\pa_t^j\theta\|_{L^2H^{2(m+1)-2j+3}}^2\right)\\
        &\quad+\sum_{j=1}^m\left(\|\pa_t^ju\|_{L^\infty H^{2(m+1)-2j+2}}^2+\|\pa_t^jp\|_{L^\infty H^{2(m+1)-2j+1}}^2+\|\pa_t^j\theta\|_{L^\infty H^{2(m+1)-2j+2}}^2\right)\\
        &\lesssim P(\mathfrak{E}_0(\eta),\mathfrak{K}(\eta))\exp\left(P(\mathfrak{E}(\eta))T\right)\big(\|u_0\|_{H^{2N}}^2+\|\theta_0\|_{H^{2N}}^2+\mathfrak{F}_0+\mathfrak{F}\big).
      \end{aligned}
   \een
      Then we apply Proposition \ref{prop:high regulatrity} with $r=2(m+1)+3\le2N+1$, square the result estimate and integrate over $[0,T]$ to see that
      \ben\label{est:bound pm j=0 l2}
      \begin{aligned}
        &\|u\|_{L^2H^{2(m+1)+3}}^2+\|p\|_{L^2H^{2(m+1)+2}}^2+\|\theta\|_{L^2H^{2(m+1)+3}}^2\\
        &\lesssim \|F^1-\pa_tu\|_{L^2H^{2(m+1)+1}}^2+\|F^3-\pa_t\theta\|_{L^2H^{2(m+1)+1}}^2\\
        &\quad+\|F^4\|_{L^2H^{2(m+1)+3/2}}^2+\|F^5\|_{L^2H^{2(m+1)+3/2}}^2\\
        &\lesssim \|F^1\|_{L^2H^{2(m+1)+1}}^2+\|\pa_tu\|_{L^2H^{2(m+1)+1}}^2+\|F^3\|_{L^2H^{2(m+1)+1}}^2+\|\pa_t\theta\|_{L^2H^{2(m+1)+1}}^2\\
        &\quad+\|F^4\|_{L^2H^{2(m+1)+3/2}}^2+\|F^5\|_{L^2H^{2(m+1)+3/2}}^2\\
        &\lesssim P(\mathfrak{E}_0(\eta),\mathfrak{K}(\eta))\exp\left(P(\mathfrak{E}(\eta))T\right)\big(\|u_0\|_{H^{2N}}^2+\|\theta_0\|_{H^{2N}}^2+\mathfrak{F}_0+\mathfrak{F}\big),
      \end{aligned}
      \een
      and then again with $r=2(m+1)+2\le2N$ to see that
      \ben\label{est:bound pm j=0 linfty}
      \begin{aligned}
        &\|u\|_{L^\infty H^{2(m+1)+2}}^2+\|p\|_{L^\infty H^{2(m+1)+1}}^2+\|\theta\|_{L^\infty H^{2(m+1)+2}}^2\\
        &\lesssim \|F^1-\pa_tu\|_{L^\infty H^{2(m+1)}}^2+\|F^3-\pa_t\theta\|_{L^\infty H^{2(m+1)}}^2\\
        &\quad+\|F^4\|_{L^\infty H^{2(m+1)+1/2}}^2+\|F^5\|_{L^\infty H^{2(m+1)+1/2}}^2\\
        &\lesssim \|F^1\|_{L^\infty H^{2(m+1)}}^2+\|\pa_tu\|_{L^\infty H^{2(m+1)}}^2+\|F^3\|_{L^\infty H^{2(m+1)}}^2+\|\pa_t\theta\|_{L^\infty H^{2(m+1)}}^2\\
        &\quad+\|F^4\|_{L^\infty H^{2(m+1)+1/2}}^2+\|F^5\|_{L^\infty H^{2(m+1)+1/2}}^2\\
        &\lesssim P(\mathfrak{E}_0(\eta),\mathfrak{K}(\eta))\exp\left(P(\mathfrak{E}(\eta))T\right)\big(\|u_0\|_{H^{2N}}^2+\|\theta_0\|_{H^{2N}}^2+\mathfrak{F}_0+\mathfrak{F}\big).
      \end{aligned}
      \een
      Thus \eqref{est:bound pm 1} is obtained by summing \eqref{est:bound pm j ge1}--\eqref{est:bound pm j=0 linfty}. This completes the proof.
\end{proof}

\section{Preliminaries for the nonlinear problem}
In order to use linear theory for the problem \eqref{equ:linear BC} to solve the nonlinear problem \eqref{equ:nonlinear BC}, we have to define forcing terms $F^1$, $F^3$, $F^4$, $F^5$ to be used in the linear estimates. Given $u$, $\theta$, $\eta$, we define
\ben\label{equ:force u p theta}
\begin{aligned}
  F^1(u,\theta,\eta)&=\pa_t\bar{\eta}(1+x_3)K\pa_3u-u\cdot\nabla_{\mathscr{A}}u\quad \text{and}\quad F^4(u,\theta,\eta)=\eta\mathscr{N},\\
  F^3(u,\theta,\eta)&=\pa_t\bar{\eta}(1+x_3)K\pa_3\theta-u\cdot\nabla_{\mathscr{A}}\theta \quad \text{and}\quad F^5(u,\theta,\eta)=-\left|\mathscr{N}\right|,
\end{aligned}
\een
where $\mathscr{A}$, $\mathscr{N}$, $K$ are determined as before by $\eta$. Then we define the quantities $\mathfrak{K}_{N}(u,\theta)$ and $\mathfrak{K}_{N}(u)$ as
\ben\label{equ:N u theta}
\begin{aligned}
\mathfrak{K}_{N}(u,\theta)&=\sum_{j=0}^{N}\Big(\|\pa_t^ju\|_{L^2H^{2N-2j+1}}^2+\|\pa_t^ju\|_{L^\infty H^{2N-2j}}^2\\
&\quad+\|\pa_t^j\theta\|_{L^2H^{2N-2j+1}}^2+\|\pa_t^j\theta\|_{L^\infty H^{2N-2j}}^2\Big),
\end{aligned}
\een
and
\beq
\mathfrak{K}_{N}(u)=\sum_{j=0}^{N}\Big(\|\pa_t^ju\|_{L^2H^{2N-2j+1}}^2+\|\pa_t^ju\|_{L^\infty H^{2N-2j}}^2\Big).
\eeq
\subsection{Initial data estimates}\label{sec:initial data}
Since $\eta$ is unknown for the full nonlinear problem, and its evolution is coupled to that of $u$, $p$ and $\theta$, we must reconstruct the initial data to contain this coupling, only with $u_0$, $\theta_0$ and $\eta_0$. Here we will define some quantities which have minor difference from \cite{GT1}.
\beq
\mathscr{E}_0:=\|u_0\|_{H^{2N}}^2+\|\theta_0\|_{H^{2N}}^2+\|\eta_0\|_{H^{2N+1/2}}^2,
\eeq
and
\beq
\mathfrak{E}_0(u,p,\theta):=\sum_{j=0}^N\|\pa_t^ju(0)\|_{H^{2N-2j}}^2+\sum_{j=0}^{N-1}\|\pa_t^jp(0)\|_{H^{2N-2j-1}}^2+\sum_{j=0}^N\|\pa_t^j\theta(0)\|_{H^{2N-2j}}^2.
\eeq
For $j=0,\ldots,N-1$,
\ben
\begin{aligned}
  &\mathfrak{F}_0^j(F^1(u,p,\theta),F^3(u,p,\theta),F^4(u,p,\theta),F^5(u,p,\theta))\\
  &:=\sum_{\ell=0}^j\Big(\|\pa_t^\ell F^1(0)\|_{H^{2N-2\ell-2}}^2+\|\pa_t^\ell F^3(0)\|_{H^{2N-2\ell-2}}^2+\|\pa_t^\ell F^4(0)\|_{H^{2N-2\ell-3/2}}^2\\
  &\quad+\|\pa_t^\ell F^5(0)\|_{H^{2N-2\ell-3/2}}^2\Big).
\end{aligned}
\een
\beq
\mathfrak{E}_0^0(\eta):=\|\eta_0\|_{H^{2N+1/2}}^2,
\eeq
and for $j=1,\ldots,N$,
\beq
\mathfrak{E}_0^j(\eta):=\|\eta_0\|_{H^{2N+1/2}}^2+\sum_{\ell=1}^j\|\pa_t^\ell\eta(0)\|_{H^{2N-2\ell+3/2}}^2.
\eeq
\beq
\mathfrak{E}_0^0(u,p,\theta):=\|u_0\|_{H^{2N}}^2+\|\theta_0\|_{H^{2N}}^2,
\eeq
and for $j=1,\ldots,N$,
\beq
\mathfrak{E}_0^j(u,p,\theta):=\sum_{\ell=0}^j\|\pa_t^\ell u(0)\|_{H^{2N-2\ell}}^2+\sum_{\ell=0}^{j-1}\|\pa_t^\ell p\|_{H^{2N-2\ell-1}}^2+\sum_{\ell=0}^j\|\pa_t^\ell\theta(0)\|_{H^{2N-2\ell}}^2.
\eeq
The following lemma is a minor modification of Lemma 5.2 in \cite{GT1}, so we omit the details of proof.
\begin{lemma}\label{lem:preliminary}
  For $j=0, \ldots,N$,
  \beq\label{est:pa t Dt u}
  \|\pa_t^ju(0)-D_t^ju(0)\|_{H^{2N-2j}}^2\le P_j(\mathfrak{E}_0^j(\eta),\mathfrak{E}_0^j(u,p,\theta))
  \eeq
  and
  \beq \label{est:pa t Dt theta}
  \|\pa_t^j(\theta(0)\nabla_{\mathscr{A}_0}y_{3,0})-D_t^j(\theta(0)\nabla_{\mathscr{A}_0}y_{3,0})\|_{H^{2N-2j}}^2\le P_j(\mathfrak{E}_0^j(\eta),\mathfrak{E}_0^j(u,p,\theta))
  \eeq
  for $P_j(\cdot,\cdot)$ a polynomial such that $P_j(0,0)=0$.

  For $F^1(u,\theta,\eta)$, $F^3(u,\theta,\eta)$, $F^4(u,\theta,\eta)$ and $F^5(u,\theta,\eta)$ defined by \eqref{equ:force u p theta} and $j=0,\ldots,N-1$, we have that
  \beq\label{est:F 0j}
  \mathfrak{F}_0^j(F^1(u,p,\theta),F^3(u,p,\theta),F^4(u,p,\theta),F^5(u,p,\theta))\le P_j(\mathfrak{E}_0^{j+1}(\eta),\mathfrak{E}_0^j(u,p,\theta))
  \eeq
  for $P_j(\cdot,\cdot)$ a polynomial such that $P_j(0,0)=0$.

  For $j=1,\ldots,N-1$, let $F^{1,j}(0)$, $F^{3,j}(0)$, $F^{4,j}(0)$ and $F^{5,j}(0)$ are determined by \eqref{equ:force 1}, \eqref{equ:force 2} and \eqref{equ:force u p theta}. Then
  \ben\label{est:force 1345}
  \begin{aligned}
  &\|F^{1,j}(0)\|_{H^{2N-2j-2}}^2+\|F^{3,j}(0)\|_{H^{2N-2j-2}}^2\\
  &\quad+\|F^{4,j}(0)\|_{H^{2N-2j-3/2}}^2+\|F^{5,j}(0)\|_{H^{2N-2j-3/2}}^2\\
  &\le P_j(\mathfrak{E}_0^{j+1}(\eta),\mathfrak{E}_0^j(u,p,\theta))
  \end{aligned}
  \een
  for $P_j(\cdot,\cdot)$ a polynomial such that $P_j(0,0)=0$.

  For $j=1,\ldots,N-1$,
  \beq
  \left\|\sum_{\ell=0}^j{j \choose \ell}\pa_t^\ell\mathscr{N}(0)\cdot\pa_t^{j-\ell}u(0)\right\|_{H^{2N-2j+3/2}}^2\le P_j(\mathfrak{E}_0^j(\eta),\mathfrak{E}_0^j(u,p,\theta))
  \eeq
  for $P_j(\cdot,\cdot)$ a polynomial such that $P_j(0,0)=0$. Also,
  \beq\label{est:pat eta0}
  \|u_0\cdot\mathscr{N}_0\|_{H^{2N-1/2}(\Sigma)}^2\le \|u_0\|_{H^{2N}}^2\left(1+\|\eta_0\|_{H^{2N+1/2}}^2\right).
  \eeq
\end{lemma}
This lemma allows us to construct all of the initial data $\pa_t^ju(0)$, $\pa_t^j\theta(0)$, $\pa_t^j\eta(0)$ for $j=0,\ldots,N$ and $\pa_t^jp(0)$ for $j=0,\ldots,N-1$.

Assume that $\mathscr{E}_0<\infty$. As before, we will iteratively construct the initial data, but this time we will use Lemma \ref{lem:preliminary}. We define $\pa_t\eta(0)=u_0\cdot\mathscr{N}_0$, where $u_0\in H^{2N-1/2}(\Sigma)$, and $\mathscr{N}_0$ is determined by $\eta_0$. \eqref{est:pat eta0} implies that $\|\pa_t\eta(0)\|_{H^{2N-1/2}}^2\lesssim P(\mathscr{E}_0)$ for a polynomial $P(\cdot)$ such that $P(0)=0$, and hence that $\mathfrak{E}_0^0(u,p,\theta)+\mathfrak{E}_0^1(\eta)\lesssim P(\mathscr{E}_0)$. Then \eqref{est:F 0j} with $j=0$ implies that
\beq
  \mathfrak{F}_0^0(F^1(u,p,\theta),F^3(u,p,\theta),F^4(u,p,\theta),F^5(u,p,\theta))\le P_0(\mathfrak{E}_0^0(\eta),\mathfrak{E}_0^0(u,p,\theta))\lesssim P(\mathscr{E}_0)
  \eeq
  for a polynomial $P(\cdot)$ such that $P(0)=0$. Note that in these estimates and in the estimates below, the polynomial $P(\cdot)$ of $\mathscr{E}_0$ are allowed to change from line to line, but they always satisfy $P(0)=0$.

  In this paragraph, we will give the iterative definition of $\pa_t^jp(0)$, $\pa_t^{j+1}u(0)$, $\pa_t^{j+1}\theta(0)$ and $\pa_t^{j+2}\eta(0)$ for $0\le j\le N-2$. Now suppose that $\pa_t^\ell u(0)$, $\pa_t^\ell\theta(0)$ are known for $\ell=0,\ldots,j$, $\pa_t^\ell \eta(0)$ is known for $\ell=0,\ldots,j+1$, $\pa_t^\ell p(0)$ is known for $\ell=0, \ldots,j-1$ (with the exception for $p(0)$ when $j=0$) and
  \ben\label{est:e0j f0j}
  \begin{aligned}
  &\mathfrak{E}_0^j(u,p,\theta)+\mathfrak{E}_0^{j+1}(\eta)\\
  &\quad+\mathfrak{F}_0^j(F^1(u,p,\theta),F^3(u,p,\theta),F^4(u,p,\theta),F^5(u,p,\theta))\\
  &\lesssim P(\mathscr{E}_0).
  \end{aligned}
  \een
  And according to \eqref{est:force 1345} and \eqref{est:pa t Dt u}, we know that
   \ben\label{est:f j1234 Dt u}
  \begin{aligned}
  &\|D_t^ju(0)\|_{H^{2N-2j}}^2+\|F^{1,j}(0)\|_{H^{2N-2j-2}}^2+\|F^{3,j}(0)\|_{H^{2N-2j-2}}^2\\
  &\quad+\|F^{4,j}(0)\|_{H^{2N-2j-3/2}}^2+\|F^{5,j}(0)\|_{H^{2N-2j-3/2}}^2\\
  &\lesssim P(\mathscr{E}_0).
  \end{aligned}
  \een
  By virtue of estimates \eqref{equ:initial g1 g4}
  \ben
  \begin{aligned}
    &\|\mathfrak{f}^1(F^{1,j}(0),D_t^ju(0))\|_{H^{2N-2i-3}}^2+\|\mathfrak{f}^2(F^{3,j}(0),D_t^ju(0))\|_{H^{2N-2i-3/2}}^2\\
    &\quad+\|\mathfrak{f}^3(F^{1,j}(0),D_t^ju(0))\|_{H^{2N-2i-5/2}}^2\\
    &\lesssim P(\mathscr{E}_0)
  \end{aligned}
  \een
  This allows us to define $\pa_t^jp(0)$ as the solution to \eqref{equ:poisson} with $f^1$, $f^2$, $f^3$ replaced by $\mathfrak{f}^1$, $\mathfrak{f}^2$, $\mathfrak{f}^3$. The Proposition 2.15 in \cite{LW} with $k=2N$ and $r=2N-2j-1$ implies that
  \beq\label{est:pa t j p}
  \|\pa_t^jp(0)\|_{H^{2N-2j-1}}^2\lesssim P(\mathscr{E}_0).
  \eeq
  Now we define
  \beq
  \pa_t^{j+1}\theta(0)=\mathfrak{E}^{02}(\pa_t^j\theta(0),F^{3,j}(0)) \in H^{2N-2j-2}.
  \eeq
  Then according to \eqref{est:e0j f0j} and \eqref{est:f j1234 Dt u}, we have that
  \beq
  \|\pa_t^{j+1}\theta(0)\|_{H^{2N-2j-2}}^2\lesssim P(\mathscr{E}_0).
  \eeq
  Now the estimates \eqref{equ:initial G1 v q},
  \eqref{est:e0j f0j} and \eqref{est:f j1234 Dt u} allow us to defined
  \beq
  D_t^{j+1}u(0):=\mathfrak{E}^{01}\left(F^{1,j}(0)+\pa_t^j(\theta(0)\nabla_{\mathscr{A}_0}y_{3,0}), D_t^ju(0),\pa_t^jp(0)\right) \in H^{2N-2j-2},
  \eeq
  and then according to \eqref{est:pa t Dt u}, we have
  \beq\label{est:pa t j+1 u}
  \|\pa_t^{j+1}u(0)\|_{H^{2N-2j-2}}^2\le P(\mathscr{E}_0).
  \eeq
  Now the estimates \eqref{est:pat eta0}, \eqref{est:e0j f0j} and \eqref{est:pa t j+1 u} allow us to define
  \[
  \pa_t^{j+2}\eta(0)=\sum_{\ell=0}^{j+1}{j+1 \choose \ell}\pa_t^\ell\mathscr{N}(0)\cdot\pa_t^{j+1-\ell}u(0),
   \]
   and imply the estimate
  \beq\label{est:pa t j+2 eta}
  \|\pa_t^{j+2}\eta(0)\|_{H^{2N-2j-5/2}}^2\le P(\mathscr{E}_0).
  \eeq
  Thus, \eqref{est:e0j f0j} together with \eqref{est:pa t j p}--\eqref{est:pa t j+2 eta} imply that
  \[
  \mathfrak{E}_0^{j+1}(u,p,\theta)+\mathfrak{E}_0^{j+2}(\eta)\le P(\mathscr{E}_0),
  \]
  and then \eqref{est:F 0j} implies that
  \[
  \mathfrak{F}_0^{j+1}(F^1(u,p,\theta),F^3(u,p,\theta),F^4(u,p,\theta),F^5(u,p,\theta))\le P(\mathscr{E}_0).
  \]
  Hence that we can deduce the estimate
  \begin{align*}
    &\mathfrak{E}_0^{j+1}(u,p,\theta)+\mathfrak{E}_0^{j+2}(\eta)\\
    &\quad+\mathfrak{F}_0^{j+1}(F^1(u,p,\theta),F^3(u,p,\theta),F^4(u,p,\theta),F^5(u,p,\theta))\\
    &\le P(\mathscr{E}_0).
  \end{align*}
  For $j=N-2$, we have
  \ben\label{est:e0 N-1 f0 N-1}
  \begin{aligned}
    &\mathfrak{E}_0^{N-1}(u,p,\theta)+\mathfrak{E}_0^N(\eta)\\
    &\quad+\mathfrak{F}_0^{N-1}(F^1(u,p,\theta),F^3(u,p,\theta),F^4(u,p,\theta),F^5(u,p,\theta))\\
    &\le P(\mathscr{E}_0).
  \end{aligned}
  \een

  Then, we only need to define $\pa_t^{N-1}p(0)$, $\pa_t^N\theta(0)$ and $\pa_t^Nu(0)$. Like the construction after Lemma \ref{lem:v,q,G}, we need the compatibility conditions on $u_0$ and $\eta_0$. Now we have constructed $\pa_t^jp(0)$ for $j=0,\ldots,N-2$, $\pa_t^ju(0)$, $\pa_t^j\theta(0)$, $F^{1,j}(0)$, $F^{3,j}(0)$, $F^{4,j}(0)$, $F^{5,j}(0)$ for $j=0,\ldots,N-1$, and $\pa_t^j\eta(0)$ for $j=0,\ldots,N$. We say that $u_0$ and $\eta_0$ satisfy the $N$-th order compatibility conditions if \ben\label{cond:compatibility N}
  \left\{
  \begin{aligned}
    &\nabla_{\mathscr{A}_0}\cdot(D_t^ju(0))=0\quad &\text{in}\thinspace\Om,\\
    &D_t^ju(0)=0\quad &\text{on}\thinspace\Sigma_b,\\
    &\Pi_0\left(F^{4,j}(0)+\mathbb{D}_{\mathscr{A}_0}D_t^ju(0)\mathscr{N}_0\right)=0\quad &\text{on}\thinspace\Sigma,
  \end{aligned}
  \right.
  \een
  for $j=0,\ldots,N-1$, where $\Pi_0$ is the projection defined as in \eqref{def:projection} and $D_t$ be the operator defined by \eqref{def:Dt}. Note that if $u_0$ and $\eta_0$ satisfy \eqref{cond:compatibility N}, then the $j$-th compatibility condition \eqref{cond:compatibility j} is satisfied for $j=0,\ldots,N-1$.
  Then the construction of $\pa_t^{N-1}p(0)$ is the same as \cite{GT1} using the compatibility condition \eqref{cond:compatibility N} and the elliptic theory of $\mathscr{A}$- Poisson equations \eqref{equ:poisson} derived by Y. Guo and I. Tice in \cite{GT1} and L. Wu in \cite{LW}. And
  \beq\label{est:pa t N-1 p}
  \|\pa_t^{N-1}p(0)\|_{H^1}^2\le P(\mathscr{E}_0).
  \eeq
  Then we set $\pa_t^N\theta(0)=\mathfrak{E}^{02}(\pa_t^{N-1}\theta(0),F^{3,N-1}(0))\in H^0$ due to \eqref{equ:e02} and \eqref{est:force 1345}, and set $D_t^Nu(0)=\mathfrak{E}^{01}(F^{1,N-1}(0)+\pa_t^{N-1}(\theta \nabla_{\mathscr{A}_0}y_{3,0}), D_t^{N-1}u(0),\pa_t^{N-1}p(0))\in H^0$ due to \eqref{equ:initial G1 v q} and Lemma \ref{lem:preliminary}. And $D_t^Nu(0)\in \mathscr{Y}(0)$ is guaranteed by the construction of $\pa_t^{N-1}p(0)$. As before, we have
  \beq\label{est:pa t N u theta}
  \|\pa_t^Nu(0)\|_{H^0}^2+\|\pa_t^N\theta(0)\|_{H^0}^2\lesssim P(\mathscr{E}_0).
  \eeq
  This completes the construction of initial data. Then summing the estimates \eqref{est:e0 N-1 f0 N-1}, \eqref{est:pa t N-1 p} and \eqref{est:pa t N u theta}, we directly have the following proposition.
  \begin{proposition}\label{prop:high order initial}
    Suppose that $u_0$, $\theta_0$ and $\eta_0$ satisfy $\mathscr{E}_0<\infty$. Let the initial data $\pa_t^ju(0)$, $\pa_t^j\theta(0)$, $\pa_t^j\eta(0)$ for $j=0,\ldots,N$ and $\pa_t^jp(0)$ for $j=0,\ldots,N-1$ be given as above. Then
    \beq\label{est:high order initial}
    \mathscr{E}_0\le \mathfrak{E}_0(u,p,\theta)+\mathfrak{E}_0(\eta)\lesssim P(\mathscr{E}_0).
    \eeq
    Here $\mathfrak{E}_0(\eta)=\mathfrak{E}_0^N(\eta)$, which is defined in \eqref{def:norm eta}.
  \end{proposition}

\subsection{Transport equation}

Here we consider the equation
\ben\label{equ:transport}
\left\{
\begin{aligned}
  &\pa_t\eta+u_1\pa_1\eta+u_2\pa_2\eta=u_3 \quad \text{on}\thinspace \Sigma,\\
  &\eta(0)=\eta_0.
\end{aligned}
\right.
\een
The local well--posedness of \eqref{equ:transport} has been proved by  L. Wu, which is the Theorem 2.17 in \cite{LW}. The idea of his proof is similar to the proof of Theorem 5.4 in \cite{GT1}. In \cite{LW}, L. Wu has proved in Lemma 2.18, that the difference of $\eta$ and $\eta_0$ in a small time period is also small.

\subsection{Forcing estimates}

In the next section for the estimates of full nonlinear problem, we need some forcing quantities. Besides $\mathfrak{F}$ and $\mathfrak{F}_0$ which have been defined in \eqref{def:force F F0}, we define the following quantities
\begin{align*}
  \mathcal{F}:&=\sum_{j=0}^{N-1}\left(\|\pa_t^j F^1\|_{L^2H^{2N-2j-1}}^2+\|\pa_t^j F^3\|_{L^2H^{2N-2j-1}}^2\right)+\|\pa_t^NF^1\|_{L^2H^0}^2+\|\pa_t^NF^3\|_{L^2H^0}^2\\
  &\quad+\sum_{j=0}^N\left(\|\pa_t^jF^4\|_{L^\infty H^{2N-2j-1/2}(\Sigma)}^2+\|\pa_t^jF^5\|_{L^\infty H^{2N-2j-1/2}(\Sigma)}^2\right),
\end{align*}
\begin{align*}
  \mathcal{H}:&=\sum_{j=0}^{N-1}\left(\|\pa_t^j F^1\|_{L^2H^{2N-2j-1}}^2+\|\pa_t^j F^3\|_{L^2H^{2N-2j-1}}^2\right)\\
  &\quad+\sum_{j=0}^{N-1}\left(\|\pa_t^jF^4\|_{L^2 H^{2N-2j-1/2}(\Sigma)}^2+\|\pa_t^jF^5\|_{L^2 H^{2N-2j-1/2}(\Sigma)}^2\right),
\end{align*}

The following theorem is similar to Theorem 2.21 in \cite{LW} with obvious modification.
\begin{theorem}\label{lem:forcing estimates}
  The forcing terms satisfy the estimates
  \ben
  \mathfrak{F}&\lesssim& P(\mathfrak{K}(\eta))+P(\mathfrak{K}_N(u,\theta)),\\
  \mathfrak{F}_0&\lesssim& P(\mathscr{E}_0),\\
  \mathcal{F}&\lesssim& P(\mathfrak{K}(\eta))+P(\mathfrak{K}_N(u,\theta)),\\
  \mathcal{H}&\lesssim& T\left(P(\mathfrak{K}(\eta))+P(\mathfrak{K}_N(u,\theta))\right).
  \een
\end{theorem}
\begin{proof}
  The proof of this theorem is the same as the proof of Theorem $2.21$ in \cite{LW}, so we omit the details here.
\end{proof}

\section{Local well-posedness for the nonlinear problem}

\subsection{Construction of approximate solutions}

In order to solve the \eqref{equ:NBC}, we will construct a sequence of approximate solutions $(u^m, p^m,\theta^m, \eta^m)$, then take the limit $m\to\infty$. First, we construct an initial pair $(u^0,\theta^0,\eta^0)$ as a start point, then we iteratively define all sequences $(u^m,p^m,\theta^m,\eta^m)$ for $m\ge1$.

Suppose that the initial data $(u_0,\theta_0,\eta_0)$ has given. According to the Lemma A.5 in \cite{GT1}, there exist $u^0$ and $\theta^0$ defined in $\Om\times [0,\infty)$ with $\pa_t^ju^0(0)=\pa_t^ju(0)$, $\pa_t^j\theta^0(0)=\pa_t^j\theta(0)$, for $j=0,\ldots,N$, satisfying
 \beq\label{est:sequence u0 theta0}
 \mathfrak{K}_N(u^0,\theta^0)\lesssim P(\mathscr{E}_0).
 \eeq
 Then we consider the equation \eqref{equ:poisson} with $u$ replaced by $u^0$. From the Theorem 2.17 in \cite{LW}, the hypothesis of which is satisfied by \eqref{est:high order initial} and \eqref{est:sequence u0 theta0}, there exists a $\eta^0$ defined in $\Om\times [0,T_0)$, which satisfies $\pa_t^j\eta^0(0)=\pa_t^j\eta(0)$ for $j=0,\ldots,N$ as well as
 \[
 \mathfrak{K}(\eta^0)\lesssim P(\mathscr{E}_0).
 \]

Then for any integer $m\ge1$, we formally define the sequence $(u^m,p^m,\theta^m,\eta^m)$ on the time interval $[0,T_m)$ as the solutions of system
\ben\label{equ:iteration equation}
\left\{
\begin{aligned}
  &\pa_tu^m-\Delta_{\mathscr{A}^{m-1}}u^m+\nabla_{\mathscr{A}^{m-1}}p^m+\theta^m\nabla_{\mathscr{A}^{m-1}}y_3^{m-1} &\\
  &\qquad=\pa_t\bar{\eta}^{m-1}(1+x_3)K^{m-1}\pa_3u^{m-1} -u^{m-1}\cdot\nabla_{\mathscr{A}^{m-1}}u^{m-1}&\text{in}\thinspace\Om,\\
  &\dive_{\mathscr{A}^{m-1}}u^m=0&\text{in}\thinspace\Om,\\
  &\pa_t\theta^m-\Delta_{\mathscr{A}^{m-1}}\theta^m=\pa_t\bar{\eta}^{m-1}(1+x_3)K^{m-1}\pa_3\theta^{m-1}-u^{m-1}\cdot\nabla_{\mathscr{A}^{m-1}}\theta^{m-1} &\text{in}\thinspace\Om,\\
  &S_{\mathscr{A}^{m-1}}(p^m,u^m)\mathscr{N}^{m-1}=\eta^{m-1}\mathscr{N}^{m-1} &\text{on}\thinspace\Sigma,\\
  &\nabla_{\mathscr{A}^{m-1}}\theta^m\cdot\mathscr{N}^{m-1}+\theta^m\left|\mathscr{N}^{m-1}\right|=-\left|\mathscr{N}^{m-1}\right| &\text{on}\thinspace\Sigma,\\
  &u^m=0,\quad\theta^m=0 &\text{on}\thinspace\Sigma_b,
\end{aligned}
\right.
\een
and
\beq \label{equ:iteration theta}
\pa_t\eta^m=u^m\cdot\mathscr{N}^m\quad\text{on}\thinspace\Sigma,
\eeq
where $\mathscr{A}^{m-1}$, $\mathscr{N}^{m-1}$, $K^{m-1}$ are determined in terms of $\eta^{m-1}$ and $\mathscr{N}^m$ is in terms of $\eta^m$, with the initial data $(u^m(0), \theta^m(0), \eta^m(0))=(u_0,\theta_0,\eta_0)$.

In the following, we will prove that these sequences can be defined for any integer $m\ge1$ and the existence time $T_m$ does not shrink to $0$ as $m\to \infty$. The following theorem is a modified version of Theorem $2.24$ in \cite{LW}, which improves the estimate \eqref{est:higher regularity} using the energy structure and elliptic estimates.

\begin{theorem}\label{thm:boundedness}
  Suppose $J(0)>\delta>0$. Assume that the initial data $(u_0,\theta_0,\eta_0)$ satisfy $\mathscr{E}_0<\infty$ and $\pa_t^ju(0)$, $\pa_t^j\theta(0)$, $\pa_t^j\eta(0)$, for $j=0,\ldots,N$, are given as above from the Proposition \ref{prop:high order initial}. Then there exists a positive constant $\mathscr{Z}<\infty$ and $0<\bar{T}<1$ depending on $\mathscr{E}_0$, such that if $0<T<\bar{T}$, then there exists a sequence $\{(u^m,p^m,\theta^m,\eta^m)\}_{m=0}^\infty$ (when $m=0$, the sequence should be considered as $(u^0,\theta^0,\eta^0)$) satisfying the iteration equation \eqref{equ:iteration equation} within the time interval $[0,T)$ and the following properties:
  \begin{enumerate}[1.]
    \item The iteration sequence satisfies
    \beq
    \mathfrak{K}_N(u^m,\theta^m)+\mathfrak{K}(\eta^m)\le \mathscr{Z}
    \eeq
    for any integer $m\ge0$, where the temporal norm is taken with respect to $[0,T)$.\\
    \item $J^m(t)\ge\delta/2$ with $0\le t\le T$, for any integer $m\ge0$.
  \end{enumerate}
\end{theorem}
\begin{proof}
  In this proof, we will follow the path of proof of Theorem 2.24 in \cite{LW}. We will use an infinite induction to prove this theorem. Let us denote the above two assertions as statement $\mathbb{P}_m$.

    Step $1$. $\mathbb{P}_0$ case. The only modification here is that the construction of $u^0$ and $\theta^0$ reveals that $\mathfrak{K}_N(u^0,\theta^0)\lesssim P(\mathscr{E}_0)$. Then the rest proof of this case is the same as the proof of Theorem $2.24$ in \cite{LW}. Hence, $\mathbb{P}_0$ holds. That is $\mathfrak{K}_N(u^0,\theta^0)+\mathfrak{K}(\eta^0)\le \mathscr{Z}$ with the temporal norm taken with respect to $[0,T)$ and $J^0(t)\ge\delta/2$ for $0\le t\le T$.

    In the following, we suppose that $\mathbb{P}_{m-1}$ holds for $m\ge1$. Then we will prove that $\mathbb{P}_m$ also holds.

    Step $2$. $\mathbb{P}_m$ case: energy estimates of $\theta^m$ and $u^m$. By Theorem \ref{thm:higher regularity}, the pair $(D_t^Nu^m,\pa_t^Np^m,\pa_t^N\theta^m)$ satisfies the equation
  \ben
  \left\{
  \begin{aligned}
    &\pa_t(D_t^Nu^m)-\Delta_{\mathscr{A}^{m-1}}(D_t^Nu^m)+\nabla_{\mathscr{A}^{m-1}}(\pa_t^Np^m)&\\
    &\qquad-\pa_t^N(\theta^m \nabla_{\mathscr{A}^{m-1}}y_3^{m-1})=F^{1,N}\quad &\text{in}\thinspace\Om,\\
    &\dive_{\mathscr{A}^{m-1}}(D_t^Nu^m)=0\quad &\text{in}\thinspace\Om,\\
    &\pa_t(\pa_t^N\theta^m)-\Delta_{\mathscr{A}^{m-1}}(\pa_t^N\theta^m)=F^{3,N}\quad &\text{in}\thinspace\Om,\\
    &S_{\mathscr{A}^{m-1}}(\pa_t^Np^m,D_t^Nu^m)\mathscr{N}^{m-1}=F^{4,N}\quad &\text{on}\thinspace \Sigma,\\
    &\nabla_{\mathscr{A}^{m-1}}(\pa_t^N\theta^m)\cdot\mathscr{N}^{m-1}+\pa_t^j\theta^m\left|\mathscr{N}^{m-1}\right|=F^{5,N}\quad &\text{on}\thinspace \Sigma,\\
    &D_t^Nu^m=0,\quad \pa_t^N\theta^m=0\quad &\text{on}\thinspace \Sigma_b,
  \end{aligned}
  \right.
  \een
  in the weak sense, where $F^{1,N}$, $F^{3,N}$, $F^{4,N}$ and $F^{5,N}$ are given in terms of $u^m$, $p^m$, $\theta^m$, and $u^{m-1}$, $p^{m-1}$, $\theta^{m-1}$, $\eta^{m-1}$. Then for any test function $\phi\in(\mathscr{H}^1_T)^{m-1}$, where $(\mathscr{H}^1_T)^{m-1}$ is the space $\mathscr{H}^1_T$ with $\eta$ replaced by $\eta^{m-1}$, the following holds
\begin{align*}
  \left<\pa_t(\pa_t^N\theta^m), \phi\right>_{\ast}+\left(\pa_t^N\theta^m,\phi\right)_{\mathscr{H}^1_T}+\left(\pa_t^N\theta^m\left|\mathscr{N}^{m-1}\right|,\phi\right)_{L^2H^0(\Sigma)}\\
  =\left(F^{3,N},\phi\right)_{\mathscr{H}^0_T}+\left(F^{5,N},\phi\right)_{L^2H^0(\Sigma)}.
\end{align*}
Therefore, when taking the test function $\phi=\pa_t^N\theta^m$, we have the energy structure
\ben
\begin{aligned}
  &\f12\int_{\Om} J^{m-1}|\pa_t^N\theta^m|^2+\int_0^t\int_{\Om} J^{m-1}|\nabla_{\mathscr{A}^{m-1}}(\pa_t^N\theta^m)|^2+\int_0^t\int_{\Sigma}|\pa_t^N\theta^m|^2\left|\mathscr{N}^{m-1}\right|\\
  &=\f12\int_{\Om} J^{m-1}(0)|\pa_t^N\theta^m(0)|^2+\f12\int_0^t\int_{\Om} \pa_tJ^{m-1}|\pa_t^N\theta^m|^2\\
  &\quad+\int_0^t\int_{\Om} J^{m-1}F^{3,N}\pa_t^N\theta^m+\int_0^t\int_{\Sigma}F^{5,N}\pa_t^N\theta^m.
\end{aligned}
\een
By induction hypothesis, \eqref{est:high order initial}, trace theory and Cauchy inequality, we have
\ben
\begin{aligned}
  &\|\pa_t^N\theta^m\|_{L^\infty H^0}^2+\|\pa_t^N\theta^m\|_{L^2H^1}^2\\
  &\lesssim \sup_{0\le t\le T}\left(\f12\int_{\Om} J^{m-1}|\pa_t^N\theta^m|^2+\int_0^t\int_{\Om} J^{m-1}|\nabla_{\mathscr{A}^{m-1}}(\pa_t^N\theta^m)|^2+\int_0^t\int_{\Sigma}|\pa_t^N\theta^m|^2\right)\\
  &\lesssim \f12\int_{\Om} J^{m-1}(0)|\pa_t^N\theta^m(0)|^2+\f12\int_0^T\int_{\Om} \pa_tJ^{m-1}|\pa_t^N\theta^m|^2+\int_0^T\int_{\Om} J^{m-1}F^{3,N}\pa_t^N\theta^m\\
  &\quad+\int_0^T\int_{\Sigma}F^{5,N}\pa_t^N\theta^m\\
  &\lesssim P(\mathscr{E}_0)+T\mathscr{Z}\|\pa_t^N\theta^m\|_{L^\infty H^0}^2+\sqrt{T}\mathscr{Z}\|F^{3,N}\|_{L^2H^0}\|\pa_t^N\theta^m\|_{L^\infty H^0}\\
  &\quad+\sqrt{T}\|F^{5,N}\|_{L^\infty H^{-1/2}(\Sigma)}\|\pa_t^N\theta^m\|_{L^2H^{1/2}(\Sigma)}\\
  &\lesssim P(\mathscr{E}_0)+T\mathscr{Z}\|\pa_t^N\theta^m\|_{L^\infty H^0}^2+\sqrt{T}\|F^{3,N}\|_{L^2H^0}^2\\
  &\quad+\sqrt{T}\mathscr{Z}^2\|\pa_t^N\theta^m\|_{L^\infty H^0}^2+\sqrt{T}\|F^{5,N}\|_{L^\infty H^{-1/2}(\Sigma)}^2+\sqrt{T}\|\pa_t^N\theta^m\|_{L^2H^{1/2}(\Sigma)}^2
\end{aligned}
\een
for a polynomial $P(0)=0$. Taking $T\le \min\{1/4, 1/(16\mathscr{Z}^4)\}$ and absorbing the extra terms on the right--hand side into left--hand side imply
\beq\label{est:energy thetam um}
\|\pa_t^N\theta^m\|_{L^\infty H^0}^2+\|\pa_t^N\theta^m\|_{L^2H^1}^2\lesssim P(\mathscr{E}_0)+\sqrt{T}\|F^{3,N}\|_{L^2H^0}^2+\sqrt{T}\|F^{5,N}\|_{L^\infty H^{-1/2}(\Sigma)}^2.
\eeq
By induction hypothesis, we have
\begin{align*}
  &\|F^{3,N}\|_{L^2H^0}^2\\
  &\lesssim P(\mathfrak{K}(\eta^{m-1}))\left(\sum_{j=0}^{N-1}\|\pa_t^ju^m\|_{L^2H^2}^2+\|\pa_t^j\theta^m\|_{L^2H^2}^2\right)+\mathcal{F}\\
  &\lesssim P(\mathscr{E}_0+\mathscr{Z})+\mathcal{F},
\end{align*}
\begin{align*}
  &\|F^{5,N}\|_{L^\infty H^{-1/2}(\Sigma)}^2\\
  &\lesssim P(\mathfrak{K}(\eta^{m-1}))\left(\sum_{j=0}^{N-1}\|\pa_t^ju^m\|_{L^\infty H^2}^2+\|\pa_t^j\theta^m\|_{L^\infty H^2}^2\right)+\mathcal{F}\\
  &\lesssim P(\mathscr{E}_0+\mathscr{Z})+\mathcal{F}.
\end{align*}
And, the energy estimates about $u^m$ is the same as the proof of of Theorem $2.24$ in \cite{LW}. Therefore, we have
\beq
\|\pa_t^N u^m\|_{L^2H^1}^2+\|\pa_t^N \theta^m\|_{L^2H^1}^2\lesssim P(\mathscr{E}_0)+\sqrt{T}P(\mathscr{E}_0+\mathscr{Z})+\sqrt{T}\mathcal{F}.
\eeq

Step 3. $\mathbb{P}_m$ case: elliptic estimates for $\theta^m$, $u^m$. For $0\le n\le N-1$, the $n$-th order heat equation is
\ben
\left\{
\begin{aligned}
&\pa_t(\pa_t^n\theta^m)-\Delta_{\mathscr{A}^{m-1}}\pa_t^n\theta^m=F^{3,n}\quad &\text{in}\thinspace \Om,\\
&\nabla_{\mathscr{A}^{m-1}}\pa_t^n\theta^m\cdot\mathscr{N}^{m-1}+\pa_t^n\theta^m\left|\mathscr{N}^{m-1}\right|=F^{5,n}\quad &\text{on}\thinspace \Sigma,\\
&\pa_t^n\theta^m=0\quad &\text{on}\thinspace \Sigma_b.
\end{aligned}
\right.
\een
The elliptic estimate in the proof of Lemma \ref{lem:S lower regularity} reveals that
\beq \label{est:elliptic thetam}
\|\pa_t^n\theta^m\|_{L^2H^{2N-2n+1}}^2\lesssim \|F^{3,n}\|_{L^2H^{2N-2n-1}}^2+\|\pa_t^{n+1}\theta^m\|_{L^2H^{2N-2n-1}}^2+\|F^{5,n}\|_{L^2H^{2N-2n-1/2}}^2.
\eeq
As what we did before,
\begin{align*}
  &\|F^{3,n}\|_{L^2H^{2N-2n-1}}^2\\
  &\lesssim T P(\mathfrak{K}(\eta^{m-1}))\left(\sum_{j=0}^{N-2}\|\pa_t^j\theta^m\|_{L^\infty H^{2N-2j-1}}^2+\|\pa_t^ju^m\|_{L^\infty H^{2N-2j-1}}^2\right)+\mathcal{H}\\
  &\lesssim T P(\mathscr{E}_0+\mathscr{Z})+\mathcal{H}.
\end{align*}
\begin{align*}
  &\|F^{5,n}\|_{L^2H^{2N-2n-1}}^2\\
  &\lesssim T P(\mathfrak{K}(\eta^{m-1}))\left(\sum_{j=0}^{N-2}\|\pa_t^j\theta^m\|_{L^\infty H^{2N-2j-1}}^2+\|\pa_t^ju^m\|_{L^\infty H^{2N-2j-1}}^2\right)+\mathcal{H}\\
  &\lesssim T P(\mathscr{E}_0+\mathscr{Z})+\mathcal{H}.
\end{align*}
But for the term $\|\pa_t^{n+1}\theta^m\|_{L^2H^{2N-2n-1}}^2$, we estimate backward from $N-1$ to $0$. First, when $n=N-1$, this is the case of energy estimate of $\theta^m$. Then we iteratively use the elliptic estimates \eqref{est:elliptic thetam} from $n=N-2$ to $n=0$ to obtain all the control of $\|\pa_t^{n+1}\theta^m\|_{L^2H^{2N-2n-1}}^2$.

And the elliptic estimate for $u^m$ is the same as the proof of of Theorem $2.24$ in \cite{LW}. Thereore, we have that
\ben\label{est:elliptic thetam um}
\begin{aligned}
&\sum_{n=0}^{N-1}\left(\|\pa_t^nu^m\|_{L^2H^{2N-2n+1}}^2+\|\pa_t^n\theta^m\|_{L^2H^{2N-2n+1}}^2\right)\\
&\lesssim P(\mathscr{E}_0)+\sqrt{T}P(1+\mathscr{E}_0+\mathscr{Z})+\sqrt{T}\mathcal{F}+\mathcal{H}.
\end{aligned}
\een
Step $4$. $\mathbb{P}_m$ case: synthesis of estimates for $u^m$ and $\theta^m$. Combining \eqref{est:energy thetam um}, \eqref{est:elliptic thetam um} and Lemma $2.19$ in \cite{LW}, we deduce that
\beq
\mathfrak{K}_N(u^m, \theta^m)\lesssim P(\mathscr{E}_0)+\sqrt{T}P(\mathscr{E}_0+\mathscr{Z})+\sqrt{T}\mathcal{F}+\mathcal{H}.
\eeq
Then by the induction hypothesis and the forcing estimates of Lemma \ref{lem:forcing estimates}, we have that
\[
\mathcal{F}\lesssim P(\mathfrak{K}(\eta^{m-1}))+P(\mathfrak{K}_N(u^{m-1}, \theta^{m-1}))\lesssim P(\mathscr{Z}),
\]
\[
\mathcal{H}\lesssim T\left(P(\mathfrak{K}(\eta^{m-1}))+P(\mathfrak{K}_N(u^{m-1}, \theta^{m-1}))\right)\lesssim T P(\mathscr{Z}).
\]
Hence we obtain the estimate
\beq
\mathfrak{K}_N(u^m, \theta^m)\le C\left( P(\mathscr{E}_0)+\sqrt{T}P(\mathscr{E}_0+\mathscr{Z})\right)
\eeq
for some universal constant $C>0$. Taking $\mathscr{Z}\ge 2 C P(\mathscr{E}_0)$ and then  taking $T$ sufficient small  which depends on $\mathscr{Z}$, we can achieve that $\mathfrak{K}_N(u^m, \theta^m)\le 2 C P(\mathscr{E}_0)\le \mathscr{Z}$.

Step 5. $\mathbb{P}_m$ case: estimate for $\eta^m$ and $J^m(t)$. These estimates are exactly the same as the proof of of Theorem $2.24$ in \cite{LW}. So we omit the details here.

Thus, we can take $\mathscr{Z}=P(\mathscr{E}_0)$ for some polynomial $P(\cdot)$ and $T$ small enough depending on $\mathscr{Z}$ to deduce that
\beq
\mathfrak{K}_N(u^m,\theta^m)\le\mathscr{Z}
\eeq
and
\beq
J^m(t)\ge\delta/2 \quad \text{for}\thinspace t\in [0,T].
\eeq
Hence $\mathbb{P}_m$ holds. By induction, $\mathbb{P}_n$ holds for any integer $n\ge0$.
\end{proof}

\begin{theorem}\label{thm:uniform boundedness}
  Assume the same conditions as Theorem \ref{thm:boundedness}. Then
  \beq
  \mathfrak{K}(u^m,p^m,\theta^m)+\mathfrak{K}(\eta^m)\lesssim P(\mathscr{E}_0)
  \eeq
  for a polynomial $P(\cdot)$ satisfying $P(0)=0$.
\end{theorem}
\begin{proof}
  From the estimates \eqref{est:higher regularity}, \eqref{est:high order initial}, Lemma \ref{lem:forcing estimates} as well as Theorem $2.17$ in \cite{LW},  we directly have that
  \[
  \mathfrak{K}(u^m,p^m,\theta^m)+\mathfrak{K}(\eta^m)\lesssim P(\mathscr{E}_0)+P(\mathfrak{K}_N(u^m,\theta^m)+\mathfrak{K}(\eta^m)).
  \]
  Then, applying the Theorem \ref{thm:boundedness}, we have that
  \[
  \mathfrak{K}(u^m,p^m,\theta^m)+\mathfrak{K}(\eta^m)\lesssim P(\mathscr{E}_0).
  \]
\end{proof}

\subsection{Contraction}

According to Theorem \ref{thm:uniform boundedness}, we may extract weakly converging subsequences from $\{(u^m,p^m,\theta^m,\eta^m)\}_{m=0}^\infty$. Unfortunately, the original sequence $\{(u^m,p^m,\theta^m,\eta^m)\}_{m=0}^\infty$ could not be guaranteed to converge to the same limit. In order to obtain the desired solution to \eqref{equ:NBC} by passing to the limit in \eqref{equ:iteration equation} and \eqref{equ:iteration theta}, we need to study its contraction in some norm.

For $T>0$, we define the norms
\ben
\begin{aligned}
  \mathfrak{N}(v,q,\Theta; T)&=\|v\|_{L^\infty H^2}^2+\|v\|_{L^2H^3}^2+\|\pa_tv\|_{L^\infty H^0}^2+\|\pa_tv\|_{L^2H^1}^2+\|q\|_{L^\infty H^1}^2+\|q\|_{L^2H^2}^2\\
  &\quad+\|\Theta\|_{L^\infty H^2}^2+\|\Theta\|_{L^2H^3}^2+\|\pa_t\Theta\|_{L^\infty H^0}^2+\|\pa_t\Theta\|_{L^2H^1}^2\\
  \mathfrak{M}(\zeta;T)&=\|\zeta\|_{L^\infty H^{5/2}}^2+\|\pa_t\zeta\|_{L^\infty H^{3/2}}^2+\|\pa_t^2\zeta\|_{L^2H^{1/2}}^2,
\end{aligned}
\een
where the norm $L^pH^k$ is $L^p([0,T];H^k(\Om))$ in $\mathfrak{N}$, and is $L^p([0,T];H^k(\Sigma))$ in $\mathfrak{M}$.

The next theorem is not only used to prove the contraction of approximate solutions, but also used to verify the uniqueness of solutions to \eqref{equ:NBC}. To avoid confusion with $\{(u^m,p^m,\theta^m,\eta^m)\}$, we refer to velocities as $v^j$, $w^j$, pressures as $q^j$, temperatures as $\Theta^j$, $\vartheta^j$, and surface functions as $\zeta^j$ for $j=1,2$.
\begin{theorem}\label{thm:contraction}
  For $j=1,2$, suppose that $v^j$, $q^j$, $\Theta^j$, $w^j$, $\vartheta^j$ and $\zeta^j$ satisfy the initial data $\pa_t^kv^1(0)=\pa_t^kv^2(0)$, $\pa_t^k\Theta^1(0)=\pa_t^k\Theta^2(0)$, for $k=0,1$, $q^1(0)=q^2(0)$ and $\zeta^1(0)=\zeta^2(0)$,  and that the following system holds:
  \ben\label{equ:difference}
  \left\{
  \begin{aligned}
    &\pa_tv^j-\Delta_{\mathscr{A}^j}v^j+\nabla_{\mathscr{A}^j}q^j-\Theta^j\nabla_{\mathscr{A}^j}y_3^j=\pa_t\bar{\zeta}^j(1+x_3)K^j\pa_3w^j\\
    &\qquad-w^j\cdot\nabla_{\mathscr{A}^j}w^j\quad &\text{in}\thinspace \Om,\\
    &\dive_{\mathscr{A}^j}v^j=0 \quad &\text{in}\thinspace \Om,\\
    &\pa_t\Theta^j-\Delta_{\mathscr{A}^j}\Theta^j=\pa_t\bar{\zeta}^j(1+x_3)K^j\pa_3\vartheta^j-w^j\cdot\nabla_{\mathscr{A}^j}\vartheta^j\quad &\text{in}\thinspace \Om,\\
    &S_{\mathscr{A}^j}(q^j,v^j)\mathscr{N}^j=\zeta^j\mathscr{N}^j\quad &\text{on}\thinspace \Sigma,\\
    &\nabla_{\mathscr{A}^j}\Theta^j\cdot\mathscr{N}^j+\Theta^j\left|\mathscr{N}^j\right|=-\left|\mathscr{N}^j\right|\quad &\text{on}\thinspace \Sigma,\\
    &v^j=0,\quad \Theta^j=0\quad &\text{on}\thinspace \Sigma_b,\\
    &\pa_t\zeta^j=w^j\cdot\mathscr{N}^j\quad &\text{on}\thinspace \Sigma,
  \end{aligned}
  \right.
  \een
  where $\mathscr{A}^j$, $\mathscr{N}^j$, $K^j$ are determined by $\zeta^j$. Assume that $\mathfrak{K}(v^j,q^j,\Theta^j)$, $\mathfrak{K}(w^j,0,\vartheta^j)$ and $\mathfrak{K}(\zeta^j)$ are bounded by $\mathscr{Z}$.

  Then there exists $0<T_1<1$ such that for any $0<T<T_1$, then we have
  \beq\label{est:n}
  \mathfrak{N}(v^1-v^2,q^1-q^2,\Theta^1-\Theta^2;T)\le \f12\mathfrak{N}(w^1-w^2,0,\vartheta^1-\vartheta^2;T),
  \eeq
  \beq\label{est:m}
  \mathfrak{M}(\zeta^1-\zeta^2;T)\lesssim \mathfrak{N}(w^1-w^2,0,\vartheta^1-\vartheta^2;T).
  \eeq
\end{theorem}
\begin{proof}
  This proof follows the path of Theorem $6.2$ in \cite{GT1}. First, we define $v=v^1-v^2$, $w=w^1-w^2$, $\Theta=\Theta^1-\Theta^2$, $\vartheta=\vartheta^1-\vartheta^2$, $q=q^1-q^2$.

  Step 1. Energy evolution for differences. Like the proof of Theorem $6.2$ in \cite{GT1}, we can derive the PDE satisfied by $v$, $q$ and $\Theta$:
  \ben
  \left\{
  \begin{aligned}
  &\pa_tv+\dive_{\mathscr{A}^1}S_{\mathscr{A}^1}(q,v)-\Theta \nabla_{\mathscr{A}^1}y_3^1=\dive_{\mathscr{A}^1}(\mathbb{D}_{(\mathscr{A}^1-\mathscr{A}^2)}v^2)+H^1 \quad &\text{in}\thinspace \Om,\\
  &\dive_{\mathscr{A}^1}v=H^2 \quad &\text{in}\thinspace \Om,\\
  &\pa_t\Theta-\Delta_{\mathscr{A}^1}\Theta=\dive_{\mathscr{A}^1}(\nabla_{(\mathscr{A}^1-\mathscr{A}^2)}\Theta^2)+H^3 \quad &\text{in}\thinspace \Om,\\
  &S_{\mathscr{A}^1}(q,v)\mathscr{N}^1=\mathbb{D}_{(\mathscr{A}^1-\mathscr{A}^2)}v^2\mathscr{N}^1+H^4 \quad &\text{on}\thinspace \Sigma,\\
  &\nabla_{\mathscr{A}^1}\Theta\cdot\mathscr{N}^1+\Theta\left|\mathscr{N}^1\right|=-\nabla_{(\mathscr{A}^1-\mathscr{A}^2)}\Theta^2\cdot\mathscr{N}^1+H^5 \quad &\text{on}\thinspace \Sigma,\\
  &v=0,\quad \Theta=0 \quad &\text{on}\thinspace \Sigma_b,\\
  &v(t=0)=0,\quad \Theta(t=0)=0,
  \end{aligned}
  \right.
  \een
  and the PDE satisfied by $\pa_tv$, $\pa_tq$, $\pa_t\Theta$ from taking temporal derivative for the above system:
  \ben
  \left\{
  \begin{aligned}
  &\pa_t(\pa_t v)+\dive_{\mathscr{A}^1}S_{\mathscr{A}^1}(\pa_t q,\pa_t v)-\pa_t(\Theta \nabla_{\mathscr{A}^1}y_3^1)&\\
  &\qquad=\dive_{\mathscr{A}^1}(\mathbb{D}_{\pa_t(\mathscr{A}^1-\mathscr{A}^2)}v^2)+\tilde{H}^1 \thinspace &\text{in}\thinspace \Om,\\
  &\dive_{\mathscr{A}^1}\pa_t v=\tilde{H}^2 \thinspace &\text{in}\thinspace \Om,\\
  &\pa_t(\pa_t\Theta)-\Delta_{\mathscr{A}^1}\pa_t\Theta=\dive_{\mathscr{A}^1}(\nabla_{(\pa_t\mathscr{A}^1-\pa_t\mathscr{A}^2)}\Theta^2)+\tilde{H}^3 \thinspace &\text{in}\thinspace \Om,\\
  &S_{\mathscr{A}^1}(\pa_t q,\pa_t v)\mathscr{N}^1=\mathbb{D}_{(\pa_t\mathscr{A}^1-\pa_t\mathscr{A}^2)}v^2\mathscr{N}^1+\tilde{H}^4 \thinspace &\text{on}\thinspace \Sigma,\\
  &\nabla_{\mathscr{A}^1}\pa_t\Theta\cdot\mathscr{N}^1+\pa_t\Theta\left|\mathscr{N}^1\right|=-\nabla_{\pa_t(\mathscr{A}^1-\mathscr{A}^2)}\Theta^2\cdot\mathscr{N}^1+\tilde{H}^5 \thinspace &\text{on}\thinspace \Sigma,\\
  &\pa_tv=0,\quad \pa_t\Theta=0 \thinspace &\text{on}\thinspace \Sigma_b,\\
  &\pa_tv(t=0)=0,\quad \pa_t\Theta(t=0)=0,
  \end{aligned}
  \right.
  \een
  where $H^2$, $H^4$, $\tilde{H}^2$ and $\tilde{H}^4$ have been given by Y. Guo and I. Tice in \cite{GT1},
  \begin{align*}
    H^1&=\Theta^2\nabla_{\mathscr{A}^1-\mathscr{A}^2}y_3^1+\Theta^2\nabla_{\mathscr{A}^2}(y_3^1-y_3^2)+\dive_{\mathscr{A}^1-\mathscr{A}^2}(\mathbb{D}_{\mathscr{A}^2}v^2)-\nabla_{\mathscr{A}^1-\mathscr{A}^2}q^2\\
    &\quad+\pa_t\bar{\zeta}^1(1+x_3)K^1(\pa_3w^1-\pa_3w^2)+(\pa_t\bar{\zeta}^1-\pa_t\bar{\zeta}^2)(1+x_3)K^1\pa_3w^2\\
    &\quad+\pa_t\bar{\zeta}^1(1+x_3)(K^1-K^2)\pa_3w^2-(w^1-w^2)\cdot\nabla_{\mathscr{A}^1}w^1-w^2\cdot\nabla_{\mathscr{A}^1}(w^1-w^2)\\
    &\quad-w^2\cdot\nabla_{\mathscr{A}^1-\mathscr{A}^2}w^2,\\
    H^3&=\dive_{\mathscr{A}^1-\mathscr{A}^2}(\nabla_{\mathscr{A}^2}\Theta^2)+\pa_t\bar{\zeta}^1(1+x_3)K^1(\pa_3\vartheta^1-\pa_3\vartheta^2)\\
    &\quad+(\pa_t\bar{\zeta^1}-\pa_t\bar{\zeta}^2)(1+x_3)K^1\pa_3\vartheta^2+\pa_t\bar{\zeta}^1(K^1-K^2)\pa_3w^2-(w^1-w^2)\cdot\nabla_{\mathscr{A}^1}\vartheta^1\\
    &\quad-w^2\cdot\nabla_{\mathscr{A}^1}(\vartheta^1-\vartheta^2)-w^2\cdot\nabla_{\mathscr{A}^1-\mathscr{A}^2}\vartheta^2,\\
    H^5&=-\nabla_{\mathscr{A}^2}\Theta^2\cdot(\mathscr{N}^1-\mathscr{N}^2)-\Theta^2\left(\left|\mathscr{N}^1\right|-\left|\mathscr{N}^2\right|\right),\\
    \tilde{H}^1&=\pa_tH^1+\dive_{\pa_t\mathscr{A}^1}(\mathbb{D}_{\mathscr{A}^1-\mathscr{A}^2}v^2)+\dive_{\mathscr{A}^1}(\mathbb{D}_{\mathscr{A}^1-\mathscr{A}^2}\pa_tv^2)+\dive_{\pa_t\mathscr{A}^1}(\mathbb{D}_{\mathscr{A}^1}v)\\
    &\quad+\dive_{\mathscr{A}^1}(\mathbb{D}_{\pa_t\mathscr{A}^1}v)-\nabla_{\pa_t\mathscr{A}^1}q,\\
    \tilde{H}^3&=\pa_tH^3+\dive_{\pa_t\mathscr{A}^1}(\nabla_{(\mathscr{A}^1-\mathscr{A}^2)}\Theta^2)+\dive_{\mathscr{A}^1}(\nabla_{(\mathscr{A}^1-\mathscr{A}^2)}\pa_t\Theta^2)+\dive_{\pa_t\mathscr{A}^1}\nabla_{\mathscr{A}^1}\Theta\\
    &\quad+\dive_{\mathscr{A}^1}\nabla_{\pa_t\mathscr{A}^1}\Theta,\\
    \tilde{H}^5&=\pa_tH^5-\nabla_{(\mathscr{A}^1-\mathscr{A}^2)}\pa_t\Theta^2\cdot\mathscr{N}^1-\nabla_{(\mathscr{A}^1-\mathscr{A}^2)}\Theta^2\cdot\pa_t\mathscr{N}^1-\nabla_{\mathscr{A}^1}\Theta\cdot\pa_t\mathscr{N}^1\\
    &\quad-\nabla_{\pa_t\mathscr{A}^1}\Theta\cdot\mathscr{N}^1-\Theta\pa_t\left|\mathscr{N}^1\right|.
  \end{align*}
  Then we can deduce the equations
  \ben\label{equ:evolution difference}
  \begin{aligned}
    &\f12\int_{\Om}|\pa_tv|^2J^1(t)+\f12\int_0^t\int_{\Om}|\mathbb{D}_{\mathscr{A}^1}\pa_tv|^2J^1\\
    &=\f12\int_0^t\int_{\Om}|\pa_tv|^2(\pa_tJ^1K^1)J^1+\int_0^t\int_{\Om}\pa_t(\Theta \nabla_{\mathscr{A}^1}y_3^1)\cdot\pa_tv J^1\\
    &\quad+\int_0^t\int_{\Om}J^1(\tilde{H}^1\cdot \pa_tv+\tilde{H}^2\pa_tq)\\
    &\quad-\f12\int_0^t\int_{\Om}J^1\mathbb{D}_{\pa_t\mathscr{A}^1-\pa_t\mathscr{A}^2}v^2:\mathbb{D}_{\mathscr{A}^1}\pa_t v-\int_0^t\int_{\Sigma}\tilde{H}^3\cdot\pa_tv,\\
    &\f12\int_{\Om}|\pa_t\Theta|^2J^1(t)+\int_0^t\int_{\Om}|\nabla_{\mathscr{A}^1}\pa_t\Theta|^2J^1+\int_0^t\int_{\Sigma}|\pa_t\Theta|^2\left|\mathscr{N}^1\right|\\
    &=\f12\int_0^t\int_{\Om}|\pa_t\Theta|^2(\pa_tJ^1K^1)J^1+\int_0^t\int_{\Om}J^1\tilde{H}^3\cdot \pa_t\Theta\\
    &\quad-\int_0^t\int_{\Om}J^1\nabla_{\pa_t\mathscr{A}^1-\pa_t\mathscr{A}^2}\Theta^2\cdot\nabla_{\mathscr{A}^1}\pa_t \Theta+\int_0^t\int_{\Sigma}\tilde{H}^5\cdot\pa_t\Theta.
  \end{aligned}
  \een

  Step 2. Estimates for the forcing terms. Now we need to estimate the forcing terms that appear on the right-hand sides of \eqref{equ:evolution difference}. Throughout this section, $P(\cdot)$ is written as a polynomial such that $P(0)=0$, which allows to be changed from line to line. The estimates for $\|\tilde{H}^1\|_0$, $\|\tilde{H}^2\|_0$, $\|\pa_t\tilde{H}^2\|_0$, $\|\tilde{H}^4\|_{-1/2}$, $\|H^1\|_r$, $\|H^2\|_{r+1}$, $\|H^4\|_{r+1/2}$, $\|\dive_{\mathscr{A}^1}(\mathbb{D}_{(\mathscr{A}^1-\mathscr{A}^2)}v^2)\|_r$ and $\|\mathbb{D}_{(\mathscr{A}^1-\mathscr{A}^2)}v^2\mathscr{N}^1\|_{r+1/2}$ have been done by Guo and Tice in \cite{GT1}. So we can directly using them only after replacing $\varepsilon$ by $\mathscr{Z}$. By the same method, we can also deduce that
  \ben\label{est:tilde H3}
  \begin{aligned}
    \|\tilde{H}^3\|_0&\lesssim P(\sqrt{\mathscr{Z}})\big(\|\Theta\|_2+\|\zeta^1-\zeta^2\|_{3/2}+\|\pa_t\zeta^1-\pa_t\zeta^2\|_{1/2}+\|\pa_t^2\zeta^1-\pa_t^2\zeta^2\|_{1/2}\\
    &\quad+\|w^1-w^2\|_0+\|\pa_tw^1-\pa_tw^2\|_0+\|\vartheta^1-\vartheta^2\|_1+\|\pa_t\vartheta^1-\pa_t\vartheta^2\|_1\big),
  \end{aligned}
  \een
  \beq\label{est:tilde H5}
  \|\tilde{H}^5\|_{-1/2}\lesssim P(\sqrt{\mathscr{Z}})\big(\|\zeta^1-\zeta^2\|_{1/2}+\|\pa_t\zeta^1-\pa_t\zeta^2\|_{1/2}+\|\Theta\|_2\big),
  \eeq
  and for $r=0,1$,
  \ben\label{est:bound H3}
  \begin{aligned}
  \|H^3\|_r &\lesssim P(\sqrt{\mathscr{Z}})\big(\|\zeta^1-\zeta^2\|_{r+1/2}+\|\pa_t\zeta^1-\pa_t\zeta^2\|_{r-1/2}\\
  &\quad+\|w^1-w^2\|_r+\|\vartheta^1-\vartheta^2\|_{r+1}\big),
  \end{aligned}
  \een
  \beq
  \|H^5\|_{r+1/2}\lesssim P(\sqrt{\mathscr{Z}})\|\zeta^1-\zeta^2\|_{r+3/2},
  \eeq
  \beq
  \|\dive_{\mathscr{A}^1}(\nabla_{\mathscr{A}^1-\mathscr{A}^2}\Theta^2)\|_r\lesssim P(\sqrt{\mathscr{Z}})\|\zeta^1-\zeta^2\|_{r+3/2},
  \eeq
  \beq\label{est:bound difference theta2}
  \|\nabla_{\mathscr{A}^1-\mathscr{A}^2}\Theta^2\cdot\mathscr{N}^1\|_{r+1/2}\lesssim P(\sqrt{\mathscr{Z}})\|\zeta^1-\zeta^2\|_{r+3/2}.
  \eeq

  Step 3. Energy estimates of $\pa_tv$ and $\pa_t\Theta$.  First, owing to the assumption and Sobolev embeddings, we obtain that
  \beq\label{est:bound J K}
  \|J^1\|_{L^\infty}+\|K^1\|_{L^\infty}\lesssim 1+P(\sqrt{\mathscr{Z}})\quad \text{and}\quad \|\pa_tJ^1\|_{L^\infty}\lesssim P(\sqrt{\mathscr{Z}}).
  \eeq
  The bounds of \eqref{est:bound J K} reveals that
  \beq\label{est:rhs 1 evolution difference}
  \f12\int_0^t\int_{\Om}|\pa_t\Theta|^2(\pa_tJ^1K^1)J^1\lesssim P(\sqrt{\mathscr{Z}})\f12\int_0^t\int_{\Om}|\pa_t\Theta|^2J^1.
  \eeq
  In addition, estimates \eqref{est:tilde H3}, \eqref{est:tilde H5} together with trace theory and the Poincar\'e inequality reveals that
  \ben\label{est:rhs 2 evolution difference}
  \begin{aligned}
    &\int_0^t\int_{\Om}J^1\tilde{H}^3\cdot \pa_t\Theta-\int_0^t\int_{\Om}J^1\nabla_{\pa_t\mathscr{A}^1-\pa_t\mathscr{A}^2}\Theta^2\cdot\nabla_{\mathscr{A}^1}\pa_t \Theta-\int_0^t\int_{\Sigma}\tilde{H}^5\cdot\pa_t\Theta\\
    &\le \int_0^t\int_{\Om}\|J^1\|_{L^\infty}\left(\|J^1\|_{L^\infty}\|\tilde{H}^3\|_0\|\pa_t\Theta\|_0+\|\nabla_{\pa_t\mathscr{A}^1-\pa_t\mathscr{A}^2}\Theta^2\|_0\|\nabla_{\mathscr{A}^1}\pa_t \Theta\|_0\right)\\
    &\quad+\int_0^t\|\tilde{H}^5\|_{-1/2}\|\pa_t\Theta\|_{1/2}\\
    &\lesssim \int_0^tP(\sqrt{\mathscr{Z}})\sqrt{\mathcal{Z}},
  \end{aligned}
  \een
  where we have written
  \ben
  \begin{aligned}
  \mathcal{Z}:&=\|\zeta^1-\zeta^2\|_{3/2}^2+\|\pa_t\zeta^1-\pa_t\zeta^2\|_{1/2}^2+\|\pa_t^2\zeta^1-\pa_t^2\zeta^2\|_{1/2}^2\\
    &\quad+\|w^1-w^2\|_1^2+\|\pa_tw^1-\pa_tw^2\|_1^2+\|\vartheta^1-\vartheta^2\|_1^2+\|\pa_t\vartheta^1-\pa_t\vartheta^2\|_1^2\\
    &\quad+\|v\|_2^2+\|q\|_1^2+\|\Theta\|_2^2.
  \end{aligned}
  \een
  Combining \eqref{est:rhs 1 evolution difference}, \eqref{est:rhs 2 evolution difference}, \eqref{equ:evolution difference}, Poincar\'e inequality of Lemma A.14 in \cite{GT1} and Lemma $2.9$ in \cite{LW} and utilizing Cauchy inequality to absorb $\|\pa_t\Theta\|_1$ into left,  yield that
  \ben
  \begin{aligned}
    &\f12\int_{\Om}|\pa_t\Theta|^2J^1(t)+\f12\int_0^t\|\pa_t\Theta\|_1^2\\
    &\le P(\sqrt{\mathscr{Z}})\f12\int_{\Om}|\pa_t\Theta|^2J^1(t)+\int_0^tP(\sqrt{\mathscr{Z}})\mathcal{Z}
  \end{aligned}
  \een
  Then Gronwall's lemma and Lemma $2.9$ in \cite{LW} imply that
  \beq
  \|\pa_t\Theta\|_{L^\infty H^0}^2+\|\pa_t\Theta\|_{L^2H^1}^2\le \exp\{P(\sqrt{\mathscr{Z}})T\}\int_0^TP(\sqrt{\mathscr{Z}})\mathcal{Z}.
  \eeq
  Then energy estimates for $\pa_tv$ are likely the same as what Guo and Tice did in \cite{GT1}, so we omit the details. The energy estimates for $\pa_tv$ and $\pa_t\Theta$ allow us to deduce that
  \ben
  \begin{aligned}
    &\|\pa_tv\|_{L^\infty H^0}^2+\|\pa_tv\|_{L^2H^1}^2+\|\pa_t\Theta\|_{L^\infty H^0}^2+\|\pa_t\Theta\|_{L^2H^1}^2\\
    &\le \exp\{P(\sqrt{\mathscr{Z}})T\}\Bigg[P(\sqrt{\mathscr{Z}})\|q\|_{L^2H^0}^2+C\|\pa_t\zeta^1-\pa_t\zeta^2\|_{L^2H^{-1/2}}^2+\int_0^TP(\sqrt{\mathscr{Z}})\mathcal{Z}\\
    &\quad+P(\sqrt{\mathscr{Z}})\|q\|_{L^\infty H^0}^2\bigg(\sum_{j=0}^1\|\pa_t^j\zeta^1-\pa_t^j\zeta^2\|_{L^\infty H^{1/2}}+\|v\|_{L^\infty H^1}\bigg)\\
    &\quad+P(\sqrt{\mathscr{Z}})\|q\|_{L^2 H^0}^2\bigg(\sum_{j=0}^2\|\pa_t^j\zeta^1-\pa_t^j\zeta^2\|_{L^2 H^{1/2}}+\|v\|_{L^2 H^1}\bigg)\Bigg],
  \end{aligned}
  \een
  where the temporal norm of $L^\infty$ and $L^2$ are computed over $[0,T]$.

  Step 4. Elliptic estimates for $v$, $q$ and $\Theta$.  For $r=0,1$, we combine Proposition \eqref{prop:high regulatrity} with estimates \eqref{est:bound H3}--\eqref{est:bound difference theta2} as well as the bounds of $\|H^1\|_r$, $\|H^2\|_{r+1}$, $\|H^4\|_{r+1/2}$ $\|\dive_{\mathscr{A}^1}(\mathbb{D}_{(\mathscr{A}^1-\mathscr{A}^2)}v^2)\|_r$, $\|\mathbb{D}_{(\mathscr{A}^1-\mathscr{A}^2)}v^2\mathscr{N}^1\|_{r+1/2}$ done in the proof of Theorem $6.2$ in \cite{GT1} to deduce that
  \ben
  \begin{aligned}
    &\|v\|_{r+2}^2+\|q\|_{r+1}^2+\|\Theta\|_{r+2}^2\\
    &\lesssim C(\eta_0)\bigg(\|\pa_tv\|_r^2+\|\dive_{\mathscr{A}^1}(\mathbb{D}_{(\mathscr{A}^1-\mathscr{A}^2)}v^2)\|_r^2+\|H^1\|_r^2+\|H^2\|_{r+1}^2+\|\pa_t\Theta\|_r^2\\
    &\quad+\|H^3\|_r^2+\|\dive_{\mathscr{A}^1}(\nabla_{\mathscr{A}^1-\mathscr{A}^2}\Theta^2)\|_r^2+\|\mathbb{D}_{(\mathscr{A}^1-\mathscr{A}^2)}v^2\mathscr{N}^1\|_{r+1/2}^2+\|H^4\|_{r+1/2}^2\\
    &\quad+\|\nabla_{\mathscr{A}^1-\mathscr{A}^2}\Theta^2\cdot\mathscr{N}^1\|_{r+1/2}^2+\|H^5\|_{r+1/2}^2\bigg)\\
    &\lesssim C(\eta_0)\bigg(\|\pa_tv\|_r^2+\|\pa_t\Theta\|_r^2+\|\zeta^1-\zeta^2\|_{r+1/2}^2\\
    &\quad+P(\sqrt{\mathscr{Z}})\big(\|\zeta^1-\zeta^2\|_{r+3/2}^2+\|\pa_t\zeta^1-\pa_t\zeta^2\|_{r-1/2}^2\\
    &\quad+\|w^1-w^2\|_{r+1}^2+\|\vartheta^1-\vartheta^2\|_{r+1}^2\big)\bigg).
  \end{aligned}
  \een
  Then we take supremum in time over $[0,T]$, when $r=0$, to deduce
  \ben
  \begin{aligned}
    &\|v\|_{L^\infty H^2}^2+\|q\|_{L^\infty H^1}^2+\|\Theta\|_{L^\infty H^2}^2\\
    &\lesssim C(\eta_0)\bigg(\|\pa_tv\|_{L^\infty H^0}^2+\|\pa_t\Theta\|_{L^\infty H^0}^2+\|\zeta^1-\zeta^2\|_{L^\infty H^{1/2}}^2\\
    &\quad+P(\sqrt{\mathscr{Z}})\big(\|\zeta^1-\zeta^2\|_{L^\infty H^{3/2}}^2+\|\pa_t\zeta^1-\pa_t\zeta^2\|_{L^\infty H^{-1/2}}^2\\
    &\quad+\|w^1-w^2\|_{L^\infty H^1}^2+\|\vartheta^1-\vartheta^2\|_{L^\infty H^1}^2\big)\bigg).
  \end{aligned}
  \een
  Then we integrate over $[0,T]$ when $r=1$ to find
  \ben
  \begin{aligned}
    &\|v\|_{L^2H^3}^2+\|q\|_{L^2H^2}^2+\|\Theta\|_{L^2H^3}^2\\
    &\lesssim C(\eta_0)\bigg(\|\pa_tv\|_{L^2H^1}^2+\|\pa_t\Theta\|_{L^2H^1}^2+\|\zeta^1-\zeta^2\|_{L^2H^{3/2}}^2\\
    &\quad+P(\sqrt{\mathscr{Z}})\big(\|\zeta^1-\zeta^2\|_{L^2H^{5/2}}^2+\|\pa_t\zeta^1-\pa_t\zeta^2\|_{L^2H^{1/2}}^2\\
    &\quad+\|w^1-w^2\|_{L^2H^2}^2+\|\vartheta^1-\vartheta^2\|_{L^2H^2}^2\big)\bigg).
  \end{aligned}
  \een

  Step 5. Estimates of $\zeta^1-\zeta^2$ and contraction. After making preparations in the above steps, we can derive the contraction results. Since this step follows exactly the same manner as the proof of  Theorem $6.2$ in \cite{GT1}, we omit the details here. Hence, we get the \eqref{est:n} and \eqref{est:m}.
\end{proof}

\subsection{Proof of Theorem \ref{thm:main}}
 Now we can combine Theorem \ref{thm:uniform boundedness} and Theorem \ref{thm:contraction} to produce a unique strong solution to \eqref{equ:NBC}. It is notable that Theorem \ref{thm:main} can be directly derived from the following theorem, which will be proved in the same manner as the proof of Theorem $6.3$ in \cite{GT1}.
 \begin{theorem}
   Assume that $u_0$, $\theta_0$, $\eta_0$ satisfy $\mathscr{E}_0<\infty$ and that the initial data $\pa_t^ju(0)$, etc. are constructed in Section \ref{sec:initial data} and satisfy the $N$-th compatibility conditions \eqref{cond:compatibility N}. Then there exists $0<T_0<1$ such that if $0<T\le T_0$, then there exists a solution $(u,p,\theta,\eta)$ to the problem \eqref{equ:NBC} on the time interval $[0,T]$ that achieves the initial data and satisfies
   \beq\label{est:bound K}
   \mathfrak{K}(u,p,\theta)+\mathfrak{K}(\eta)\le CP(\mathscr{E}_0),
   \eeq
   for a universal constant $C>0$. The solution is unique through functions that achieve the initial data. Moreover, $\eta$ is such that the mapping $\Phi(\cdot,t)$, defined by \eqref{map:phi}, is a $C^{2N-1}$ diffeomorphism for each $t\in [0, T]$.
 \end{theorem}
 \begin{proof}
   Step 1. The sequences of approximate solutions. From the assumptions, we know that the hypothesis of Theorems \ref{thm:boundedness} and \ref{thm:uniform boundedness} is satisfied. These two theorems allow us to produce a sequence of $\{(u^m,p^m,\theta^m,\eta^m)\}_{m=1}^\infty$, which achieve the initial data, satisfy the systems \eqref{equ:iteration equation}, and obey the uniform bounds
   \beq\label{est:uniform bound}
   \sup_{m\ge1}\left(\mathfrak{K}(u^m,p^m,\theta^m)+\mathfrak{K}(\eta^m)\right)\le CP(\mathscr{E}_0).
   \eeq
   The uniform bounds allow us to take weak and weak-$\ast$ limits, up to the extraction of a subsequence:
   \begin{align*}
     &\pa_t^ju^m\rightharpoonup \pa_t^ju\quad \text{weakly in}\thinspace L^2([0,T];H^{2N-2j+1}(\Om))\thinspace \text{for}\thinspace j=0,\ldots,N,\\
     &\pa_t^{N+1}u^m\rightharpoonup\pa_t^{N+1}u\quad \text{weakly in}\thinspace (\mathscr{X}_T)^\ast,\\
     &\pa_t^ju^m\stackrel{\ast}\rightharpoonup\pa_t^ju\quad \text{weakly}-\ast\thinspace\text{in}\thinspace L^\infty([0,T];H^{2N-2j}(\Om))\thinspace \text{for}\thinspace j=0,\ldots,N,\\
     &\pa_t^jp^m\rightharpoonup\pa_t^jp\quad \text{weakly in}\thinspace L^2([0,T];H^{2N-2j}(\Om))\thinspace \text{for}\thinspace j=0,\ldots,N,\\
     &\pa_t^jp^m\stackrel{\ast}\rightharpoonup \pa_t^jp\quad \text{weakly}-\ast\thinspace\text{in}\thinspace L^\infty([0,T];H^{2N-2j-1}(\Om))\thinspace \text{for}\thinspace j=0,\ldots,N,\\
     &\pa_t^j\theta^m\rightharpoonup \pa_t^j\theta\quad \text{weakly in}\thinspace L^2([0,T];H^{2N-2j+1}(\Om))\thinspace \text{for}\thinspace j=0,\ldots,N,\\
     &\pa_t^{N+1}\theta^m\rightharpoonup\pa_t^{N+1}\theta\quad \text{weakly in}\thinspace (\mathscr{H}^1_T)^\ast,\\
     &\pa_t^j\theta^m\stackrel{\ast}\rightharpoonup\pa_t^j\theta\quad \text{weakly}-\ast\thinspace\text{in}\thinspace L^\infty([0,T];H^{2N-2j}(\Om))\thinspace \text{for}\thinspace j=0,\ldots,N,
   \end{align*}
   and
   \begin{align*}
     &\pa_t^j\eta^m\rightharpoonup\pa_t^j\eta \quad \text{weakly in}\thinspace L^2([0,T];H^{2N-2j+5/2}(\Sigma))\thinspace\text{for}\thinspace j=2,\ldots,N+1,\\
     &\eta^m\stackrel{\ast}\rightharpoonup\eta\quad \text{weakly}-\ast\thinspace\text{in}\thinspace L^\infty([0,T];H^{2N+1/2}(\Sigma)),\\
     &\pa_t^j\eta^m\stackrel{\ast}\rightharpoonup\pa_t^j\eta\quad \text{weakly}-\ast\thinspace\text{in}\thinspace L^\infty([0,T];H^{2N-2j+3/2}(\Sigma))\thinspace \text{for}\thinspace j=1,\ldots,N.
   \end{align*}
   The collection $(v,q,\Theta,\zeta)$ achieving the initial data, that is, $\pa_t^jv(0)=\pa_t^ju(0)$, $\pa_t^j\Theta(0)=\pa_t^j\theta(0)$, $\pa_t^j\zeta(0)=\pa_t^j\eta(0)$ for $j=0,\ldots,N$ and $\pa_t^jq(0)=\pa_t^jp(0)$ for $j=0,\ldots,N-1$, is closed in the above weak topology by Lemma A.4 in \cite{GT1}. Hence the limit $(u,p,\theta,\eta)$ achieves the initial data, since each $(u^m,p^m,\theta^m,\eta^m)$ is in the above collection.

   Step 2. Contraction. For $m\ge1$, we set $v^1=u^{m+2}$, $v^2=u^{m+1}$, $w^1=u^{m+1}$, $w^2=u^m$, $q^1=p^{m+2}$, $q^2=p^{m+1}$, $\Theta^1$=$\theta^{m+2}$, $\Theta^2=\theta^{m+1}$, $\vartheta^1=\theta^{m+1}$, $\vartheta^2=\theta^m$, $\zeta^1=\eta^{m+1}$, $\zeta^2=\eta^m$. Then from the construction of initial data, the initial data of $v^j$, $w^j$, $q^j$, $\Theta^j$, $\vartheta^j$, $\zeta^j$ math the hypothesis of Theorem \ref{thm:contraction}. Because of \eqref{equ:iteration equation}, \eqref{equ:difference} holds. In addition, \eqref{est:uniform bound} holds. Thus, all hypothesis of Theorem \ref{thm:contraction} are satisfied. Then
   \ben\label{est:bound N um pm thetam}
   \begin{aligned}
   &\mathfrak{N}(u^{m+2}-u^{m+1},p^{m+2}-p^{m+1},\theta^{m+2}-\theta^{m+1};T)\\
   &\le\f12\mathfrak{N}(u^{m+1}-u^m,p^{m+1}-p^m,\theta^{m+1}-\theta^m;T),
   \end{aligned}
   \een
   \beq\label{est:bound M etam}
   \mathfrak{M}(\eta^{m+1}-\eta^m;T)\lesssim \mathfrak{N}(u^{m+1}-u^m,p^{m+1}-p^m,\theta^{m+1}-\theta^m;T).
   \eeq
   The bound \eqref{est:bound N um pm thetam} implies that the sequence $\{(u^m,p^m,\theta^m)\}_{m=0}^\infty$ is Cauchy in the norm $\sqrt{\mathfrak{N}(\cdot,\cdot,\cdot;T)}$. Thus
   \ben
   \left\{
   \begin{aligned}
     &u^m\to u\quad &\text{in}\thinspace L^\infty\left([0,T];H^2(\Om)\right)\cap L^2\left([0,T];H^3(\Om)\right),\\
     &\pa_tu^m\to\pa_tu\quad &\text{in}\thinspace L^\infty\left([0,T];H^0(\Om)\right)\cap L^2\left([0,T];H^1(\Om)\right),\\
     &p^m\to p\quad &\text{in}\thinspace L^\infty\left([0,T];H^1(\Om)\right)\cap L^2\left([0,T];H^2(\Om)\right),\\
     &\theta^m\to \theta\quad &\text{in}\thinspace L^\infty\left([0,T];H^2(\Om)\right)\cap L^2\left([0,T];H^3(\Om)\right),\\
     &\pa_t\theta^m\to\pa_t\theta\quad &\text{in}\thinspace L^\infty\left([0,T];H^0(\Om)\right)\cap L^2\left([0,T];H^1(\Om)\right),
   \end{aligned}
   \right.
   \een
   as $m\to\infty$.
   Because of \eqref{est:bound M etam}, we deduce that the sequence $\{\eta^m\}_{m=1}^\infty$ is Cauchy in the norm $\sqrt{\mathfrak{M}(\cdot;T)}$. Thus,
   \ben
   \left\{
   \begin{aligned}
   &\eta^m\to\eta\quad &\text{in}\thinspace L^\infty\left([0,T];H^{5/2}(\Sigma)\right),\\
   &\pa_t\eta^m\to\pa_t\eta\quad &\text{in}\thinspace L^\infty\left([0,T];H^{3/2}(\Sigma)\right),\\
   &\pa_t^2\eta^m\to\pa_t^2\eta\quad &\text{in}\thinspace L^2\left([0,T];H^{1/2}(\Sigma)\right),
   \end{aligned}
   \right.
   \een
   as $m\to\infty$.

   Step 3. Interpolation and passing to the limit. This section is exactly the same as the proof of Theorem 6.3 in\cite{GT1}, which gives the existence of solutions and the estimate \eqref{est:bound K}.

   Step 4. Uniqueness and diffemorphism.  This section is similar to the proof of Theorem 6.3 in\cite{GT1}.
 \end{proof}

\section{acknowledgement}

The author would like to thank professor Yan Guo for several helpful discussions.

\end{document}